\documentclass[11pt,a4paper,reqno]{amsart}
\usepackage[margin=1.37in]{geometry}
\usepackage{amsmath, amsxtra, amsthm, amsfonts, amssymb, mathtools}
\usepackage{mathrsfs}
\usepackage{graphicx}
\usepackage{url}
\usepackage[dvipsnames]{xcolor}
\usepackage{bbm}
\usepackage{bm}
\usepackage{multicol}

\usepackage{tikz-cd}
\usepackage{tikz}
\tikzstyle{rectan} = [rectangle, draw, rounded corners, minimum width=2.5cm, minimum height=0.8cm, text centered]
\usetikzlibrary{backgrounds}
\usepackage[cal=boondoxo,scr=euler]{mathalfa}
\usepackage{float}
\usetikzlibrary{decorations.pathreplacing, positioning}
\usetikzlibrary{decorations.markings}
\usetikzlibrary{shapes,positioning,intersections,quotes}
\usetikzlibrary{shapes,fit}
\usetikzlibrary{arrows.meta, calc}
\definecolor{curvyarrow}{RGB}{200, 100, 150}
\usepackage{comment}
\usepackage[shortlabels]{enumitem} 
\usepackage{enumitem}
\usepackage{relsize}
\usepackage{stmaryrd}

\usepackage[all,cmtip]{xy}	
\xyoption{arrow}
\usepackage[T1]{fontenc}
\usepackage{enumitem}
\usepackage[backref=page,linktocpage]{hyperref} 
\hypersetup{
    colorlinks,
    allcolors=teal 
}

\usepackage[capitalize]{cleveref}

\newtheoremstyle{results}
{8pt}
{8pt}
{\itshape}
{}
{\bfseries}
{}
{.5em}
{}
\theoremstyle{results}
\newtheorem{thm}{Theorem}[section]

\newtheorem{lemma}[thm]{Lemma}
\newtheorem{corollary}[thm]{Corollary}
\newtheorem{proposition}[thm]{Proposition}

\newtheoremstyle{definitions}
{7pt}
{7pt}
{}
{}
{\bfseries}
{}
{.5em}
{}

\theoremstyle{definitions}
\newtheorem{definition}[thm]{Definition}
\newtheorem{conjecture}[thm]{Conjecture}

\newtheorem{example}[thm]{Example}

\newtheorem{remark}[thm]{Remark}

\newtheorem{notation}[thm]{Notation}

\newtheorem{question}[thm]{Question}

\crefname{thm}{theorem}{theorems}
\Crefname{thm}{Theorem}{Theorems}
\crefname{lemma}{lemma}{lemmas}
\Crefname{lemma}{Lemma}{Lemmas}
\crefname{proposition}{proposition}{propositions}
\Crefname{proposition}{Proposition}{Propositions}
\crefalias{proposition}{proposition}
\crefname{corollary}{corollary}{corollaries}
\Crefname{corollary}{Corollary}{Corollaries}
\crefname{definition}{definition}{definitions}
\Crefname{definition}{Definition}{Definitions}
\crefname{conjecture}{conjecture}{conjectures}
\Crefname{conjecture}{Conjecture}{Conjectures}
\crefname{example}{example}{examples}
\Crefname{example}{Example}{Examples}
\crefname{remark}{remark}{remarks}
\Crefname{remark}{Remark}{Remarks}

\def\Z{\mathbb{Z}}
\def\Q{\mathbb{Q}}
\def\R{\mathbb{R}}
\def\C{\mathbb{C}}

\def\P{\mathbb{P}}

\def\N{\mathbb{N}}
\def\Z{\mathbb{Z}}

\def\calA{\mathcal{A}}

\def\calG{\mathcal{G}}

\def\calL{\mathcal{L}}
\def\calM{\mathcal{M}}
\def\calN{\mathcal{N}}

\def\calP{\mathcal{P}}

\def\calS{\mathcal{S}}
\newcommand{\T}{\mathcal{T}}

\def\N{\mathbf{N}}

\newcommand{\abs}[1]{\left\lvert#1\right\rvert}

\newcommand{\ch}{C}
\newcommand{\cl}{\operatorname{cl}}
\newcommand{\lk}{\operatorname{lk}}

\newcommand{\Cone}{\operatorname{Cone}}

\newcommand{\e}{\mathbf{e}}

\newcommand{\Hom}{\operatorname{Hom}}

\renewcommand{\hat}{\widehat}

\newcommand{\aug}{\operatorname{aug}}

\newcommand{\uCH}{\operatorname{CH}}
\newcommand{\uH}{\textrm{H}}

\newcommand{\rk}{\operatorname{rk}}

\newcommand{\M}{\mathrm{M}}
\renewcommand{\emptyset}{\varnothing}

\newcommand{\floor}[1]{\lfloor #1 \rfloor}

\renewcommand{\L}{\mathcal{L}}

\newcommand{\un}{\hat{1}}
\newcommand{\zero}{\hat{0}}

\newcommand{\G}{\calG}

\renewcommand{\S}{\calS}

\newcommand{\CH}{\uCH}
\newcommand{\fact}{\mathrm{fact}}

\newcommand{\Mcut}{\mathcal{M}}
\newcommand{\At}{\mathrm{At}}

\newcommand{\Flag}{\mathbf{FL}}
\newcommand{\Gpos}{\mathbf{\Gamma P}}
\newcommand{\Gposs}{\Gpos^{\ast}}

\newcommand{\des}{\mathrm{des}}

\newcommand{\PL}{\textrm{PL}}

\newcommand{\Tr}{\mathrm{Tr}}
\newcommand{\MW}{\mathrm{MW}}

\newcommand{\B}{\mathrm{B}}
\renewcommand{\H}{\mathrm{H}}

\newcommand{\U}{\mathrm{U}}
\newcommand{\Nest}{\mathcal{N}}
\renewcommand{\N}{\mathcal{N}}
\newcommand{\stable}{\mathrm{stable}}
\newcommand{\Des}{\mathrm{Des}}
\newcommand{\FY}{\mathrm{FY}}
\newcommand{\supp}{\mathrm{supp}}
\newcommand{\Moduli}{\overline{\mathcal{M}}}

\newcommand{\Braid}{\mathcal{B}raid}
\newcommand{\Leaf}{\operatorname{Leaf}}

\renewcommand{\subset}{\subseteq}
\renewcommand{\supset}{\supseteq}

\newcommand{\bigast}{\mathop{\scalebox{2.1}{$\ast$}}\displaylimits}
\newcommand{\mult}{\textrm{mult}}
\newcommand{\Nlatt}{\mathbf{N}}
\newcommand{\Mlatt}{\mathbf{M}}

\title{Matroid analogues of Gal's conjecture}

\author{Basile Coron}
\address{B. Coron, Centre de Mathématiques Laurent Schwartz, Polytechnique, Paris, France}
\email{basile.coron@polytechnique.edu}

\author{Luis Ferroni}
\address{L. Ferroni, Dipartimento di Matematica, Universit\`a di Pisa, Pisa, Italy}\email{luis.ferroni@unipi.it}

\author{Shiyue Li}
\address{S. Li, Department of Mathematics, University of Michigan, Ann Arbor, USA}\email{shiyueli@umich.edu}
\date{\today}

\begin{document}
\begin{abstract}
   Well-known conjectures of Charney--Davis, Gal, and Nevo--Petersen predict increasingly strong positivity phenomena for the $h$-vectors of flag simplicial spheres. In this paper, we formulate and prove matroid analogues of these conjectures in the setting of Chow polynomials of matroids with building sets.  We introduce a new class of matroids with building sets, called complete built matroids, encompassing many prominent families of built matroids such as arbitrary matroids with maximal building sets and braid matroids with minimal building sets. For complete built matroids, we prove $\gamma$-positivity as an analogue of Gal's conjecture, via a combinatorial formula for the $\gamma$-coefficients. We further realize the $\gamma$-vector as the $f$-vector of a simplicial complex, as an analogue of the Nevo--Petersen conjecture. As an application, we obtain a new formula for the $\gamma$-polynomial of the Poincar\'e polynomial of $\overline{\mathcal{M}}_{0,n}$, together with new coefficient inequalities. We also study flag built matroids, and prove $\gamma$-positivity of their Chow polynomials, extending several known results. Our proofs crucially use toric geometry and tropical intersection theory. Finally, we construct an infinite family of flag chordal nestohedra whose $h$-polynomials are not real-rooted, invalidating a natural strengthening of our result at this level of generality.
\end{abstract}

\maketitle

\setcounter{tocdepth}{1}
\tableofcontents

\section{Introduction}
Throughout, $\M$ is a loopless matroid. 

\subsection{Feichtner--Yuzvinsky Chow rings and their Hilbert--Poincar\'e series}
In a landmark paper, Feichtner and Yuzvinsky~\cite{feichtner2004chow} introduced the Chow ring $\uCH(\M, \G)$ associated to $\M$ together with a building set $\mathcal{G}$ of the lattice of flats of $\M$. This ring generalizes the Chow ring of the wonderful compactification $W_{\M, \G}$ associated to $\G$ introduced by De Concini and Procesi~\cite{de1995wonderful} in the realizable case. In general, Feichtner and Yuzvinsky showed that $\uCH(\M, \G)$ can be identified with the Chow ring of a smooth but usually non-complete toric variety. These rings played a fundamental role in the recent development of combinatorial Hodge theory. Adiprasito, Huh, and Katz \cite{adiprasito-huh-katz} show that, when $\mathcal{G}$ is \emph{maximal}, the Chow ring $\uCH(\M, \G)$ is a graded Artinian algebra, enjoying the \emph{K\"ahler package}: 
\begin{enumerate}[leftmargin=28pt]
    \item[(PD)] Poincar\'e duality; 
    \item[(HL)] the hard Lefschetz property; and 
    \item[(HR)] the mixed Hodge--Riemann relations in arbitrary degree.
\end{enumerate} The work of Ardila, Denham, and Huh strengthened this result by developing the tropical Hodge theory \cite{ardila-denham-huh} (see also \cite{pagaria2023hodge,amini2023hodgetheorytropicalfans}), which implies that for any loopless matroid $\M$ and any building set $\mathcal{G}$, the algebra $\uCH(\M, \G)$ satisfies the K\"ahler package.
This package implies positivity properties of the Hilbert--Poincar\'e series $\H(\M,\mathcal{G})(t)$ of $\uCH(\M, \G)$, commonly referred to as the \emph{Chow polynomial} of $(\M,\mathcal{G})$: 
\begin{align*}
    &\text{Poincar\'{e} duality implies the palindromicity of }\H(\M,\mathcal{G})(t). \\
    &\text{The hard Lefschetz property implies the unimodality of }\H(\M,\mathcal{G})(t). 
\end{align*}
These Hilbert--Poincar\'e series exhibit various stronger positivity properties. The following conjecture, posed independently by Huh–Stevens \cite{stevens-bachelor} and Ferroni–Schr\"{o}ter~\cite{ferroni-schroter}, predicts that they satisfy the strongest such property: real-rootedness.

\begin{conjecture}\label{conj:chow-rr-maximal}
    If $\G$ is maximal, then $\H(\M,\G)(t)$ is real-rooted.
\end{conjecture}
Below real-rootedness, we have a hierarchy of positivity properties for palindromic polynomials. A palindromic polynomial $h(t) \in \R[t]$ of degree $d$ is \emph{$\gamma$-positive} if 
\[ 
h(t) = \sum_{i=0}^{\lfloor d/2\rfloor} \gamma_i\, t^i(1+t)^{d-2i}, \text{ with } \gamma_i \geqslant 0 \text{ for every } 0 \leqslant i \leqslant \floor{d/2}. 
\]
If the coefficients of $h(t)$ are positive and palindromic, then these implications\footnote{A sequence of real numbers $a_0,\ldots,a_n$ is said to be \emph{unimodal}, if there exists some index $j$ such that 
    \[
    a_0 \leqslant \cdots \leqslant a_{j-1} \leqslant a_j \geqslant a_{j+1} \geqslant \cdots \geqslant a_n.
    \] A sequence is \emph{log-concave} if $a_i^2 \geqslant a_{i-1} a_{i+1}$ for every $0 < i < n$. Given a sequence of nonnegative real numbers with no internal zeros (for all $i < j <k$, if $a_i \ne 0, a_k \ne 0$, then $a_j \ne 0$), if it is log-concave, then it is unimodal. A polynomial is unimodal (resp. log-concave), if its sequence of coefficients is unimodal (resp. log-concave). A classical theorem by Newton states that if a polynomial is real-rooted, then its coefficients are log-concave \cite[Theorem~2]{stanley-unimodality}.} hold. 
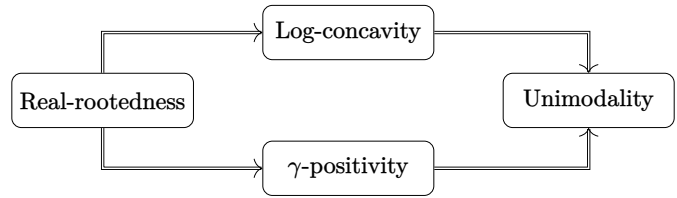
\begin{figure}[ht]
    \centering
    \scalebox{0.9}{
    \begin{tikzpicture}
        \tikzset{node distance = 2.0cm and 1cm}
        \tikzstyle{arrow} = [-{>[scale=1.8, length=2, width=4.5]}, double]
        \node (real-rooted)     [rectan]                                    {Real-rootedness};
        \node (log-concave)     [rectan, right=of real-rooted, yshift= 1.0cm] {Log-concavity};
        \node (gamma-positivity)[rectan, right=of real-rooted, yshift=-1.0cm] {$\gamma$-positivity};
        \node (unimodal)        [rectan, right=of log-concave, yshift=-1.0cm] {Unimodality};
        \draw[arrow] (real-rooted)      |- (log-concave);
        \draw[arrow] (real-rooted)      |- (gamma-positivity);
        \draw[arrow] (gamma-positivity) -| (unimodal);
        \draw[arrow] (log-concave)      -| (unimodal);
    \end{tikzpicture}}
    \caption{Hierarchy of positivity properties for positive, palindromic polynomials.}
    \label{fig:hierarchy}
\end{figure}

Except for unimodality, the positivity properties appearing in Figure~\ref{fig:hierarchy} are highly sensitive to the choice of building sets. For the choice of maximal building set $\G_{\max}$, Conjecture~\ref{conj:chow-rr-maximal} has been proven in various special families of matroids \cite{ferroni-matherne-stevens-vecchi,hoster-stump,branden-vecchi2,coron-ferroni-li}, and $\gamma$-positivity holds for all matroids \cite{ferroni-matherne-stevens-vecchi,stump,ferroni-matherne-vecchi}. However, their behavior becomes mixed once we leave the realm of maximal building sets. For example, the uniform matroid $\U_{n-1,n}$ with the minimal building set $\mathcal{G}_{\min}$ has associated wonderful variety $W_{\U_{n-1, n}, \G_{\min}}$ isomorphic to the projective space $\P^{n-2}$. Therefore, the Hilbert--Poincar\'{e} series is  
\[
\H(\U_{n-1, n},\mathcal{G}_{\min})(t) = 1 + t + \cdots + t^{n-2}, 
\] which fails to be 
real-rooted or $\gamma$-positive for $n \geqslant 4$. Log-concavity can fail too \cite[Section~6]{eur-ferroni-matherne-pagaria-vecchi}.  

Despite the examples above, these four positivity properties appear to hold in a surprisingly broad range of cases, far beyond the realm of maximal building sets. A particularly important example arises when $\M$ is the braid matroid $\Braid_n$ and $\G=\G_{\min}$ is the minimal building set. In this case, extensive computer experiments support the following conjecture of Aluffi--Chen--Marcolli \cite[Conjecture~1]{aluffi-chen-marcolli}.
\begin{conjecture}\label{conj:chow-rr-braid-min}
    The Poincaré polynomial $P_{\overline{\mathcal{M}}_{0, n+1}}(t) = \H(\Braid_n, \G_{\min})(t)$ of the moduli space $\overline{\mathcal{M}}_{0, n+1}$ is real-rooted.
\end{conjecture}
This suggests that one should look for conceptual explanations for the positivity properties of $\H(\M,\G)$. Our goal is to identify structural conditions on the main objects arising from $(\M,\G)$---including the nested set complex $\N(\M,\G)$, the wonderful variety $W_{\M,\G}$, and the toric variety $X_{\M,\G}$---that guarantee such properties. In the present article, we focus in particular on understanding the relationship between these structural features and $\gamma$-positivity.

\subsection{Gamma-positivity for flag simplicial spheres} 
A long-standing open conjecture in geometric and topological combinatorics posed by Charney and Davis \cite[Conjecture~A]{charney-davis} predicts the sign of the Euler characteristic of any even dimensional, nonpositively curved compact piecewise Euclidean manifold. The origins of this conjecture trace back to open conjectures by Hopf from the 1930s, which predict that the signed Euler characteristics of compact Riemannian manifolds of even dimension with nonpositive sectional curvature are always positive. 

The Charney--Davis conjecture, in the special case of \emph{cubical} piecewise Euclidean manifolds, can be recast in purely combinatorial terms. Recall that a simplicial complex $\Delta$ is said to be \emph{flag}, if every minimal nonface of $\Delta$ has cardinality $2$. Gromov \cite[p.~122]{gromov} observed that the cubicality of a nonpositively curved piecewise Euclidean manifold is equivalent to flagness of all vertex links. Babson--Billera--Chan \cite{babson-billera-chan} proved that for every simplicial sphere $\Delta$, there exists a cubical piecewise Euclidean manifold all of whose links are isomorphic to $\Delta$. If we denote the number of faces of size $i$ of a simplicial complex $\Delta$ by $f_i(\Delta)$, we have the following reformulation of the Charney--Davis conjecture (cf. \cite[p.~186]{forman}). 

\begin{conjecture}[Charney--Davis conjecture]\label{conj:charney-davis}
    If $\Delta$ is a flag simplicial sphere of dimension $2n-1$, then the \emph{Charney--Davis quantity},
    \[
    \widetilde{\gamma}_{n}(\Delta) :=  \sum_{i=0}^{2n} f_i(\Delta) (-2)^{n-i},
    \]
    is nonnegative.
\end{conjecture}

If $\Delta$ is a simplicial $(d-1)$-sphere, the polynomial $h(\Delta;x) := \sum_{i=0}^d f_i(\Delta)x^i(1-x)^{d-i}$ is palindromic by the Dehn--Sommerville equations, which allows the expansion
\begin{equation*} \label{eq:h-gamma}
    h(\Delta; x) = \sum_{i=0}^{\lfloor d/2\rfloor} \gamma_i(\Delta)\, x^i(1+x)^{d-2i}, 
\end{equation*}
where each $\gamma_i(\Delta)$ is an integer number. 
When $\Delta$ is odd-dimensional, the specialization $h(\Delta;x)$ at $x = -1$ yields $\gamma_n(\Delta) = 2^{n} \,\widetilde{\gamma}_n(\Delta)$. This implies that the nonnegativity of the Charney--Davis quantity is determined by the nonnegativity of the top entry of the $\gamma$-polynomial of $h(\Delta;x)$.
Gal proposes a significant strengthening of the Charney--Davis conjecture \cite[Conjecture~2.1.7]{gal2005real}, while removing the parity assumption on dimensionality. 

\begin{conjecture}[Gal's conjecture]\label{conj:gal}
    If $\Delta$ is a flag simplicial sphere, then the polynomial $h(\Delta;x)$ is $\gamma$-positive.
\end{conjecture}

Nevo and Petersen further strengthen Gal's conjecture in \cite[Conjecture~1.4]{nevo-petersen} by proposing that the sequence $\gamma(\Delta) \coloneqq (\gamma_0(\Delta),\ldots,\gamma_{\lfloor d/2\rfloor}(\Delta))$ is the $f$-vector of a simplicial complex, which is well-known to be characterized by the Kruskal--Katona inequalities.

\begin{conjecture}[Nevo--Petersen conjecture]\label{conj:nevo-petersen}
    If $\Delta$ is a flag simplicial sphere, then there exists a simplicial complex $\Gamma$ such that $\gamma(\Delta) = f(\Gamma)$. 
\end{conjecture}

Furthermore, Nevo and Petersen conjecture that the simplicial complex $\Gamma$ appearing in the above conjecture can in fact be chosen to be \emph{balanced} \cite[Conjecture~6.3]{nevo-petersen}. Thus, for flag simplicial spheres, the following hierarchy of implications holds. 

\begin{figure}[ht]
    \centering
    \scalebox{0.9}{
    \begin{tikzpicture}
        \tikzset{node distance = 2.0cm and 1cm}
        \tikzstyle{arrow} = [-{>[scale=1.8, length=2, width=3.5]}, double]
        \node (np-bal)  [rectan]                        {Nevo--Petersen (balanced $\Gamma$)};
        \node (np)      [rectan, right=of np-bal]       {Nevo--Petersen};
        \node (gal)     [rectan, right=of np]           {Gal};
        \node (cd)      [rectan, right=of gal]          {Charney--Davis};
        \draw[arrow] (np-bal) -- (np);
        \draw[arrow] (np)     -- (gal);
        \draw[arrow] (gal)    -- (cd);
    \end{tikzpicture}}
    \caption{Strengths of conjectures on $\gamma(\Delta)$ for a flag simplicial sphere $\Delta$.}
    \label{fig:implications}
\end{figure}
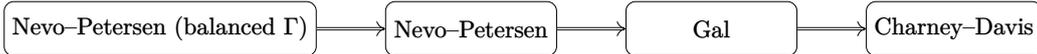

\subsection{Main results}
We formulate and prove analogues of Conjecture~\ref{conj:gal} and Conjecture~\ref{conj:nevo-petersen} in the setting of Chow rings of matroids along with building sets. A pair $(\M,\G)$ consisting of a matroid and a building set on its lattice of flats will be referred to as a \emph{built matroid}.

Our first main result concerns the class of \emph{complete built matroids}, which we introduce and analyze in Section ~\ref{subsec:complete}. This class encompasses several important families of built matroids such as matroids with maximal building sets, free coextensions with the augmented building set, chordal nestohedra, and supersolvable built lattices (in particular, braid matroids with the minimal building set).

We establish $\gamma$-positivity of the Chow polynomial for complete built matroids, and provide a combinatorial formula for the $\gamma$-expansion of these polynomials which gives a generalization of a recent result by Stump~\cite{stump} in the maximal and augmented building set case.
\begin{thm}[Theorem~\ref{thm:complete-formula}]\label{thm:main-complete-des-formula}
    If $(\M,\G)$ is a complete built matroid, and $\gamma(\M, \G)(t)$ is the $\gamma$-expansion of $\H(\M, \G)(t)$, then we have the formula 
    \[ \gamma(\M, \G)(t) = \sum_{\S \in \Nest(\M,\G)^{\max}_{\stable}} t^{\des(\S)}.\]
    In particular, the polynomial $\H(\M, \G)$ is $\gamma$-positive. 
\end{thm}

We also obtain the following combinatorial extension of Theorem~\ref{thm:main-complete-des-formula}, which establishes an analogue of the Nevo--Petersen conjecture for complete built matroids (and the stronger version for maximal building sets).

\begin{thm}[Theorems~\ref{thm:gamma-is-simplicial-complex} and \ref{thm:gamma-maximal-is-balanced}]\label{thm:main-complete-implies-nevo-petersen+balanced}
    If $(\M, \G)$ is complete, then there exists a simplicial complex $\Gamma$ such that 
    \[
    \gamma(\M, \G) = f(\Gamma).
    \]
    Furthermore, if $\G$ is maximal, then $\Gamma$ is balanced. 
\end{thm}

\medskip

The combinatorics and intersection-theoretic properties of the boundary strata of the wonderful compactification $W_{\M, \G}$ when $\M$ is realizable, as well as those of the torus orbit closures of $X_{\M, \G}$ for a general matroid $\M$, are governed by the \emph{nested set complex} $\mathcal{N}(\M,\mathcal{G})$, a simplicial complex introduced by Feichtner--Kozlov~\cite{feichtner2004incidence}. Requiring this complex to be flag is another reasonable structural hypothesis to furnish an analogue of Gal's conjecture. We say that a pair $(\M,\mathcal{G})$ is a \emph{flag built matroid}, if the nested set complex $\N(\M, \G)$ is flag.  Our second result establishes the analogue of Gal's conjecture for flag built matroids.

\begin{thm}[Theorem~\ref{thm:flag-gamma-pos}]\label{thm:main-flag-implies-gamma}
    If $(\M, \G)$ is a flag built matroid, then $\H(\M,\mathcal{G})$ is $\gamma$-positive.
\end{thm}

Flag built matroids include several previously studied built matroids as special cases: 
\begin{enumerate}[(i), itemsep=2pt, leftmargin = 20pt]
    \item All matroids with maximal building sets; Theorem~\ref{thm:main-flag-implies-gamma} recovers the first part of \cite[Theorem~3.25]{ferroni-matherne-stevens-vecchi}. 
    \item All supersolvable built matroids, introduced in \cite{Coron_2025}. 
    \item All flag nestohedra, studied in \cite{volodin, aisbett}. (In particular, as we discuss below, this also generalizes the $\gamma$-positivity of the simple generalized permutohedra known as \emph{chordal nestohedra}, considered in \cite{Postnikov2008})
\end{enumerate} 

\smallskip

We now describe the relationship among the classes of flag built matroids, complete built matroids, and other well-studied built matroids (see Figure~\ref{fig:euler-diagram}). Flag and complete built matroids are different generalizations of several built matroids which are prominently featured in the literature. In a recent paper, Coron \cite{Coron_2025} introduced a class of built matroids called \emph{supersolvable built matroids}---which, in a nutshell, consists of loopless matroids whose lattice of flats is supersolvable, along with a building set that is compatible with the supersolvability of the lattice. Sitting in the intersection of flag and complete built matroids, the class of supersolvable built matroids includes the following: 
\begin{itemize}[itemsep=2pt, leftmargin=10pt]
    \item All type A braid matroids with minimal building sets $(\Braid_n, \G_{\min})$. Theorem~\ref{thm:main-flag-implies-gamma} recovers the result by Aluffi--Chen--Marcolli in \cite[Theorem 1.2]{aluffi-chen-marcolli}. 
    \item All type B braid matroids with minimal building sets, considered in Pagaria--Ferroni--Vecchi \cite{ferroni-pagaria-vecchi}.  
    \item All chordal Boolean matroids, or chordal nestohedra, studied by Postnikov--Reiner--Williams \cite{Postnikov2008}. In fact, for a Boolean matroid, the notion of complete building set corresponds exactly to chordal nestohedra. In other words, the set of chordal nestohedra is the intersection of flag nestohedra and complete built matroids. 
\end{itemize}
Despite their overlaps, significant examples of non-flag complete built matroids are given by the augmented built matroids, whose Chow rings are the \emph{augmented Chow rings of matroids} (Example~\ref{ex:complete-ex-aug}). 

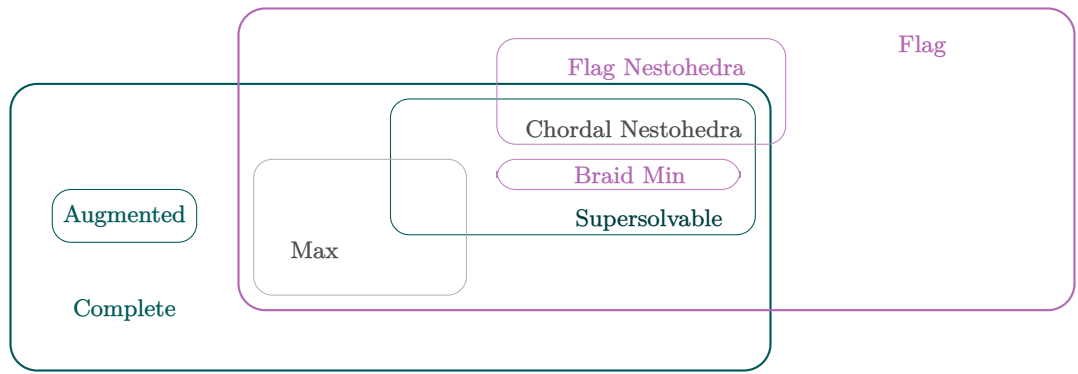
\begin{figure}[ht]
    \centering
    \begin{tikzpicture}
        \tikzset{
            outer/.style = {rounded corners=10pt, thick},
            inner/.style = {rounded corners=7pt}
        }

        \draw[outer, color=teal!70!black]
            (-5.0, -1.8) rectangle (5.0, 2);
        \draw[outer, color=violet!60]
            (-2.0, -1) rectangle (9.0, 3);
        \draw[inner, color=teal!70!black]
            (0, 0) rectangle (4.8, 1.8);
        \draw[inner, color=gray!60]
            (-1.8, -0.8) rectangle (1, 1);
        \draw[inner, color=violet!50]
            (1.4, 0.6) rectangle (4.6, 1); 
        \draw[inner, color=violet!50]
            (1.4, 1.2) rectangle (5.2, 2.6); 
        \draw[inner, color=teal!70!black]
            (-4.45, -0.10) rectangle (-2.55, 0.60);

        \node[font=\small, color=teal!70!black] at (-3.5, -1)  {Complete};
        \node[font=\small, color=violet!60]     at  (7, 2.5)  {Flag};
        \node[font=\small, color=teal!70!black] at (-3.5, 0.25) {Augmented};
        \node[font=\small,    color=teal!50!black] at  (3.4, 0.2)  {Supersolvable};
        \node[font=\small,    color=gray!60!black] at  (-1, -0.2)  {Max};
        \node[font=\small,    color=violet!60!] at  (3.15, 0.8)  {Braid Min};
        \node[font=\small,    color=gray!60!black] at  (3.2, 1.4)  {Chordal Nestohedra};
        \node[font=\small,    color=violet!60!] at  (3.5, 2.2)  {Flag Nestohedra};
    \end{tikzpicture}
    \caption{Relationships between relevant built matroids.}
    \label{fig:euler-diagram}
\end{figure}

Our proof of Theorem \ref{thm:main-flag-implies-gamma} proceeds by induction on the Chow ring of the toric variety $X_{\M, \G}$ associated to a built matroid, via a composition of three geometric operations which change the topology of $X_{\M, \G}$ in a controlled way that preserves $\gamma$-positivity: iterated $\P^1$-bundles, iterated blowups, and tropical modification. The order of these operations is dictated by the combinatorial structure of the built matroid, described below. 

Among all building sets, a flag built matroid $(\M, \G)$ is distinguished by a key combinatorial property: it can be constructed from an \emph{arbitrary} building subset via a \emph{binary filtration}, which prescribes a sequence of additions of elements terminating at the flag building set. This filtration has a precise geometric counterpart: the toric variety $X_{\M, \G}$ is realized as an iterated sequence of $\P^1$-bundles and iterated blowups along codimension-$2$ torus-invariant subvarieties. In the inductive step, along with iterated $\P^1$-bundles and iterated blowups, we perform a single-element extension of the built matroid, which corresponds geometrically to a tropical modification. Tropical modification, introduced by Mikhalkin in \cite{mikhalkin2007tropical}, is a procedure that alters the topology of the tropical fan while preserving the cohomology of the associated toric variety. Altogether, these operations propagate $\gamma$-positivity from a simpler $\gamma$-positive built Boolean matroid, to a flag built matroid. Our first proof of Theorem ~\ref{thm:complete-gamma-pos}, which is the first part of Theorem~\ref{thm:main-complete-des-formula} for the case of complete built matroids, employs a similar but simpler strategy and analysis by omitting tropical modification. 

\smallskip

As an application of Theorem~\ref{thm:main-complete-des-formula} and Theorem~\ref{thm:main-complete-implies-nevo-petersen+balanced}, we further derive a combinatorial formula for the $\gamma$-expansion of the Poincar\'{e} polynomial of the Deligne--Mumford--Knudsen compactification $\overline{\mathcal{M}}_{0, n}$ of rational stable marked curves.
\begin{corollary}[Corollary~\ref{cor:braid-min}]
    For any $n \geqslant 2$, the $\gamma$-expansion of the Poincar\'{e} polynomial of $\overline{\mathcal{M}}_{0, n+1}$ is 
    \[
    \gamma(P_{\Moduli_{0,n+1}};t) = \sum_{\T \in T_{n,\stable}^{\max}}t^{\des(\T)}.
    \] 
    In addition, the coefficients of the $\gamma$-expansion of $P_{\Moduli_{0,n+1}}(t)$ form the $f$-vector of a simplicial complex. 
\end{corollary}

In this formula, $T_{n, \stable}^{\max}$ denotes the collection of stable maximal trees of the well-studied complex of trees $T_n$ (see \cite{Boardman1971HomotopyTrees}, \cite{trappmann-ziegler}, and \cite{feichtner} for instance), which is bijective to the dual trees of maximally degenerated rational stable marked curves in the boundary of $\overline{\mathcal{M}}_{0, n+1}$. It is worth noting that stability for a tree in the complex of trees is a condition on a certain descent statistic introduced in our previous work \cite{coron2026structuralpropertiesnestedset}, different from the notion of stability for rational marked curves with nodal singularities.

The fact that the $\gamma$-expansion of $\H(\M,\G_{\max})$ corresponds to the $f$-vector of a balanced complex, allows us to deduce the following consequences below (see Section~\ref{subsec:f-vector-inequalities} and Section~\ref{subsec:beyond-rr}). We improve a previous real-rootedness result on matroids of rank at most $5$ \cite[Theorem~5.10]{ferroni-matherne-stevens-vecchi}.

\begin{corollary}[Corollary~\ref{cor:rk6-RR}]
    If $\M$ is a matroid of rank $r \leqslant 6$, then $\H(\M,\G_{\max})(t)$ is real-rooted.
\end{corollary}

Another numerical consequence of Theorem~\ref{thm:main-complete-implies-nevo-petersen+balanced} generalizes recent results by Schweitzer and Vecchi \cite{schweitzer-vecchi} in a different setting.

\begin{corollary}[Corollaries~ \ref{cor:rk6-complete-logconcave} and \ref{cor:g-complete-is-pure-f-vector}]
    If $(\M, \G)$ is a complete built matroid, then the $g$-vector of the coefficients of the Chow polynomial $\H(\M, \G)(t)$ forms a pure $f$-vector. In particular, if $\rk(\M) \leqslant6$, then $\H(\M,\G)(t)$ is log-concave. 
\end{corollary} 

Finally, we include a discussion on real-rootedness phenomena in the setting of flag and complete built matroids. The following result shows that our $\gamma$-positivity results are best possible at this level of generality.

\begin{thm}[Theorem~\ref{thm:chordal-not-RR}]
    There exists an infinite family of flag and chordal nestohedra whose $h$-polynomials are not real-rooted. In other words, there exists an infinite family of flag and complete Boolean matroids whose Chow polynomial is not real-rooted.
\end{thm}

The minimal example contained in the infinite family is presented in Example~\ref{ex:non-RR-chordal-flag-nestohedra} and was discovered using the assistance of ChatGPT 5.5 Pro. Since $h$-polynomials of nested set complexes over Boolean lattices agree with Chow polynomials, our example settles in the negative a conjecture we posed in \cite[Conjecture~6.1]{coron2026structuralpropertiesnestedset}. We regard this construction as one of intrinsic and independent interest, as it provides a nestohedral counterexample to the real-rootedness version of Gal's conjecture (in his paper \cite{gal2005real}, Gal constructed a flag simplicial polytope whose $h$-polynomial is not real-rooted, but doing it among flag nestohedra was open).

\subsection*{Acknowledgments} We thank the organizers and participants of the special year on Algebraic and Geometric Combinatorics at the Institute for Advanced Study, during which this collaboration was initiated. We are indebted to Omid Amini and David Speyer for many insightful suggestions on the topic of tropical fans and tropical modifications. We also thank Jason McCullough for bringing to our attention relevant work on Koszul algebras. BC was supported at IAS by the Ambrose Monell Foundation. LF was supported at the IAS by the Minerva Research Foundation, and is currently a member of the INdAM research group GNSAGA. SL was supported by the NSF Grant DMS-1926686 at the IAS, and is currently supported by the Donald J. Lewis fellowship at the University of Michigan. 

\section{Built matroids and nested set complexes}
We begin by reviewing the main results on simplicial complexes, building sets, and nested set complexes. We extend several classical matroid operations to matroids with building sets. For background on matroids, we refer to \cite{welsh,oxley}.

\subsection{Simplicial complexes}
A \emph{simplicial complex} $\Delta$ on a finite vertex set $V=V(\Delta)$ is a collection of subsets of $V$, called \emph{faces}, which is closed under inclusion: if $F \in \Delta$ and $G \subseteq F$, then $G \in \Delta$. We also require that $\Delta$ contains every singleton in $V$. The \emph{dimension} of a face $F$ is $\dim F = |F| - 1$, and the dimension of $\Delta$ is the maximum dimension among its faces. A simplicial complex is \emph{pure} if all its maximal faces, called \emph{facets}, have the same dimension.  A \emph{subcomplex} $\Delta'$ of a simplicial complex $\Delta$ is a subset of $\Delta$ which is a simplicial complex. For a simplicial complex $\Delta$ of dimension $d$ and a nonempty face $F \in \Delta$, the \emph{deletion} of $F$ from $\Delta$ is \(
\Delta \setminus F \coloneqq \{G \in \Delta \mid F \not\subseteq G\}. 
\)
The \emph{restriction} of $\Delta$ to a subset $A\subseteq V$ is the simplicial complex $\Delta|_A = \{F\in \Delta: F\subseteq A\}$.
The \emph{link} of a nonempty face $F$ in $\Delta$ is the subcomplex $\lk_{\Delta}(F)$ defined by 
\[
\lk_{\Delta}(F) \coloneq \{G \in \Delta \mid F \cap G = \varnothing, F \cup G \in \Delta\}. 
\] 
The \emph{star} of a nonempty face is defined by
\[
\operatorname{star}_{\Delta}(F) \coloneq \{G \in \Delta \mid F \cup G \in \Delta\}. 
\] 
For simplicial complexes $\Delta$ and $\Delta'$ on disjoint vertex sets $V$ and $V'$, the \emph{join} of $\Delta$ and $\Delta'$ is the simplicial complex on the vertex set $V \sqcup V'$ defined by
\[
\Delta \ast \Delta' \coloneqq \{F \cup F' \mid F \in \Delta,\, F' \in \Delta'\}. 
\]
The \emph{cone} over a simplicial complex $\Delta$, denoted $\Cone(\Delta)$, is the join of $\Delta$ with a point. For more background on simplicial complexes, we refer to \cite{stanley-96}.

\subsection{Built matroids}
Geometric lattices are the lattices of flats of simple matroids, and building sets are distinguished subsets of geometric lattices. When the matroid $\M$ is realizable, a building set prescribes a collection of blowup loci in the construction of the wonderful variety $W_{\M, \G}$, whose boundary divisor has simple normal crossings. Throughout this section, we let $\L = \L(\M)$ be the lattice of flats of a matroid $\M$, with unique minimal element $\widehat{0}$ given by $\varnothing$, and unique maximal element $\widehat{1}$ given by $E$, and the partial order $\leqslant$ given by set containment. If $\M$ is simple, the ground set $E$ is identified with the set of elements of rank $1$ of $\L$, also called the \emph{atoms} of $\L(\M)$. For a subset $\calG\subseteq \calL \setminus \{\hat{0}\}$, and any element $F\in \calL$, 
\[
\calG_{\leqslant F} \coloneqq \{G \in \calG \mid G \leqslant F\}, 
\] and let $\max \calG_{\leqslant F}$ be the set of maximal elements in $\calG_{\leqslant F}$. 

\begin{definition}
    For a finite geometric lattice $\calL$, a subset $\calG$ in $\calL \setminus \{\hat{0}\}$ is a \textit{building set} of $\calL$ if, for any $F\in \L \setminus \{\zero\}$ the join map is an isomorphism of posets: 
    \[
    \psi_{F} \colon \prod_{G \in \max \G_{\leqslant F}}[\zero, G] \xrightarrow{\sim} [\zero, F], \quad (H_G)_G \mapsto \bigvee_G H_G. 
    \] The isomorphism $\psi_F$ is called the \emph{structural isomorphism} of $\G$ at $F$.  
\end{definition}
If $\G$ is a building set of $\L$, the pair $(\calL,\calG)$ of a geometric lattice $\calL$ and a building set $\calG$ will be referred to as a \emph{built lattice}, and the pair $(\M,\G)$ will be referred to as a \emph{built matroid}. For any $F \in \L$, the elements of $\max \calG_{\leqslant F}$ are called the $\G$-factors of $F$, denoted by $\fact_{\calG}(F)$. 

Every geometric lattice $\L$ admits two distinguished building sets, called the \emph{maximal building set} and the \emph{minimal building set} respectively:
\begin{align*}
    \G_{\max} &\coloneqq \L \setminus \{\zero\}, \text{ and }\\
    \G_{\min} &\coloneqq \{ F \in \L \setminus \{\zero\} \mid F \text{ is irreducible}\}, 
\end{align*}
where a flat $F$ is \emph{irreducible}, if the interval $[\zero, F]$ only admits trivial product decompositions. 

For any building set $\G$, we have the inclusion $\G_{\min} \subseteq \G \subseteq \G_{\max}$. In particular, $\G$ contains all the atoms of $\L$. If $E \in \G$, then $(\M, \G)$ is called an \emph{irreducible built matroid}. 
In particular, if $\M$ is not connected, then $E \notin \G_{\min}$ and $(\M, \G_{\min})$ is not irreducible; on the other hand, $(\M, \G_{\max})$ is always irreducible.  We will make use of the following useful characterization of building sets.  

\begin{proposition}[{\cite[Proposition~2.11]{backman2025convexgeometrybuildingsets}}]\label{prop:backman-char-building-set}
A subset $\G\subset \L\setminus \{\zero\}$ is a building set of $\L$ if and only if $\G_{\min} \subseteq \G$, and for any $G, G' \in \G$, the following implication holds: 
\[
G \wedge G' \neq \zero \implies G \vee G' \in \G.
\] 
\end{proposition}

\subsection{Nested set complexes}
To a geometric lattice $\mathcal{L}$ and a building set $\mathcal{G}$, one can associate a simplicial complex called the \emph{nested set complex}, which reflects the combinatorial and intersection-theoretic properties of both the boundary strata of the wonderful compactification $W_{\M, \G}$ if $\M$ is a realizable matroid, and the torus-invariant subvarieties of the toric variety $X_{\M, \G}$ (see Section \ref{subsec:nested-set-fans}) for general built matroids. Recall that a collection of pairwise incomparable elements in a poset is called an \emph{antichain.}

\begin{definition}
\label{def:nested-set}
    A subset $\calS \subseteq \calG$ is \textit{nested}, if for every antichain $\calA \subseteq \S$ the following implication holds: 
    \[
    \forall \, \{S_1, \ldots, S_{k}\} \subseteq \calA, \, k \geqslant 2 \implies S_1 \vee \cdots \vee S_k \notin \G. 
    \] 
\end{definition}

\begin{example}\label{ex:nested-sets}
    Some examples of nested sets are as follows. 
    \begin{enumerate}[\normalfont(i), leftmargin = 20pt, itemsep = 3pt] 
        \item If $\calG$ is the maximal building set, then a subset $\calS \subseteq \calG$ is nested if and only if it is a chain in $\calL \setminus \{\hat{0}\}$. 
        \item \label{item:nested-sets-min} If $\calG$ is the minimal building set, then a subset of $\calS \subseteq \calG$ is nested if and only if for every antichain $\{S_1, \ldots, S_k\}$ in $\calS$ with size $k \geqslant 2$, the join of the antichain is reducible, i.e., 
        \[
        [\zero, S_1 \vee \cdots \vee S_k] \cong \prod_{i\in I} [\hat{0}, F_i]
        \] for some collection $I$ of flats $\{F_i\}_{i \in I}$ in $\L$. 
        For example, if $\L$ is the lattice of flats of a Boolean matroid, then the minimal building set $\G_{\min}$ consists of all the atoms of $\L$, and any subset of $\G_{\min}$ is nested, because every flat in $\L$ of rank at least $2$ is reducible. 
    \end{enumerate}
\end{example}

The collection of nested sets of $(\calL,\calG)$ forms an abstract simplicial complex with vertex set $\G$, called the \textit{nested set complex} of $(\L, \G)$, and denoted $c\mathcal{N}(\calL, \calG)$. We also denote $\N(\L,\G) = c\N(\L,\G)|_{\G\setminus \max \G}.$ For instance, if $\calG$ is the maximal building set, then $c\N(\L, \G)$ and $\N(\L,\G)$ can be identified with the order complexes of $\calL\setminus \{\hat{0}\}$ and $\L\setminus\{\zero, \un\}$ respectively. 

\subsection{Restriction and contraction}\label{sec:restriction-contraction}

The familiar matroid operations of restriction and contraction extend naturally to the setting of built matroids.
\begin{definition}
Let $(\L, \G)$ be a built lattice. For any $F \in \L\setminus \zero$, we define the \emph{restriction} of $(\L,\G)$ to $F$ as the built lattice $(\L^F,\G^F)$ where 
\[    \L^F \coloneqq [\zero, F] \qquad\text{ and }\qquad
    \G^F \coloneqq \G \cap [\zero, F].
\]
Analogously, if $F\neq \un$ we define the \emph{contraction} of $(\L,\G)$ at $F$ as the built lattice $(\L_F,\G_F)$, where
\[
\L_F \coloneqq [F, \un] \qquad \text{ and } \qquad
\G_F \coloneqq \{ F \vee G \, | \, G \in \G  \} \setminus \{ F \}.
\]
\end{definition}
If the flat $F$ is a singleton $\{e\}$ we will write $\G/e$ instead of $\G_{\{e\}}$. The following lemma will allow us to reduce our analysis to irreducible built lattices, simplifying many arguments.  
\begin{lemma}[{\cite[Lemma 2.17]{coron2026structuralpropertiesnestedset}}]\label{lem:reducible}

For any built lattice $(\L, \G)$, if $G_1, \ldots, G_k$ are the maximal elements of $\G,$ then we have an isomorphism of nested set complexes
\begin{equation}\label{eq:nested-set-cone}
c\N(\L,\G) \simeq \Cone(\N(\L^{G_1}, \G^{G_1}))* \cdots * \Cone(\N(\L^{G_k}, \G^{G_k})).
\end{equation}
\end{lemma}

The nested set complexes $\N(\L^F,\G^F)$ and $\N(\L_F,\G_F)$ appear naturally inside the link of the nested set complex $\N(\L, \G)$ at the vertex $F$. To describe these links, we recall the classical notion of local intervals of nested sets (see \cite{BDF_2022}, \cite{coron2024matroidsfeynmancategorieskoszul} for instance). 
\begin{definition}
Let $(\L, \G)$ be a built lattice and $\S \in \N(\L, \G)$ a nested set. For any $G \in \S \cup \un$, the \emph{local interval of $\S$ at $G$}, denoted $L^G({\S}),$ is the built lattice 
\[
L^G(\S) \coloneqq \left(\left[\bigvee \S_{< G}, G\right], \G_{\bigvee \S_{< G}}^G\right).
\]
In the sequel we will denote $J^G \coloneqq \bigvee \S_{< G}$ (the dependency in $\S$ being omitted). 
\end{definition}

\begin{proposition}\label{prop:composition-nested-sets}
    Let $(\M, \G)$ be an irreducible built matroid. 
    \begin{enumerate}[\normalfont(i), leftmargin = 20pt, itemsep = 3pt]
        \item\label{it:comp-nested-sets-1} For any $F \in \L(\M)\setminus \{\un\}$, every element $G \in \G_{F}$ has a unique $\G$-factor which is not a $\G$-factor of $F$, denoted by $G/F$.\footnote{In \cite{coron2024matroidsfeynmancategorieskoszul}, $G/F$ is denoted as $\mathrm{Comp}_{F}(G)$.}
        \item\label{it:comp-nested-sets-2}For any nested set $\S \in \Nest(\M, \G),$ and for any collection of nested sets $(\S_G)_{G \in \S \cup \un}$ with $\S_G \in \Nest(L^G(\S))$ for every $G \in \S\cup \un$,  the collection 
        \[ \S \circ (\S_G)_G \coloneqq \S \cup \bigcup_{G \in \S \cup \un} \{ \, H/J^G \, | \, H \in \S_G \, \} \]
        is a nested set of $(\M, \G)$. \footnote{If we only have a collection of nested sets $(\S_G)_{G\in I}$ for some proper subset $I \subset \S\cup \{\un\}$, then $\S\circ(\S_G)_G$ means the composition of nested sets where we tacitly set $\S_G = \emptyset$ for $G \notin I$.} \footnote{When computing the unique factor of each $H/J^{G}$ for each $G$, we compute the unique factor with respect to the global building set $\G$, rather than the local building set of $L^{G}(\S)$.} \footnote{The notation $\circ$ comes from the fact that this operation can be interpreted as some composition of morphisms in a certain Feynman category (see \cite[Section~3]{coron2024matroidsfeynmancategorieskoszul} for more details).}
        \item\label{it:comp-nested-sets-3} Using the same notations as above, for any $G \in \S\cup \un$, and any $G'\in \S_G$, we have an isomorphism of built lattices 
        \[
        L^{G'/J^G}(\S\circ(\S_G)_G)\simeq L^{G'}(\S_G),
        \]
        which is given by taking the join with all the factors of $G'$ which are not below $G'/J^G$. 
        \item\label{it:comp-nested-sets-4} For any nested set $\S \in \N(\M, \G)$, the map $\S \circ -$ establishes a bijection between the collections of nested sets $(\S_G)_{G\in \S \cup \{\un\}}$ where $\S_G \in \N(L^G(\S))$ for any $G \in \S \cup \un$, and the nested sets of $(\M, \G)$ containing $\S.$ This gives an isomorphism of simplicial complexes 
        \[
        \lk_{\N(\L, \G)}(\S) \simeq \bigast_{G \in \S \cup \un } \N(L^G(\S)).
        \] 
        \item \label{it:comp-nested-sets-5} The composition $\circ$ of nested sets is associative, meaning that in the situation of \ref{it:comp-nested-sets-2}, if in addition, for every $G\in \S\cup \un$ we have nested sets $(\S_G^F)_{F\in \S_{G} \cup \{G\}}$ such that $\S_G^F\in L^F(\S_G)$ for all $F\in \S_{G}$, then we have the equality of nested sets 
        \[
        \S\circ ((\S_G)_{G}\circ (\S^F_G)_F) = (\S\circ (\S_G)_G) \circ (\S^F_G)_F.
        \] 
    \end{enumerate}
\end{proposition}
For the proof of the above proposition, see \cite[Section 2.3]{coron2024matroidsfeynmancategorieskoszul}, \cite[Section 2]{BDF_2022}, and \cite[Section 2.6]{coron2026structuralpropertiesnestedset}. In the setting of Proposition~\ref{prop:composition-nested-sets} \ref{it:comp-nested-sets-2}, when dealing with the maximal building set, nested sets are chains of elements and the composition of nested sets is simply the concatenation of chains. The reader is invited to check that Proposition~\ref{prop:composition-nested-sets}\ref{it:comp-nested-sets-5} is straightforward in that case.
\begin{example}[Composition of nested sets for maximal building sets]
    Let $\U_{4,4}$ be the Boolean matroid on $E = \{1, 2, 3, 4\}$, and $\G = \G_{\max}$. In this case, all nested sets are chains. Let $\S = \{1, 12\}$, and let $(\S_G)_{G\in \S\cup\{1234\}}$ be the collection of nested sets: 
    \begin{align*}
        \S_{1} = \emptyset \in \N([\zero, 1], \{1\}), \S_{12} = \varnothing \in \N([1, 12], \{12\}), \\
        \S_{1234} = \{123\} \in \N([12, 1234], \{123, 124, 1234\}). 
    \end{align*}
    Then the composition is 
    \[
    \S \cup \{123\} = \{1, 12, 123\}. 
    \]
\end{example}

We give below an example of composition of nested sets in a non-maximal building set. 
\begin{example}[Composition of nested sets]
    Let $\U_{3,3}$ be the Boolean matroid on $E = \{1,2,3\}$. Let $\G$ be the building set $\{1,2,3,12,123\}$. Let $\S = \{2\}, \S_{2} = \emptyset , \S_{123} = \{23\}$. 
    The only factor of $23$ which is not $2$ is $3$, so we get $23/2 = 3$ and thus
    \[
    \S \circ (\S_2 , \S_{123}) = \{2\}\circ \{23\} = \{2,3\},
    \] 
    which is indeed a nested set. 
\end{example}

\subsection{Single-element extension}\label{subsec:single-element-extension}
Single-element extension is a classical matroid operation that augments a matroid one element at a time. It was invented by Crapo \cite{crapo1964single} as a systematic technique to study all matroids on a fixed ground set. Every single element extension corresponds to a distinguished set of elements in the geometric lattice, called a modular cut. 

We extend these notions naturally to matroids with building sets. 
\begin{definition}
    A \emph{modular pair} of $\M$ is a pair of flats $F, G \in \L(\M)$ such that
    \[
    \rk(F \vee G) + \rk(F \wedge G) = \rk(F) + \rk(G). 
    \] 
    A \emph{modular cut} of $\M$ is a subset $\Mcut \subset \L(\M)$ which is upward closed and such that for every modular pair $F, G$ of $\M$, the following implication holds: 
    \[
    F, G \in \Mcut \implies  F \wedge G \in \Mcut.
    \]
\end{definition}

\begin{definition}
    The \emph{single-element extension} of $\M$ along a modular cut $\Mcut$, denoted $\M \cup_{\Mcut} e$, is the matroid on the ground set $E\sqcup e$ defined by the following rank function: for any $S \subseteq E$, 
    \[
    \rk_{\M\cup_{\Mcut} e}(S) \coloneqq \rk_{\M}(S),
    \quad \text{ and } \quad 
    \rk_{\M \cup_{\Mcut} e}(S\sqcup e) \coloneqq \begin{cases}
    \rk_{\M}(S) & \text{if } \cl_{\M}(S) \in \Mcut, \\
    \rk_{\M}(S) + 1  & \text{otherwise. }
    \end{cases}
    \]
\end{definition}

The lemma below immediately follows from the definition. 
\begin{lemma}\label{lem:e-membership-extension}
    For any $S \subseteq E$, $e \in \cl_{\M\cup_{\Mcut}e}(S)$ if and only if $\cl_{\M}(S) \in \Mcut$. 
\end{lemma}

\begin{proof}
    The element $e$ is in $\cl_{\M\cup_{\Mcut}e}(S)$ if and only if $\rk_{\M\cup_{\Mcut}e}(S \sqcup e) = \rk_{\M\cup_{\Mcut}e}(S) = \rk_{\M}(S)$, if and only if $\cl_{\M}(S) \in \Mcut$. 
\end{proof}

\begin{lemma}[{\cite[Corollary 7.5]{oxley}}]\label{lem:extension-flats}
    The flats of $\M \cup_{\Mcut} e$ consist of three distinct classes: 
    \begin{enumerate}[\normalfont(i),leftmargin=15pt, itemsep=3pt]
        \item every flat $F$ of $\M$ not in $\Mcut$; 
        \item every set $F \cup e$ for every flat $F$ of $\M$ in $\Mcut$; and 
        \item every set $F \cup e$ for every flat $F$ of $\M$ not in $\Mcut$.
    \end{enumerate}
\end{lemma}
\begin{lemma}[{\cite[Lemma~7.2.2]{oxley}}]\label{lem:extension-to-modular-cut}
    The modular cut $\Mcut$ consists of the set of flats $F$ of $\M$ such that $F \sqcup e$ is a flat of $\M \cup_{\Mcut} e$ and 
    \[
    \rk_{\M \cup_{\Mcut} e}(F \sqcup e) = \rk_{\M \cup_{\Mcut} e}(F)
    \]
\end{lemma}

\begin{lemma}\label{lem:extension-of-F-is-flat}
    If $F \in \Mcut$ is a flat of $\M$, then $F \sqcup e = \cl_{\M \cup_{\Mcut} e}(F)$ is a flat of $\M \cup_{\Mcut} e$. 
\end{lemma}

\begin{proof}
    First, we show that $F \sqcup e$ is a flat in $\M \cup_{\Mcut} e$. For any $a \in (E \sqcup e) \setminus (F \sqcup e)$, 
    \[
    \rk_{\M \cup_{\Mcut} e}(F \sqcup e \sqcup a) = \rk_{\M \cup_{\Mcut} e}((F \sqcup a) \sqcup e) = \rk_{\M}(F \sqcup a) = \rk_{\M}(F) + 1 > \rk_{\M \cup_{\Mcut} e}(F \sqcup e),
    \] where the second equality follows from definition together with the facts that $\Mcut$ is upward closed and $F \leqslant \cl_{\M}(F \sqcup a)$ in $\M$. Next, we show that $F \sqcup e$ is contained in the closure $\cl_{\M \cup_{\Mcut} e}(F)$. Indeed, since $F \in \Mcut$, $\rk_{\M \cup_{\Mcut} e}(F \sqcup e) = \rk_{\M}(F) = \rk_{\M \cup_{\Mcut} e}(F)$ by the definition of a single-element extension. Therefore, since $\cl_{\M \cup_{\Mcut} e}(F)$ is the smallest flat in $\M \cup_{\Mcut} e$ containing $F$, $F \sqcup e = \cl_{\M \cup_{\Mcut} e}(F)$. 
\end{proof}
A single-element extension is \emph{proper} if the underlying modular cut is proper. We have seen that one can construct a single-element extension from a modular cut. The reciprocal is true; every single-element extension gives rise to a modular cut, by the following classical proposition.  
\begin{proposition}[{\cite[Section 7.2]{oxley}}]\label{prop:ext-del}
The deletion of $e$ from any single-element extension $(\M\cup_{\Mcut} e)$ is $\M$: 
\[
\M = (\M\cup_{\Mcut}e) \setminus e.
\] 
Moreover, for any $e \in E$, together with the modular cut of $\M \setminus e$ 
\[
\Mcut_e \coloneqq \{ F \in \L(\M\setminus e) \, | \, e \in \cl_{\M}(F) \},
\] we have the equality of matroids
\[
\M = (\M\setminus e)\cup_{\Mcut_e} e. 
\]  
\end{proposition}

If a modular cut is proper and nonempty, the single-element extension has interesting properties, described below. We henceforth focus on only proper and nonempty modular cuts, unless specified otherwise. 

\begin{lemma}\label{lem:e-not-loop-nor-coloop}
    The element $e$ is neither a loop nor a coloop in $\M \cup_{\Mcut} e$ if and only if the modular cut $\calM$ is proper and nonempty. 
\end{lemma}

\begin{proof}
    The element $e$ is a loop in $\M \cup_{\Mcut} e$ if and only if $\rk_{\M \cup_{\Mcut} e}(\varnothing \sqcup e) = \rk_{\M}(\varnothing) = 0$, if and only if $\cl_{\M}(\varnothing) \in \Mcut$, if and only if $\Mcut$ is not proper. The element $e$ is a coloop in $\M \cup_{\Mcut} e$, if and only if $\rk_{\M \cup_{\Mcut} e}(E \sqcup e) = \rk_{\M}(E) + 1$, if and only if $\cl_{\M}(E) \notin \Mcut$, if and only if $\Mcut = \varnothing$ by upward-closure. The desired statement follows by combining these two conditions. 
\end{proof}

\begin{lemma}\label{lem:min-mcut-is-irreducible}
    If $F$ is a minimal element in a proper modular cut $\Mcut$, then $F \sqcup e$ is an irreducible element in $\M \cup_{\Mcut} e$. 
\end{lemma}

\begin{proof}
    By Lemma ~\ref{lem:extension-of-F-is-flat}, $F \sqcup e$ is indeed a flat in $\M \cup_{\Mcut} e$. Suppose for contradiction that $F \sqcup e$ is reducible. Then by definition, there exist disjoint and nonempty flats $G, H$ in $\M \cup_{\Mcut} e$ such that $F \sqcup e = G \sqcup H$ with 
    \[
    \rk_{\M \cup_{\Mcut}e }(F \sqcup e) = \rk_{\M \cup_{\Mcut} e}(G) + \rk_{\M \cup_{\Mcut} e}(H).
    \] Assume $e \in G$, without loss of generality. If $e$ is a coloop in $G$, then $e$ is a coloop in $F \sqcup e = G \sqcup H$. However, since $F \in \Mcut$, $\rk_{\M \cup_{\Mcut} e}(F \sqcup e) = \rk_{\M}(F)$ implies that $e$ is not a coloop in $F \sqcup {e}$, a contradiction. If $e$ is not a coloop in $G$, then $\rk_{\M \cup_{\Mcut} e}((G \setminus e) \sqcup e) = \rk_{\M \cup_{\Mcut} e}(G \setminus e)$. By Lemma~\ref{lem:extension-to-modular-cut}, $G \setminus e$ is an element in the modular cut $\Mcut$. Since both $G$ and $H$ are nonempty with $G \sqcup H = F \sqcup e$, $G \setminus e$ is a proper subset of $F$. Furthermore, since $e$ is not a coloop in $F\sqcup e$ we have $G\setminus e \neq \emptyset$. Therefore, $G \setminus e$ is a proper nonempty subset of $F$ which is also in $\Mcut$, contradicting the minimality of $F$. 
\end{proof}

The notion of single-element extension naturally extends to built matroids. Let $(\M, \G)$ be a built matroid, together with its built lattice $(\L, \G)$.  
\begin{definition}
A modular cut $\Mcut \subset \L(\M)$ is \emph{$\G$-compatible} if $\min \Mcut \subset \G.$ For any modular cut $\Mcut,$ the \emph{single-element extension} $\G\cup_{\Mcut}e$ of $\G$ is 
\[
\G \cup_{\Mcut}e \coloneqq \{\cl_{\M \cup_{\Mcut} e}(G) \mid G \in \calG \} \sqcup \{\{e\}\} \subset \L(\M\cup_{\Mcut} e). 
\]
\end{definition}

\begin{example}
Let $\M = \U_{3,3}$ be the Boolean matroid over three elements $1,2,3$, and set $\Mcut \coloneqq \{123\}$, which is a modular cut of $\U_{3,3}.$ The single-element extension $\M\cup_{\Mcut}4$ is the corank $1$ uniform matroid over four elements $1,2,3,4$. Consider the two building sets $\G \coloneqq \{1,2,3\}$ and $\G' \coloneqq \{1,2,3, 123\}$ of $\U_{3,3}.$ We have 
\[
\G\cup_{\Mcut}4 = \{1,2,3,4\} \quad \textrm{    and    } \quad \G'\cup_{\Mcut} 4 = \{1,2,3,4, 1234\}.
\]
\end{example}

Note that in the above example, the subset $\G\cup_{\Mcut} 4$ is not a building set of $\M\cup_{\Mcut} e$, whereas $\G'\cup_{\Mcut}e$ is a building set of $\M\cup_{\Mcut}e$. The rest of this section is devoted to providing a characterization of modular cuts for which the set $\G \cup_{\Mcut} e$ is indeed a building set of $\M \cup_{\Mcut} e$. We prepare this characterization by the following lemma, which may be viewed as a converse of Lemma~\ref{lem:min-mcut-is-irreducible}. 
\begin{lemma}\label{lem:extension-irreducibles-in-building-sets}
    Let $\Mcut$ be a proper, nonempty and $\G$-compatible modular cut, and let $F$ be a flat of $\M$. If $F \sqcup e$ is irreducible in $\M \cup_{\Mcut} e$, then $F \in \G$ and $F \sqcup e \in \G \cup_{\Mcut} e$.  
\end{lemma}

\begin{proof}
Since $F \sqcup e$ is irreducible in $\M \cup_{\Mcut} e$, $e$ is not a coloop in $F$, and thus $F \sqcup e = \cl_{\M \cup_{\Mcut} e}(F)$. Therefore, $F \in \G$ is equivalent to $F \sqcup e \in \G \cup_{\Mcut} e$, and we want to prove that $F \in \G$. By Lemma~\ref{lem:e-membership-extension}, $e \in F \sqcup e = \cl_{\M \cup_{\Mcut} e}(F)$ implies that $F \in \Mcut$. Let $G_1, \ldots, G_k$ be the factors of $F$ in $\G$. Since $\Mcut$ is $\G$-compatible, at least one of the factors $G_1, \ldots, G_k$ is in $\Mcut$. We claim that exactly one of the factors is in $\Mcut$. Indeed, by the definition of $\G$-factors, for every $1 \leqslant i < j \leqslant k$, the pair $G_{i}, G_j$ is a modular pair, with $G_i \wedge G_{j} = \zero$. By the definition of a modular cut, if any pair $G_i, G_j$ were in $\Mcut$, then $\zero \in \Mcut$, contradicting $\Mcut$ is proper. Therefore, there exists exactly a unique factor $G_\ell$ such that $G_{\ell} \in \Mcut$. By Lemma ~\ref{lem:extension-flats}, every $G_{i}$ for $i \ne \ell$ is a flat in $\M \cup_{\Mcut} e$. Therefore, as a set $F$ can be written as a disjoint union 
\[
F \sqcup e= (G_{\ell} \sqcup e) \sqcup \bigg(\bigsqcup_{i \ne \ell} G_i \bigg) 
\] with rank in the matroid $\M \cup_{\Mcut} e$ equal to the sum of the ranks of the $G_i$'s
\[
\rk_{\M \cup_{\Mcut} e}(F \sqcup e) = \rk_{\M \cup_{\Mcut} e}(G_{\ell} \sqcup e) + \sum_{i \ne \ell} \rk_{\M \cup_{\Mcut} e}(G_i). 
\]
By irreducibility of $F \sqcup e$, this decomposition cannot exist, unless $F \sqcup e = G_{\ell} \sqcup e$ and $F \in \G$. 
\end{proof}
The following proposition is the main new result of this section. 
\begin{proposition}\label{lem:built-see}
Let $\Mcut$ be a proper nonempty modular cut. The single-element extension $\G\cup_{\Mcut}e$ is a building set of $\M\cup_{\Mcut}e$ if and only if $\Mcut$ is $\G$-compatible.
\end{proposition}

\begin{proof}
Suppose $\Mcut$ is a modular cut of $\M$ such that $\G\cup_{\Mcut}e$ is a building set of $\M\cup_{\Mcut}e$. For every flat $F$ which is minimal in $\Mcut$, we want to show that $F \in \G$. By Lemma ~\ref{lem:min-mcut-is-irreducible} and Proposition~\ref{prop:backman-char-building-set}, $F \sqcup e \in \G \cup_{\Mcut} e$. Then the definition of $\G \cup_{\Mcut} e$ implies that $F \sqcup e$ is a flat of the form $\cl_{\M \cup_{\Mcut} e}(G)$ for some $G \in \G$. By Lemma~\ref{lem:extension-of-F-is-flat}, $F \sqcup e = \cl_{\M \cup_{\Mcut} e}(F)$. Therefore, $F = G$, and $F \in \G$.

Suppose $\Mcut$ is a $\G$-compatible modular cut. To prove that $\G\cup_{\Mcut}e$ is a building set, by Proposition \ref{prop:backman-char-building-set}, it suffices to prove that $\G\cup_{\Mcut}e$ contains the irreducible flats of $\M\cup_{\Mcut}e$, and that for every $G, G' \in \G\cup_{\Mcut}e$, the following implication holds
\[
G \wedge G' \neq \zero \implies G \vee G' \in \G\cup_{\Mcut}e.
\] 
For the first property, let $F$ be an irreducible flat of $\M\cup_{\Mcut}e$ and we want to show that $F \in \G\cup_{\Mcut} e$. There are $2$ cases depending on whether $F$ contains $e$. If $F$ contains $e$, then Lemma ~\ref{lem:extension-irreducibles-in-building-sets} implies that $F \in \G \cup_{\Mcut} e$. If $F$ does not contain $e$, then $F$ is also irreducible in $\M$. Proposition ~\ref{prop:backman-char-building-set} implies $F \in \G$ and $F = \cl_{\M \cup_{\Mcut}e}(F) \in \G \cup_{\Mcut} e$. 
For the second property, it immediately follows because $\G$ also satisfies that property.  
\end{proof}

In the situation of the above lemma, we denote by $(\M, \G) \cup_{\Mcut} e$ the single-element extension $(\M\cup_{\Mcut}e, \G\cup_{\Mcut}e)$ of the built matroid $(\M, \G)$ along the $\G$-compatible modular cut $\Mcut.$

\subsection{Deletion}
There are two types of deletion operations on a built matroid $(\M, \G)$. The first type is induced by a single-element deletion from the ground set.

\begin{lemma}\label{lem:built-del-e}
    Let $(\M, \G)$ be a simple built matroid, and let $e\in E(\M)$. The subset 
    \[
    \G\setminus e \coloneqq \{ S \subseteq E(\M)\setminus e  \, | \, \cl_{\M}(S) \in \G\} \subset \L(\M\setminus e)
    \] 
    is a building set of $\M\setminus e$. 
\end{lemma}
\begin{proof}
Let $F$ be an irreducible flat of $\M\setminus e.$ If $e\notin \cl_{\M}(F),$ then $\cl_{\M}(F)=F $ is irreducible in $\M$, and so $\cl_{\M}(F)$ belongs to $\G$, which implies that $F$ belongs to $\G\setminus e$. On the other hand, if $e \in \cl_{\M}(F)$, using the same arguments as in the proof of Lemma~\ref{lem:min-mcut-is-irreducible} $\cl_{\M}(F)$ is irreducible in $\M$, which implies that $\cl_{\M}(F)$ belongs to $\G,$ and so $F$ belongs to $\G\setminus e.$

If $G$ and $G'$ are two flats of $\M\setminus e$ in $\G\setminus e$ such that $G \wedge G' \neq \zero$, then $\cl_{\M}(G)$ and $\cl_{\M}(G')$ are two flats of $\M$ in $\G$ with $\cl_{\M}(G) \wedge \cl_{\M}(G') \neq \zero$. By Proposition~\ref{prop:backman-char-building-set}, this implies that $\cl_{\M}(G)\vee \cl_{\M}(G) \in \G$. One can check that $\cl_{\M}(G\vee G') = \cl_{\M}(G) \vee \cl_{\M}(G')$, which implies $G \vee G' \in \G\setminus e$. 

By Proposition \ref{prop:backman-char-building-set} this implies that $\G\setminus e$ is a building set of $\M\setminus e.$
\end{proof}
We denote the single-element deletion from a built matroid $(\M, \G)$ by 
\[
(\M, \G)\setminus e \coloneqq (\M\setminus e, \G\setminus e).
\] 
We have the following result, relating the deletion and the extension of a built matroid. 
\begin{proposition}\label{prop:built-del-ext}
    If $\G$ is a building set of $\M$, then the modular cut of $\M \setminus e$ 
    \[
    \Mcut_e \coloneqq \{ F \in \L(\M \setminus e) \, | \, e\in \cl_{\M}(F)\}
    \] 
    is a $(\G\setminus e)$-compatible modular cut. Moreover,  if we write 
    \[
    S_e \coloneqq \{G \in \L(\M) \mid e \text{ is a coloop in } G \text{ and } G \ne \{e\}\}, 
    \]
    then there is an equality of building sets on $\M = (\M \setminus e) \cup_{\Mcut_e} e$, 
    \[
    (\G\setminus e)\cup_{\Mcut_e} e = \G \setminus S_e. 
    \]  
\end{proposition}

\begin{proof}
If $F$ is a flat of $\M\setminus e$ which is a minimal element of $\Mcut_e$, then by Lemma \ref{lem:min-mcut-is-irreducible} we have that $F\sqcup e$ is irreducible in $\G$, which implies that $F\sqcup e = \cl_{\M}(F)$ belongs to $\G$. By definition of $\G\setminus e$ this implies that $F$ belongs to $\G\setminus e$. 

By definition we have 
\begin{align*}
    (\G\setminus e) \cup_{\Mcut} e &= \{\cl_{\M}(G) \, | \, G\in \L(\M\setminus e) \textrm{ s.t. } \cl_{M}(G) \in \G\} \subset \G.
\end{align*}
If a flat $F$ of $\M$ contains $e$ as a coloop, then it is not the closure of a flat of $\M\setminus e$. Therefore,
\[
(\G\setminus e) \cup_{\Mcut} e \subset \G\setminus S_e.
\] 
On the other hand, every flat $F$ of $\M$ not containing $e$ as coloop is the closure of the flat $F\setminus e$ of $\M\setminus e$, and in that case, if $F$ belongs to $\G$, then $F\setminus e$ belongs to $\G\setminus e$. This implies that 
\[
\G\setminus S_e \subset (\G\setminus e) \cup_{\Mcut} e.
\]
\end{proof}

The second type of deletion operation is to delete an element in the building set. The resulting set is not necessarily a building set, and we provide a characterization for when such deletions indeed result in a building set.

\begin{lemma}\label{lem:building-set-deletion-of-one-element}
Let $(\L, \G)$ be a built lattice. For any element $G$ in $\G$, the set $\G\setminus \{G\}$ is a building set of $\L$ if and only if the join gives an isomorphism of posets
\begin{equation}
\label{eq:product-over-factors-after-deletion}
 \prod_{F \in \max (\calG \setminus G)_{\leqslant G}}[\zero, F] \cong [\zero, G]. 
\end{equation}
\end{lemma}

\begin{proof}
If $\G\setminus \{G\}$ is a building set, then the structural isomorphism at $G$ gives the isomorphism \eqref{eq:product-over-factors-after-deletion}. If we have the isomorphism in \eqref{eq:product-over-factors-after-deletion}, then we will use Proposition \ref{prop:backman-char-building-set} to prove that $\G\setminus \{G\}$ is a building set. We first check that $\G \setminus \{G\}$ contains all the irreducible flats of $\L$. Since $\G$ is a building set, it must contain the irreducible flats of $\L$. Furthermore, the isomorphism \eqref{eq:product-over-factors-after-deletion} implies that $G$ is not irreducible in $\L$. Therefore, the deletion $\G \setminus \{G\}$ still contains all the irreducible flats of $\L$. We now check that for every pair $H, H' \in \G \setminus \{G\}$ with $H \wedge H' \ne \zero$, we have $H \vee H' \in \G \setminus \{G\}$. Indeed, Proposition \ref{prop:backman-char-building-set} implies that $H \vee H' \in \G$. Moreover, if $H \vee H'$ was equal to $G$, then $H$ and $H'$ would be below two different elements of $\max (\calG \setminus G)_{\leqslant G}$, and the isomorphism \eqref{eq:product-over-factors-after-deletion} would imply $H \wedge H' = \zero$, a contradiction. 
\end{proof}

\subsection{Truncation}
We have seen in Section \ref{subsec:single-element-extension} that the deletion of a single-element extension $\M\cup_{\Mcut} e$ at $e$ is simply the original matroid $\M$. One can also describe the contraction of $\M\cup_{\Mcut} e$ at $e$, via the notion of truncation. 
\begin{definition}
Let $\M$ be a matroid, and $\Mcut$ a modular cut of $\M.$ The truncation $\Tr_{\Mcut}(\M)$ of $\M$ along $\Mcut$  is the matroid on the ground set $E(\M)$ defined by the rank function 
\begin{equation*}
    \rk_{\Tr_{\Mcut}(\M)}(S) = \begin{cases}
        \rk_{\M}(S) -1 &\textrm{ if } \cl_{\M}(S) \in \Mcut\\
        \rk_{\M}(S) &\textrm{ otherwise. } 
    \end{cases}
\end{equation*}
\end{definition}
It is a standard fact of matroid theory that the above rank function does define a matroid. In terms of the lattice of flats, we have the following lattice-theoretic description of the truncation along a modular cut. Let the \emph{collar} of $\Mcut$ be the set of all flats of $\M$ not in $\Mcut$, but which are covered by members of $\Mcut$: 
    \[
    C_{\Mcut} \coloneqq \{ F \in \L(\M) \, | \, F \notin \Mcut \textrm{ and there exists }  e \in E(\M) \textrm{ such that } \cl_{\M}(F\cup e) \in \Mcut \}, 
    \] 
\begin{proposition}[{\cite[Section 7.3]{oxley}}]\label{prop:truncation-description}
    The lattice of flats of $\Tr_{\Mcut}(\M)$ is the lattice of $\M$ without the collar of $\Mcut$: 
    \[
    \L(\Tr_{\Mcut}(\M)) = \L(\M) \setminus C_{\Mcut}.
    \] 
\end{proposition}

\begin{example}[Examples of modular cuts]\leavevmode
    \begin{enumerate}[\normalfont(i), leftmargin = 20pt, itemsep = 3pt] 
        \item If $\Mcut = \varnothing$, then $\Tr_{\Mcut}(\M) = \M$.
        \item If $\Mcut = \{E(\M)\}$, then $\Tr_{\Mcut}(\M)$ is the usual truncation of $\M$, whose lattice of flats consists of all but the corank $1$ flats of $\M$. 
        \item If $\Mcut$ is atom-free, then $\Tr_{\Mcut}(\M)$ is a loopless matroid. If $\Mcut$ contains an atom $a$, then $\Tr_{\Mcut}(\M)$ contains a loop $a$. 
    \end{enumerate} 
\end{example}

\begin{example}
    Let $\M$ be the uniform matroid $\U_{r, n}$ on ground set $[n]$ of rank $r$. Then $C_{\Mcut} = \{F \mid F \notin \Mcut, \text{ and there exists } G \in \Mcut, \text{ with } \rk(G) = \rk(F) + 1\}$. For example, let $\M = \U_{3, 4}$. If the modular cut $\Mcut = \{12, 1234\}$, then $\Tr_{\Mcut}(\U_{3, 4}) \cong \U_{2, 3}$. If the modular cut $\Mcut = \{1, 12, 13, 14, 1234\}$, then $\Tr_{\Mcut}(\U_{3, 4}) \cong \U_{2, 3}$. However, the underlying flats of these two truncations are different. 
\end{example}

The truncation is classically known to be the contraction of the single-element extension at the newly added element, also called the elementary quotient of $\M$ with respect to $\Mcut$. 
\begin{proposition}[{\cite[Proposition 7.3.9]{oxley}}]\label{prop:ext-contract-is-trun}
There is an equality of matroids 
\begin{equation}\label{eq:contraction:see}
\left(\M\cup_{\Mcut}e\right)/e = \Tr_{\Mcut}(\M).
\end{equation}
\end{proposition}
We now extend the notion of truncation to built matroids. Let $(\M, \G)$ be a built matroid with its built lattice $(\L, \G)$. We will say that a modular cut $\Mcut$ is \emph{atom-free} if it does not contain any atom. 
\begin{definition}
Let $\Mcut$ be an atom-free $\G$-compatible modular cut of $\M$. The truncation of $\G$ along $\Mcut$, denoted $\Tr_{\Mcut}(\G)$, is the subset of $\L(\Tr_{\Mcut}(\M))$ defined by 
\[
\Tr_{\Mcut}(\G) \coloneqq \G \cap \L(\Tr_{\Mcut}(\M)).
\] 
\end{definition}

We have the following analog of equality \eqref{eq:contraction:see}.
\begin{lemma}\label{lem:built-truncation-equality}
For any atom-free $\G$-compatible modular cut $\Mcut$, the following holds:
\[
\Tr_{\Mcut}(\G) = (\G\cup_{\Mcut}e)/e.
\] 
In particular, $\Tr_{\Mcut}(\G)$ is a building set of $\Tr_{\Mcut}(\M)$. 
\end{lemma}

\begin{proof}
Recall that the contraction of building set is defined as
\[
(\G \cup_{\Mcut} e)/e := \{F \subseteq E(\M \cup_{\Mcut} e) \setminus e \mid F \cup e \in (\G \cup_{\Mcut} e) \vee e\}.
\] 

Let $F \subseteq E(M)$ be a flat of $\Tr_{\Mcut}(\M)$ such that $F \cup e = G \vee e$ for some $G \in \G \cup_{\Mcut} e$. If $F \cup e$ contains $e$ as a coloop in $\M\cup_{\Mcut}e$, then we necessarily have $F = G \in \G$, so we must have $F \in \G \cap \mathcal{L}(\operatorname{Tr}_{\Mcut}(M)) = \operatorname{Tr}_{\Mcut}(\G)$. On the other hand if $e$ is not a coloop of $F \cup e$, then $G$ and $e$ are not the factors of $F \cup e$ in the building set $\G \cup_{\Mcut} e$ (Lemma 2.25), and so we must have $F \cup e \in \G \cup_{\Mcut} e$. This shows that $F$ belongs to $\G$, and so $F$ belongs to $\operatorname{Tr}_{\Mcut}(\G)$.

Let $F \subseteq E(M)$ be an element of $\operatorname{Tr}_{\Mcut}(\G)$. If $F \notin \Mcut$, then $F = \cl_{\M\cup_{\Mcut}e}(F)$ so $F \in \G \cup_{\Mcut} e$ and we have that $F \cup e = F \vee e \in (\G \cup_{\Mcut} e) \vee e$, which shows that $F$ belongs to $(\G \cup_{\Mcut} e)/e$. If $F \in \Mcut$, then we have $F \cup e = \cl_{M \cup_{\Mcut} e}(F) \in \G \cup_{\Mcut} e$, and so $F \cup e$ belongs to $(\G \cup_{\Mcut} e) \vee e$, which proves that $F$ belongs to $(\G \cup_{\Mcut} e)/e$.
\end{proof}

We denote the truncation of a built matroid $(\M, \G)$ along a modular cut $\Mcut$ by $\Tr_{\Mcut}(\M, \G)$. 

\section{Complete and flag built matroids}
In this section, we introduce complete built matroids and flag built matroids. We study their behaviors under the built matroid operations introduced in the previous section. 

\subsection{Complete built matroids}
\label{subsec:complete}

Let $(\M, \G)$ be a loopless built matroid on a ground set $E$. Consider a total order $\vartriangleleft$ on $E$. Let us fix $F \in \L(\M)$ and $G \in \G$ such that $F \leqslant G$. If we write the complement $G \setminus F = \{e_1 \vartriangleleft e_2 \vartriangleleft \cdots \vartriangleleft e_k\}$, and set 
\[
F_i \coloneqq \cl_{\M}(F\cup \{e_1, \ldots, e_i)\}), 
\] then we define the \emph{$\vartriangleleft$-chain from $F$ to $G$} by  
\[ 
\ch_{\vartriangleleft}(F,G)\coloneqq \{F \leqslant F_1 \leqslant F_2 \leqslant \cdots \leqslant F_k =  G\}. 
\]
By the axioms of the closure operator of a matroid, this is a saturated chain from $F$ to $G.$
\begin{definition}\label{def:complete-built-matroid}
Given a total order $\vartriangleleft$ on $E$, the loopless built matroid $(\M, \G)$ is \emph{$\vartriangleleft$-complete}, if for any $F \in \L(\M)$, and any $G \in \G$ with $F \leqslant G$, the $\vartriangleleft$-chain from $F$ to $G$ is contained in $\G_{F}^G\cup \{F\}$. A loopless built matroid $(\M, \G)$ is \emph{complete}, if it is $\vartriangleleft$-complete for some total order $\vartriangleleft$ on $E$.  
\end{definition}

\begin{example}[Matroids with maximal building sets]
    By definition, any matroid $\M$ with the maximal building set $\G_{\max}$ is complete, with respect to any total order on $E$. 
\end{example}

\begin{example}[Augmented nested set complexes of matroids]\label{ex:complete-ex-aug}
    \emph{Augmented nested set complexes} of matroids, introduced by Braden, Huh, Matherne, Proudfoot, and Wang in \cite{braden-huh-matherne-proudfoot-wang} are examples of nested set complexes of complete built matroids. Indeed, by an observation of Chris Eur written in \cite[Section~5.1]{mastroeni-mccullough} the augmented nested set complex of a matroid $\M$ is isomorphic to the nested set complex of the built matroid $(\M', \G_{\aug})$, where $\M' = \Tr((\M^*\cup_{\emptyset} e)^*))$ is the \emph{free coextension} of $\M$ and $\G_{\aug}$ is the building set of $\L(\M')$ consisting of the flats of the form $F\cup e$ with $F\in \L(\M)$, together with all the atoms of $\L(\M').$ The built matroid $(\M', \G_{\aug})$ is complete with respect to any total order on $E(\M')$ such that $e$ is the smallest atom with respect to that order. 
\end{example}

The next lemma describes how $\vartriangleleft$-chains behave with joins, and interval restrictions. 
\begin{lemma}\label{lem:first-chain-restriction-contraction}
    Let $\M$ be a matroid and $\vartriangleleft$ a total order on $E(\M)$. For any $F \in \L(\M)$,  
    \begin{enumerate}[\normalfont(i), leftmargin = 20pt, itemsep = 3pt]
        \item \label{item:contraction} 
        there is an equality of sets $\ch_{\vartriangleleft}(F, \un) = F\vee \ch_{\vartriangleleft}(\zero, \un)$, and 
        \item \label{item:restriction} 
        there is an inclusion of sets $\ch_{\vartriangleleft}(\zero, \un)\cap [\zero, F] \subset \ch_{\vartriangleleft}(\zero, F)$.
    \end{enumerate}
\end{lemma}

\begin{proof}\leavevmode
    \begin{enumerate}[(i),leftmargin = 20pt, itemsep = 3pt]
    \item Since $\L(\M)$ is a geometric lattice, $F\vee \ch_{\vartriangleleft}(\zero, \un)$ is a maximal chain of $[F, \un]$. We prove by induction on $k$ that the maximal chains $F\vee \ch_{\vartriangleleft}(\zero, \un)$ and $\ch_{\vartriangleleft}(F, \un)$ have the same first $k$ elements for every $k \leqslant \rk \un - \rk F +1$. For $k = 1$, the first element of each chain is $F$. Assume that the $k$ first elements are the same for both maximal chains. We denote by $G$ the $k$-th element on both side, by $G'$ the $(k+1)$-th element of $\ch_{\vartriangleleft}(F, \un)$ and by $G''$ the $(k+1)$-th element of $F\vee \ch_{\vartriangleleft}(\zero, \un)$. We aim to show that $G'= G''$. By the definition of $\ch_{\vartriangleleft}(F, \un)$, we have
    \[
    G' = G\vee e \enspace \text{ with } e = \min (E(\M) \setminus G).
    \]
    On the other hand there is a maximal element $G_0$ of $\ch_{\vartriangleleft}(\zero,\un)$ such that $G = G_0 \vee F$ and a minimal element $G_0''$ of $\ch_{\vartriangleleft}(\zero,\un)$ such that we have $G'' = G''_0\vee F.$ The elements $G_0$ and $G''_0$ are consecutive in $\ch_{\vartriangleleft}(\zero, \un)$ so we have $G_0'' = G_0\vee e',$ where $e' = \min E(\M) \setminus G_0$. Since $G_0 < G$ we have $e' \leqslant e.$ Since $G'' = G_0''\vee F$ we have $e'\notin G$ and so $e \leqslant e',$ which gives $e= e'$. We then have 
    \[ G'' = G_0''\vee F = G_0\vee F \vee e = G\vee e = G'.\]
    \item Let us show by induction that the $k$ first elements of $\ch_{\vartriangleleft}(\zero, \un)\cap [\zero, F]$ belong to $\ch_{\vartriangleleft}(\zero, F)$ for all $k$ less than or equal to the size of $\ch_{\vartriangleleft}(\zero, \un)\cap [\zero, F]$. For $k = 1$ the first element of $\ch_{\vartriangleleft}(\zero, \un)\cap [\zero, F]$ is $\zero$ which belongs to $\ch_{\vartriangleleft}(\zero, F).$ Assume that the $k$ first elements of $\ch_{\vartriangleleft}(\zero, \un)\cap [\zero, F]$ belong to $\ch_{\vartriangleleft}(\zero, F).$ Denote by $G$ the $k$-th element of $\ch_{\vartriangleleft}(\zero, \un)$ and denote by $G'$ its $(k+1)-$th element. If $G' \nleqslant F$ we are done. Otherwise we have 
    \[ G' = G \vee (\min E(\M)\setminus G) = G\vee (\min F\setminus G) \in \ch_{\vartriangleleft}(\zero, F).\qedhere\]
    \end{enumerate}
\end{proof}



This has the following useful consequence. 
\begin{remark}\label{rmk:complete-simple-def}
Let $(\M,\G)$ be a built matroid, and let $\vartriangleleft$ be a total order on $E(\M).$ If for any $G \in \G$, we have the inclusion \[
\ch_{\vartriangleleft}(\zero, G) \subset \G\cup \{\zero\},\] 
then $(\M, \G)$ is $\vartriangleleft$-complete. Indeed, in that case, by Lemma \ref{lem:first-chain-restriction-contraction} for any $F\leqslant G$, we have 
\[
\ch_{\vartriangleleft}(F,G) = F\vee \ch_{\vartriangleleft}(\zero, G) \subset \G_F^G\cup \{F\}.
\] 
\end{remark}

\begin{example}[Supersolvable built matroids]\label{ex:supersolvable-built-matroid}
Supersolvable built matroids were introduced in \cite{Coron_2025}. A built matroid $(\M, \G)$ is \emph{supersolvable}, if there is a maximal chain $\zero = G_0 <  G_1 < \cdots < G_n = \un$ of modular elements in $\G$ such that for every $G \in \G$, the meet $G \wedge G_i$ is in $\G \cup \{\zero\}$ for every $0 \leqslant i \leqslant n$. In particular, this implies that $\L(\M)$ is  a supersolvable lattice (see \cite{stanley1972supersolvable} for the definition). Conversely, any supersolvable geometric lattice with its maximal building set is a supersolvable built lattice, and by \cite[Proposition 3.11]{Coron_2025}, any supersolvable geometric lattice with its minimal building set is also a supersolvable built lattice.

If $\vartriangleleft$ is a total order on $E(\M)$ satisfying the condition:  
\begin{equation}\label{eq:condition-ss-order}
\text{for all $e,e' \in E(\M)$, if there exists $i$ such that $e \in G_i$ and $e' \notin G_i$ then $e \vartriangleleft e'$,}
\end{equation}
then 
\[
\ch_{\vartriangleleft}(\zero, \un) = \{ G_0 = \zero \leqslant  G_1 \leqslant \cdots \leqslant G_n = \un \}.
\] 
More generally, for any $G \in \G$, 
\[
\ch_{\vartriangleleft}(\zero, G) = \{G_0 = \zero \leqslant G_1 \wedge G \leqslant \cdots \leqslant  G_n\wedge G = G \},
\] 
and the modularity of the $G_i$'s implies that the above chain is saturated. By Remark~\ref{rmk:complete-simple-def}, this shows that supersolvable built matroids are $\vartriangleleft$-complete for any total order $\vartriangleleft$ satisfying~\eqref{eq:condition-ss-order}. 
\end{example}

If $\M$ is a Boolean matroid, then complete building sets agree with a notion introduced by Postnikov, Reiner, and Williams \cite{Postnikov2008}.

\begin{example}[Complete built Boolean matroids]\label{ex:boolean-chordal}
    Let $\B_{n}$ be the Boolean matroid on the ground set $[n]$. The Chow polynomial $\H(\B_n,\mathcal{G})(t)$ equals the $h$-polynomial of the nested set complex $\N(\L,\G)$. In turn, the $h$-polynomial of $\N(\L,\G)$ agrees with the $h$-polynomial of a generalized permutohedron \cite{postnikov-generalized-perm} called a \emph{nestohedron}.
    
    In \cite[Section~9]{Postnikov2008}, a building set $\G\subset \B_n$ is \emph{chordal} if for any $I = \{i_1 < \cdots < i _k\} \in \G$,
    \[ \{i_1, \ldots, i_p \} \in \G, \quad 1\leqslant p \leqslant n.\]
    By Remark~\ref{rmk:complete-simple-def} this is equivalent to saying that $(\B_n, \G)$ is $<$-complete for the natural order $<$ on $[n]$. Equivalently, this means that the built matroid $(\B_n,\G)$ is a supersolvable built matroid, witnessed by the maximal chain of modular elements 
    $\emptyset \prec \{1\} \prec \cdots \prec \{1,\ldots,n\}$. Indeed, in the Boolean lattice, every flat is modular. 
\end{example}

Complete built matroids are stable under all standard operations on built matroids. 
\begin{proposition}\label{prop:complete-stability}
    For any $F \in \L(\M)$, any $\G$-compatible modular cut $\Mcut$, and $e \notin E$, the following statements hold for a $\vartriangleleft$-complete built matroid $(\M, \G)$: 
    \begin{enumerate}[\normalfont(i), leftmargin = 20pt, itemsep = 3pt]
        \item \label{item:complete-restriction} The restriction $(\M^{F}, \G^F)$ is a complete built matroid with ground set $F$.
        \item \label{item:complete-contraction} The contraction $(\M_{F}, \G_F)$ is a complete built matroid with ground set $E \setminus F$. 
        \item \label{item:complete-truncation} The truncation $\Tr_{\Mcut}(\M, \G)$ along $\Mcut$ is a complete built matroid on $E$. 
        \item \label{item:complete-extension} The single-element extension $(\M \cup_{\Mcut} e, \G \cup_{\Mcut} e)$ is a complete built matroid on $E \cup e$. 
        \item \label{item:complete-deletion} If $m = \max_{\vartriangleleft} (E)$, the deletion $(\M \setminus m, \G \setminus m)$ is a complete built matroid on $E \setminus m$. 
    \end{enumerate}
\end{proposition}

\begin{proof}
    The first and second statements follow directly from the definition. 
    
    For \ref{item:complete-truncation}, let $\vartriangleleft$ be a total order on $E(\M)$ such that $(\M, \G)$ is $\vartriangleleft$-complete. Let us show that the built matroid $(\Tr_{\Mcut}(\M), \Tr_{\Mcut}(\G))$ is $\vartriangleleft$-complete. Let $F\leqslant G \in \L(\Tr_{\Mcut}(\M))$ be such that $G\in \Tr_{\Mcut}(\G),$ and denote $\{e_1 \vartriangleleft \cdots \vartriangleleft e_k\} = G\setminus F.$ In this proof we will write $\ch^{\Tr_{\Mcut}(\M)}_{\vartriangleleft}(F,G)$ and $\ch^{\M}_{\vartriangleleft}(F,G)$ for the chain $\ch_{\vartriangleleft}(F,G)$ in $\M$ and $\Tr_{\Mcut}(\M)$ respectively. If $G\notin \Mcut,$ we have the equality 
    \[
    \ch^{\Tr_{\Mcut}(\M)}_{\vartriangleleft}(F,G) = \ch^{\M}_{\vartriangleleft}(F,G),
    \] which gives the inclusion $\ch^{\Tr_{\Mcut}(\M)}_{\vartriangleleft}(F,G) \subset \Tr(\G)^G_F$ since in that case $\Tr(\G)_F^G = \G_F^G$. On the other hand if $G\in \Mcut$ and $F\notin \Mcut$, then we have 
    \[ \ch^{\Tr_{\Mcut}(\M)}_{\vartriangleleft}(F,G) = \ch^{\M}_{\vartriangleleft}(F,G) \setminus \{F\vee \sigma(\{e_1, \ldots, e_p\})\},\]
    where $F\vee \sigma(\{e_1, \ldots, e_p\})$ is the biggest element of $\ch^{\M}_{\vartriangleleft}(F,G)$ not in $\Mcut$. This also gives the inclusion $ \ch^{\Tr_{\Mcut}(\M)}_{\vartriangleleft}(F,G) \subset \Tr_{\Mcut}(\G)_F^G.$ Finally, if both $F$ and $G$ belong to $\Mcut$ then we again have 
    \[ \ch^{\Tr_{\Mcut}(\M)}_{\vartriangleleft}(F,G) = \ch^{\M}_{\vartriangleleft}(F,G),\]
    which also gives the inclusion $ \ch^{\Tr_{\Mcut}(\M)}_{\vartriangleleft}(F,G) \subset \Tr_{\Mcut}(\G)_F^G.$ 

    For \ref{item:complete-extension}, we denote by $\vartriangleleft'$ the total order on $E(\M)\cup e$ which is $\vartriangleleft$ on $E(\M)$ and such that $e$ is greater than all the elements of $E(\M)$. If $e$ is not a coloop of $\M\cup_{\Mcut}e$ then, if $\un \in \G$ we have 
    \[ \ch^{\M \cup_{\Mcut} e}_{\vartriangleleft'}(\zero, \un) = \cl_{\M\cup_{\Mcut}e}(\ch_{\vartriangleleft}^{\M}(\zero, \un)),\]
    which implies $\ch^{\M \cup_{\Mcut} e}_{\vartriangleleft'}(\zero, \un)\subset \G\cup_{\Mcut}e.$ Moreover, a proper restriction or contraction of $(\M\cup_{\Mcut} e, \G \cup_{\Mcut}e)$ is either a single-element extension of a proper restriction/contraction of $(\M, \G)$, or a restriction or contraction of $(\Tr_{\Mcut}(\M), \Tr_{\Mcut}(\G))$ so we can conclude by induction and the previous result. 

    For \ref{item:complete-deletion}, let us denote by $\vartriangleleft_m$ the restriction of $\vartriangleleft$ to $E(\M)\setminus m$. If $F\leqslant G \in \L(\M\setminus m)$ with $G\in \G\setminus m$, then by definition of $\G\setminus m$ we have $\cl_{\M}(G) \in \G$ and since $m = \max_{\vartriangleleft} E(\M)$ we have 
    \[\ch^{\M}_{\vartriangleleft}(\cl_{\M}(F),\cl_{\M}(G)) = \cl_{\M}(\ch^{\M\setminus m}_{\vartriangleleft_m}(F,G)),\]
    which proves the inclusion $\ch_{\vartriangleleft_m}^{\M\setminus m}(F,G) \subset (\G\setminus m)_F^G.$ 
\end{proof}

\subsection{Flag built matroids}
  
Recall that a simplicial complex $\Delta$ is said to be a \emph{flag} simplicial complex, if the minimal nonfaces of $\Delta$ have cardinality $2$.\footnote{A vertex set $S \subseteq V(\Delta)$ is a nonface if $S \notin \Delta$. Furthermore, $S$ is a minimal nonface if all proper subsets of $S$ are faces of $\Delta$.} Flag simplicial complexes are preserved under taking restrictions, links, joins, and cones. 
\begin{lemma}\label{lem:restriction-of-flag-is-flag}
If $\Delta$ is a flag simplicial complex, and $F \subset V(\Delta)$, then $\Delta|_{F}$ is flag. 
\end{lemma}

\begin{proof}
Let $G \subset F$ be a set of vertices of $\Delta|_{F}$ such that for all $u,v \in G$, $\{u, v\} \in \Delta|_{F}$. Then the flagness of $\Delta$ implies that $G$ is a simplex of $\Delta,$ and so it is a simplex of $\Delta_{F}$.
\end{proof}

\begin{lemma}\label{lem:link-of-flag-is-flag}
    If $\Delta$ is a flag simplicial complex and $F \in \Delta$, then the link $\lk_{\Delta}(F)$ is flag.
\end{lemma}

\begin{proof}
    Let $G \subset V(\Delta)\setminus F$ be a subset of vertices such that for every pair of vertices $u, v \in G$ we have $\{u, v\} \in \lk_{\Delta}(F)$. This implies that for any pair of vertices $u, v \in F\cup G$ we have $\{u, v\} \in \Delta,$ which shows that $G \in \lk_{\Delta}(F)$.    
\end{proof}

\begin{lemma}\label{lem:join-of-flag-is-flag}
   Two simplicial complexes $\Delta$ and $\Delta'$ are flag if and only if their join $\Delta \ast \Delta'$ is flag.
\end{lemma}

\begin{proof}
If $\Delta \ast \Delta'$ is flag, then $\Delta$ and $\Delta'$ are flag by Lemma \ref{lem:restriction-of-flag-is-flag}. For the converse, assume $\Delta$ and $\Delta'$ are flag. Let $F \subset V(\Delta)$, $F' \subset V(\Delta')$ be two sets of vertices such that for any $v, v' \in F \cup F'$, $\{v,v'\} \in \Delta \ast \Delta'$. Since $\Delta$ and $\Delta'$ are flag, $F \in \Delta$ and $F' \in \Delta'$, which implies that $F \cup F' \in \Delta \ast \Delta'$.
\end{proof}

\begin{lemma}\label{lem:cone-of-flag-is-flag}
    A simplicial complex $\Delta$ is flag if and only if $\Cone(\Delta)$ is flag.
\end{lemma}

\begin{proof}
    This follows from Lemma \ref{lem:join-of-flag-is-flag}.
\end{proof}

\begin{definition}\label{def:flag-built-matroid}
    A building set $\calG$ is \textit{flag}, if the nested set complex $\calN(\calL, \calG)$ is a flag simplicial complex. A built lattice $(\L, \G)$ or a built matroid $(\M, \G)$ is \emph{flag}, if $\G$ is flag. 
\end{definition}

The class of flag built matroids neither contains nor is contained in the class of complete built matroids, as is shown by the following example.

\begin{example}
Let $\B_4$ denote the Boolean matroid on the ground set $E = \{1, 2, 3, 4\}$. If $\G$ is the building set $\{1,2,3,4, 12, 34, 1234\}$, then $(\B_4, \G)$ is flag, but it is not complete because $\G$ does not contain any maximal chain of flats. 

On the other hand, if $\U_{3,4}$ is the uniform matroid of rank $3$ on the ground set $E = \{1,2,3,4\}$, and $\G$ is the building set $\{1,2,3,4, 12,1234\}$, then $(\U_{3,4},\G)$ is $\vartriangleleft$-complete for the order $1\vartriangleleft 2 \vartriangleleft 3 \vartriangleleft 4$, but $(\U_{3,4}, \G)$ is not flag because we have $2\vee 3 \vee 4 = 1234 \in \G$ and $2\vee 3, 2\vee 4, 3\vee 4 \notin \G$.
\end{example}
We have the following characterization of a flag built lattice, by directly reformulating the condition that the minimal nonfaces of $\N(\L, \G)$ have cardinality $2$.

\begin{lemma}
    A built lattice $(\L, \G)$ is flag if and only if, for any antichain $\{G_1, \ldots, G_k\}\subseteq \calG$ with $k \geqslant 2$, the following implication  holds:
    \[
    \text{$G_1 \vee \cdots \vee G_k \in \G$ $\implies$ there exist $i \ne j$ such that $G_i \vee G_j \in \calG$.}
    \]
\end{lemma}

\begin{example}\label{ex:flag-built-lattices}
    Some examples of flag building sets are as follows. 
    \begin{enumerate}[\normalfont(i), leftmargin = 20pt, itemsep = 3pt] 
        \item For any loopless matroid $\M$, the maximal building set $\calG_{\max}$ of $\calL(\M)$ is flag. Indeed, order complexes of posets are flag.
        \item For $n \geqslant 1$ and the Boolean matroid $\U_{n,n}$, the minimal building set $\calG_{\min}$ consisting of all the flats of rank $1$ is flag. 
        \item For $n \geqslant 3$ and the braid matroid $\mathrm{K}_n$, which is the graphic matroid of the complete graph on $n$ vertices, the minimal building set $\calG_{\min}$ of $\mathrm{K}_n$ is flag. 
        \item More generally, by the proof of \cite[Proposition 2.32]{Coron_2025}, any supersolvable built lattice is flag. 
        \item For the Boolean matroid $\U_{4,4}$, the following building set is flag: 
        \[
        \calG \coloneq \{1, 2, 3, 4, 12, 34, 1234\}. 
        \]
    \end{enumerate}
\end{example}

Flag built matroids are preserved under taking restriction, contraction. 
\begin{proposition}\label{prop:res-cont-preserves-flag}
If $(\L, \G)$ is a flag built lattice and $F \in \L \setminus\{\zero, \un\}$, then the restriction $(\L^F, \G^F)$ and the contraction $(\L_F, \G_F)$ are flag. 
\end{proposition}

\begin{proof}
By Lemma \ref{lem:reducible} and Lemma \ref{lem:cone-of-flag-is-flag},  we can assume that $(\L, \G)$ is irreducible. By successive restrictions and contractions, we can further assume $F \in \G$. If $\N(\L,\G)$ is flag, then by Lemma \ref{lem:link-of-flag-is-flag}, the link $\lk_{\N(\L, \G)}(\{F\}) \simeq \N(\L_F, \G_F)\ast\N(\L^F, \G^F)$ is also flag (the isomorphism being given by Proposition \ref{prop:composition-nested-sets}), and so by Lemma \ref{lem:join-of-flag-is-flag}, both $\N(\L_F,\G_F)$ and $\N(\L^F, \G^F)$ are flag. 
\end{proof}

\begin{lemma}\label{lem:flag-deletion}
If $(\M, \G)$ is a simple flag built matroid, and $e \in E$, then $(\M, \G) \setminus e$ is flag. 
\end{lemma}

\begin{proof}
The simplicial complexes $\N(\L\setminus e, \G\setminus e)$ and $\N(\L, \G)|_{\cl_{\M}(\G\setminus e)}$ are isomorphic, via the isomorphism given by sending a vertex $G \in \G\setminus e$ to $\cl_{\M}(G).$ Hence, the result follows from Lemma~\ref{lem:restriction-of-flag-is-flag}.
\end{proof}

As we shall now prove, among nested set complexes of building sets, those arising from flag building sets are distinguished by the property that they can be obtained from the nested set complex of any building subset through successive stellar subdivisions along edges. Equivalently, in the language of building sets, a flag building set can be obtained from any building subset by successively adjoining elements that are joins of exactly two existing elements. This motivates the following definition. 

\begin{definition}
Let $\L$ be a geometric lattice and $\G$ a building set of $\L$. An \emph{augmentation} of $\G$ is a building set $\G' \supset \G$ such that $\G' \setminus \G$ is a singleton $\{G\}$. If $G$ has exactly two factors in $\G\setminus \{G\}$, then we say that $\G'\supset \G$ is a \emph{binary augmentation}. A \emph{filtration of $\G'$ from $\G$} is a sequence of building sets 
\[
\G = \G_p \subset \G_{p-1} \subset \cdots \subset \G_{1} \subset \G_0 = \G'
\] such that for each $i$, $\G_i \subset \G_{i+1}$ is an augmentation. A filtration is called binary if all its augmentations are binary. 
\end{definition}

One has the following classical result. For completeness we provide a short proof.
\begin{proposition}[{\cite[Theorem~4.2]{FM_2005}}]\label{prop:building-sets-augmentation}
Let $\G \subseteq \G'$ be two building sets of the same matroid. There exists a filtration of $\G'$ from $\G.$
\end{proposition}
\begin{proof}
By Lemma \ref{lem:building-set-deletion-of-one-element}, for any minimal element $G$ of  $\G'\setminus \G$ we have that $\G'\setminus G$ is a building set. We can then conclude by induction. 
\end{proof}
The following proposition shows that if the bigger building set is flag, then one can further ask for the filtration to be binary. This generalizes \cite[Theorem~1(1)]{volodin} from the case of Boolean matroids to arbitrary loopless matroids. 
\begin{proposition}\label{lem:flag-binary-filtration}
Let $\G\subset \G'$ be two building sets of the same matroid. If $\G'$ is flag, then there exists a binary filtration of $\G'$ from $\G$. 
\end{proposition}

\begin{proof}
    Let us construct the desired filtration by induction on the cardinality of $\G'\setminus \G$.
    Suppose we have constructed a filtration of $\G'$ of length $k$ 
    \[
     \G_k \subsetneq \G_{k-1} \subsetneq \cdots \subsetneq \G_1 \subsetneq \G_0 = \G'
    \] 
    for some $\G_k \supset \G$. By Lemma~\ref{lem:building-set-deletion-of-one-element}, the set 
    \[
    \mathcal{R}_k = \{G \in \G_k \setminus \G \mid (\G_k \setminus \{G\}) \text{ is a building set of $\L$}\}
    \] is nonempty, as it contains the minimal elements of $\G_k \setminus \G$. Let $G$ be any maximal element of $\mathcal{R}_k$, and set $\G_{k+1} \coloneqq \G_{k} \setminus G$. By construction, this process terminates at some integer $p$ such that $\G_{p} = \G$. It remains to show that the resulting filtration 
    \[
    \G = \G_{p} \subsetneq \G_{p-1} \subsetneq \cdots \subsetneq \G_{1} \subsetneq \G_{0} = \G'
    \]
    is a binary filtration, using the flagness of $\G'$.
    
    For all $0 \leqslant i \leqslant p-1$, let us denote by $F_i$ the unique element in the singleton $\G_i \setminus \G_{i+1}.$ Suppose for contradiction that $q$ is the minimal index such that $\G_{q+1} \subset \G_{q}$ is not a binary augmentation. Then the unique element $G \in \G_{q} \setminus \G_{q+1}$ has at least $3$ factors $G_1, \ldots, G_k$ in $\G_{q+1}$, and we have the structural isomorphism 
    \begin{equation}\label{eq:structural-iso-in-proof-filtration}
    \varphi_G: \prod_{1\leqslant i \leqslant k}[\zero, G_i] \xrightarrow{\sim} [\zero, G].
    \end{equation}
    
    The factors $G_1, \ldots, G_{k}$ form an antichain in $\G'$ of size $\geqslant 3$. By flagness of $\G'$, the set 
    \[ \{0 \leqslant i < q  \, | \, F_i \leqslant G \textrm{ and } F_i \nleqslant G_j \textrm{ for all } 1 \leqslant j \leqslant k\}
    \]
    is nonempty. Let $m$ be the maximum of that set and denote 
    \[
    (F_{m}^1, \ldots, F_{m}^k) = \varphi_G^{-1}(F_{m})
    \]
    the decomposition of $F_{m}$ via the isomorphism \eqref{eq:structural-iso-in-proof-filtration}. For every $1\leqslant i \leqslant k$, if $F_{m}^i \neq \zero$, then $F_{m} \wedge G_i \neq \zero$ which gives $F_{m}\vee G_i \in \G_{m}.$ By maximality of $m$ this implies the equality $F_{m}\vee G_i = G_{m}$. Therefore, there exists some subset $I \subsetneq [k]$ of cardinality at least $2$ such that 
    \[ 
    F_{m} = \bigvee_{i \in I} G_i.
    \]
    By the maximality of $m$, the factors of $G$ in $\G_{m}\setminus \{G\}$ are $G_{m}$ together with all the $G_i$'s with $i \in I$. By Lemma \ref{lem:building-set-deletion-of-one-element} this shows that $\G_{m}\setminus \{G\}$ is a building set. This contradicts the fact that $G_{m}< G$ is a maximal removable flat at step $m$.
\end{proof}

\section{Tropical fans and tropical modifications}
In this section, we recall key definitions in tropical intersection theory, and study tropical modifications in our setup. We refer the reader to \cite{gathmann2009tropical} for tropical intersection theory; for standard toric geometry and toric intersection theory, we use \cite{fulton1993introduction, fulton1997intersection}. 

Let $\mathbf{N}$ be a lattice of finite rank and $\Mlatt$ its dual. We denote by $\langle-,-\rangle$ the pairing
\[
\langle -, - \rangle \colon \Nlatt \times \Mlatt \to \Z, \quad (f, v) \mapsto f(v). 
\]
Let us also denote $\Nlatt_{\R} \coloneqq \Nlatt \otimes \R$ and $\Mlatt_{\R} = \Mlatt\otimes \R.$ For any $f \in \Mlatt_{\R}$, there is an associated hyperplane and halfspace in $\Nlatt_{\R}$ given by 
\[
H_{f} \coloneqq \{v \in V \mid \langle f, v \rangle  = 0\}, \quad H_{f}^{\geqslant 0} \coloneqq \{v \in V \mid \langle f, v\rangle  \geqslant 0\},
\]
A \emph{polyhedral cone} is a finite intersection of halfspaces. A hyperplane or a halfspace is \emph{rational} if $f$ is in the lattice $\Mlatt$. A polyhedral cone is \emph{rational} if it is a finite intersection of rational halfspaces. 

Given a polyhedral cone $\sigma$ in $\Nlatt_{\R}$, let us denote by $\Nlatt_{\sigma}$ the sublattice $\Nlatt\cap \langle \sigma \rangle$. A \emph{face} of $\sigma$ is a set of the form $\sigma \cap H_f$ such that $\sigma \subseteq H_f^{\geqslant 0}$. 

A collection $\Sigma$ of polyhedral cones in $\Nlatt_{\R}$ is a \emph{fan} if 
\begin{enumerate}[leftmargin = 20pt, itemsep = 3pt] 
    \item[(i)] (closed under taking faces) for any $\sigma \in \Sigma$ and $\tau$ a face of $\sigma$, $\tau$ is in $\Sigma$, and  
    \item[(ii)] (closed under intersection) for any $\sigma, \sigma' \in \Sigma$, $\sigma \cap \sigma'$ is a face of both $\sigma$ and $\sigma'$. 
\end{enumerate}
A fan is \emph{rational} if all its cones are rational.

The linearity space $\mathbf{L}(\sigma)$ of a polyhedral cone $\sigma$ is the maximal linear subspace contained in $\sigma$. The dimension of $\sigma$ is the dimension of the linear span $\langle \sigma \rangle$. 
A polyhedral cone $\sigma$ is \emph{unimodular} if there exists a subset $\{v_1, \ldots, v_{k}\}$ of a basis of $\Nlatt$ such that 
\[
\sigma = \operatorname{cone}(v_1,\ldots,v_p) + \langle v_{p+1}, \ldots, v_k\rangle. 
\]
A fan $\Sigma$ is \emph{unimodular} if every cone in $\Sigma$ is unimodular. The linearity space $\mathbf{L}(\Sigma)$ of $\Sigma$ is the maximal linear subspace contained in every cone in $\Sigma$, and its dimension is denoted by $\ell$. By closedness under intersection we have $\mathbf{L}(\Sigma) = \mathbf{L}(\sigma)$ for all $\sigma \subset \Sigma$. The support $\abs{\Sigma}$ of a fan $\Sigma$ is the union of all the cones of $\Sigma$. Throughout this work, we assume all fans are rational and unimodular. 

\subsection{Minkowski weights}
The free abelian group of Minkowski weights of a rational unimodular fan is a tropical analogue of the homology groups of an algebraic variety.  
Let $\Sigma$ be a rational unimodular fan of pure dimension $d$, meaning that all the maximal cones of $\Sigma$ have dimension $d$, and let $\Sigma_k$ denote the set of $k$-dimensional cones in $\Sigma$. For every pair of cones $\tau \subset \sigma$ with $\tau$ of codimension $1$ in $\sigma$, we denote by $H_{\sigma,\tau}$ the open half-space of $\langle\sigma \rangle \setminus \langle\tau \rangle$ containing $\sigma\setminus \tau$. A primitive generator $v$ of the pair $\tau\subset \sigma$ is a vector of $\Nlatt_{\sigma}\cap H_{\sigma,\tau}$ such that we have the isomorphism of abelian groups 
\[ \mathbf{N}_{\sigma} \simeq \mathbf{N}_{\tau} \oplus \Z v.\]
We choose such a collection of primitive generators $\mathbf{e}_{\sigma/\tau}$ for all pairs $\tau \subset \sigma$ with $\tau$ of codimension $1$ in $\sigma$. 

A \emph{$k$-dimensional Minkowski weight} on $\Sigma$ is a function $\omega \colon \Sigma_{k+\ell} \to \Z$ satisfying the \emph{balancing condition}: For every $(k-1+ \ell)$-dimensional cone $\tau$ in $\Sigma$, we have
\[
\sum_{\tau \subset \sigma} \omega(\sigma) \mathbf{e}_{\sigma/\tau} \in \langle\tau \rangle. 
\] 
The balancing condition does not depend on the choice of primitive generators. Let $\MW_{k}(\Sigma)$ be the abelian group of all $k$-dimensional Minkowski weights on $\Sigma$ and denote
\[
\MW_{\bullet}(\Sigma) \coloneqq \bigoplus_{i= 0}^{d} \MW_k(\Sigma). 
\] 
Elements of $\MW_{k}(\Sigma)$ are called \emph{tropical $k$-cycles} of $\Sigma$. We also denote, for all $k \geqslant 0$, 
\[
\MW_k(\Sigma,\Q) \coloneqq \MW_k(\Sigma) \otimes \Q, \quad \text{and } \quad \MW_{\bullet}(\Sigma, \Q) \coloneqq \bigoplus_{k \geqslant 0} \MW_k(\Sigma, \Q).
\]

\subsection{Chow rings}
Let $\e_{\Sigma}$ denote the collection of primitive generators $\e_{\sigma/\tau}$ with $\tau = \mathbf{L}(\Sigma)$. Let $R_{\Sigma}$ be the polynomial ring over $\Z$ with variables indexed by elements of $\e_\Sigma$: 
\[
R_{\Sigma} \coloneqq \Z[x_{\mathbf{e}}]_{\mathbf{e} \in \mathbf{e}_{\Sigma}}. 
\]
The \emph{Chow ring} of $\Sigma$ is the graded commutative ring
\[
\uCH^{\bullet}(\Sigma) \coloneqq R_{\Sigma} / (I_{\Sigma} + J_{\Sigma})
\] where the generators have degree $1$, and $I_{\Sigma}$ and $J_{\Sigma}$ are the ideals given by 
\begin{align}
    I_{\Sigma} &\coloneqq \langle x_{\mathbf{e}_1} \cdots x_{\mathbf{e}_k} \mid \{\mathbf{e}_1, \ldots, \mathbf{e}_k\} \text{ is not in any } \sigma \in \Sigma\rangle, \\
    J_{\Sigma} &\coloneqq \bigg\langle \sum_{\mathbf{e} \in \mathbf{e}_{\Sigma}} \langle \mathbf{e}, m\rangle x_{\mathbf{e}} \mid m \in \Mlatt\bigg\rangle.
\end{align}
It is a classical result that the degree $k$ piece $\uCH^{k}(\Sigma)$ is generated by square-free monomials 
\[x_{\sigma} \coloneqq \prod_{\mathbf{e} \in \mathbf{e}_{\Sigma}, \mathbf{e} \in \sigma} x_{\mathbf{e}},
\] where $\sigma$ ranges over all cones of dimension $k + \ell$  \cite[Proposition 5.5]{adiprasito-huh-katz}. By \cite{brion1996piecewise}, we have an isomorphism $\CH^{\bullet}(\Sigma) \cong \CH^{\bullet}(X(\Sigma))$, where $X(\Sigma)$ is the  toric variety associated to $\Sigma$.  

\begin{thm}[{\cite[Proposition 5.6]{adiprasito-huh-katz}}]
    Let $\Sigma$ be a rational unimodular fan of pure dimension $d$. For $0 \leqslant k \leqslant d$, the map 
    \[
    \MW_{k}(\Sigma) \rightarrow \Hom(\uCH^{k}(\Sigma), \Z), \quad \omega \mapsto \bigg(x_{\sigma} \mapsto \omega(\sigma) \bigg).
    \]
    is a well-defined isomorphism. 
\end{thm}
The isomorphism above induces an analogue of the Kronecker duality pairing in classical algebraic topology. Under this duality, the Chow ring and the group of Minkowski weights enjoy the following properties: 
\begin{enumerate}[(1), leftmargin = 20pt, itemsep = 3pt] 
\item (Cap product) The analogue of Kronecker duality above induces the cap product 
\[
\uCH^{k}(\Sigma) \times \MW_{\bullet}(\Sigma) \to \MW_{\bullet - k}(\Sigma), \quad (\xi, \omega ) \mapsto \xi \cap \omega \coloneqq \left( \sigma \mapsto \omega(\xi \cdot x_{\sigma})\right).
\]  
which endows the graded $\Z$-module $\MW_k(\Sigma)$ with the structure of a graded $\CH^{\bullet}(\Sigma)$-module. 

\item (Degree map) For any Minkowski weight $\omega$ of top dimension, the cap product above induces the degree map 
\[
\deg_{\omega} \colon \uCH^{d}(\Sigma) \to \Z \quad \xi \mapsto \xi \cap \omega. 
\] 

\item (Poincar\'{e} pairing) The degree map induces an analogue of the Poincar\'{e} pairing
\[
\uCH^{k}(\Sigma) \times \MW_{d-k}(\Sigma) \to \Z, \quad (\xi, \omega) \mapsto \langle \xi, \omega \rangle \coloneqq \sum_{\sigma \in \Sigma_{d-k}} \omega(\sigma) \cdot \deg_{\omega}(\xi \cdot x_{\sigma}) \in \Z. 
\]
\end{enumerate}

\subsection{Divisors}
Let $\PL(\Sigma)$ denote the abelian group of piecewise linear integral functions on $\Sigma$, i.e., continuous functions $f \colon \abs{\Sigma} \to \R$ satisfying that, for every $\sigma \in \Sigma$, $f$ restricted to $|\sigma|$ coincides with the linear extension of a function in $\Mlatt$ restricted to $|\sigma|$. Let $\Mlatt_\Sigma$ denote the abelian group of functions on $|\Sigma|$ which are the restriction of an element of $\Mlatt$. The abelian group of divisors on $\Sigma$ is defined as 
\[
\textrm{D}(\Sigma) \coloneqq \PL(\Sigma)/\textrm{L}(\Sigma).
\] Elements of $\textrm{D}(\Sigma)$ are called \emph{tropical Cartier divisors}. 

When $\Sigma$ is complete i.e., $\abs{\Sigma} = \Nlatt_\R$, a Cartier divisor $D$ can be defined as 
\[
\sum_{\mathbf{e} \in \mathbf{e}_{\Sigma}}  c_{\mathbf{e}} x_{\mathbf{e}} \in \uCH^{1}(\Sigma)
\] by setting $D(\mathbf{e}) = c_{\mathbf{e}}$. The vector space of Cartier divisors $\textrm{D}(\Sigma)$ on $\Sigma$ satisfies the following properties: there is an isomorphism of abelian groups 
\[
\textrm{D}(\Sigma) \xrightarrow{\sim} \uCH^{1}(\Sigma), \quad D_{\rho} \mapsto x_{\rho}, 
\]
which induces a pairing given by the Poincar\'{e} pairing of $\textrm{D}(\Sigma)$ on $\MW_{\bullet}(\Sigma)$ 
\[
\textrm{D}(\Sigma) \times \MW_{k}(\Sigma) \to \MW_{k -1}(\Sigma), \quad (D, \omega) \mapsto D \cap \omega. 
\] 
The image of this action is called the \emph{tropical Weil divisor} of $D$ on the support of $\omega$, first introduced in \cite[Definition 3.4]{allermann2010first}. As a Minkowski weight, $D \cap \omega$ has the following concrete description, which is a standard result in tropical intersection theory. 
\begin{proposition}\label{prop:tropical-divisor-as-cap-product}
    Let $\Sigma$ be a complete fan with an $\ell$-dimensional linearity space, let $\omega \in \MW_k(\Sigma)$, and let $D \in \operatorname{D}(\Sigma) \cong \uCH^{1}(\Sigma)$ be a class represented by a function $\varphi$. For every $(k + \ell -1)$-dimensional cone $\tau$, the tropical Weil divisor $D\cap \omega$ has weight 
    \[
    (D \cap \omega)(\tau) = \left(\sum_{\substack{\sigma \in \Sigma_{k + \ell} \\ \sigma \supsetneq \tau}} \omega_{\varphi}(\sigma) \varphi_{\sigma} \left( \mathbf{e}_{\sigma/\tau} \right)\right) - \varphi_{\tau}\left( \sum_{\substack{\sigma \in \Sigma_{k + \ell} \\ \sigma \supsetneq \tau}} \omega_{\varphi}(\sigma) \mathbf{e}_{\sigma/\tau}\right), 
    \] where $\varphi_{\sigma}$ is the linear extension to $\langle \sigma \rangle$ of $\varphi_{|\sigma}$ for every $\sigma$, and $\varphi_{\tau}$ is the linear extension to $\langle \tau \rangle$ of $\varphi_{|\tau}$.
\end{proposition}

\subsection{Nested set fans}
\label{subsec:nested-set-fans}
Let $(\M, \G)$ be a built matroid on ground set $E$. Let $\max \G = \{G_1, \ldots, G_{k}\}$ be the maximal elements of $\G$. Let $\mathbf{N} \coloneqq \Z^E$, and $\Nlatt_{\R}$ be endowed with the standard basis $\{\mathbf{e}_i\}_{i \in E}$.
For any subset $S \subseteq E$, we set 
\[
\mathbf{e}_S \coloneqq \sum_{i \in S} \mathbf{e}_{i} \in \Nlatt_{\R}.
\] 
The equality of matroids 
\[
\M = \bigoplus_{G_i \in \max \G} \M^{G_i}. 
\]
induces the identification 
\[
\Nlatt_{\R} = \prod_{G_i \in \max \G} \R^{G_i} = \R^{G_1} \times \R^{G_2} \times \cdots \times \R^{G_k}, 
\]
together with the \emph{linearity space} in $\Nlatt_{\R}$ defined by 
\[
\mathbf{L}_{\M, \G} \coloneqq \langle \mathbf{e}_{G_1}, \mathbf{e}_{G_2}, \ldots, \mathbf{e}_{G_k}\rangle. 
\] 

For every nested set $\S \in \N(\L, \G)$, we define the cone of $\S$ to be the Minkowski sum of the linearity space and the cone spanned by vectors given by $\S$: 
\[
\sigma_{\S} \coloneqq \mathbf{L}_{\M, \G} + \textrm{cone}\{\mathbf{e}_{S} \mid S \in \S\} \subseteq \R^{E}. 
\] 

\begin{definition}\label{def:nested-set-fan}
The \textit{nested set fan} $\Sigma_{\M, \calG}$ of $(\M, \calG)$ is the fan in $\R^{E}$ consisting of the cones corresponding to nested sets in $\mathcal{N}(\calL, \calG)$: 
\[
    \Sigma_{\M, \calG} \coloneqq \{\sigma_{\calS} \mid \calS \in \mathcal{N}(\M, \calG)\}. 
\]
\end{definition}
The following properties follow fr. 
\begin{proposition}
    The nested set fan $\Sigma_{\M, \G}$ enjoys the following properties: 
    \begin{enumerate}[\normalfont(i), leftmargin = 20pt, itemsep = 3pt] 
        \item $\Sigma_{\M, \G}$ is \textit{unimodular} (\cite[Proposition 2]{feichtner2004chow}). 
        \item $\Sigma_{\M, \G}$ is balanced, that is, the assignment on each maximal cone $\sigma \mapsto 1$ is a Minkowski weight. 
        \item $\Sigma_{\M, \G}$ contains a unique minimal cone $\sigma_{\varnothing} = \mathbf{L}_{\M, \G}$. 
        \item $\Sigma_{\M, \G}$ is pure of dimension $\rk(\M)$, that is, every maximal cone is of equal dimension $\rk(\M)$. 
    \end{enumerate}
\end{proposition}

\begin{remark}
    The crucial difference between the definition above and a typical nested set fan in literature is the existence of linearity spaces coming from each component of the building set, which destroys strong convexity. This enables the fan to be balanced, which is required for tropical intersection theory, while maintaining the isomorphism classes of toric varieties. 
\end{remark}

\begin{example}[Nested set fans of Boolean matroids with minimal building sets]\label{ex:fan-Boolean-min}
Let $n \geqslant 1$ and let $(\U_{n, n}, \G_{\min})$ be the Boolean matroid with the minimal building set. The nested set complex $\N(\U_{n, n}, \G_{\min})$ is described in Example \ref{ex:nested-sets} \ref{item:nested-sets-min}. The only cone of $\Sigma_{\U_{n, n}, \G_{\min}}$ is the linearity space, which is the whole space $\R^E$: 
\[
\mathbf{L}_{\U_{n, n}, \G_{\min}} = \langle \mathbf{e}_i \rangle_{i \in E} = \R^n.   
\]
The associated toric variety $X(\Sigma_{\U_{n, n}, \G_{\min}})$ is isomorphic to a point. 
\end{example}

\subsection{Stellar subdivisions}
Stellar subdivisions on fans induce blowups of toric varieties on the associated toric varieties. 
Due to the existence of linearity spaces in $\Sigma_{\M, \G}$, our definition of stellar subdivisions depends on whether the subdivision center is in the linearity space. 
\begin{definition}
    Let $\Sigma$ be a fan in $\mathbf{N}^E_{\R}$ with linearity space $\mathbf{L}$, and $v \in \abs{\Sigma} \cap \mathbf{N}^{E}$.
    \begin{enumerate}[(a), itemsep = 5pt, leftmargin = 20pt] 
    \item If $v = \sum_{a \in E} c_a \mathbf{e}_{a} \in \mathbf{L}$ with $c_{a} \ne 0$ and we denote the support of $v$ by
    \[
    \textrm{supp}(v) \coloneqq \{a \in E \mid c_a \ne 0\} \subseteq E, 
    \] then the \emph{stellar subdivision (within linearity space)} $\Sigma^{\ast}_{v}$ consists of cones 
    \begin{align*}
        \mathbf{L}^{\ast}_v &\coloneqq \R \langle v \rangle \oplus \left(\mathbf{L} \bigg/\bigoplus_{a \in \supp(v)} \R \langle \mathbf{e}_{a} \rangle\right) \subseteq \mathbf{N}_{\R}^E, \\
    \end{align*}
    together with the cones generated by primitive generators in $\mathbf{L}$
    \[
        \{\sigma_a \coloneqq \mathbf{L}_v^{\ast} + \textrm{cone}(\mathbf{e}_a) \mid a \in \supp(v)\}
    \] and cones in $\Sigma$ properly containing $\mathbf{L}$ 
    \[
        \{\sigma^{\ast} \coloneqq \mathbf{L}^{\ast} + \textrm{image of $\sigma$ under } \R^E \to \R^E/\mathbf{L} \mid \sigma \in \Sigma \}.
    \]
    \item If $v \notin \mathbf{L}$, then the \emph{stellar subdivision} $\Sigma^{\ast}_{v}$ consists of cones 
    \[
        \{\sigma \mid v \notin \sigma \in \Sigma\} \cup \{\textrm{Cone}(\tau, v) \mid v \notin \tau \in \Sigma, \{v\} \cup \tau \subseteq \tau' \in \Sigma\}. 
    \]
    \end{enumerate}
\end{definition}

\begin{example}[Stellar subdivision of $\Sigma_{\U_{n, n}, \G_{\min}}$]
    Let $(\M, \G) = (\U_{n, n}, \G_{\min})$. We perform stellar subdivision of the nested set fan $\Sigma$ at $v = \mathbf{e}_{ij}$ for $i \ne j \in [n]$. 
    Since $v \in \mathbf{L} = \R^n$ with support $\supp(v) = \{i, j\}$, the stellar subdivision $\Sigma^{\ast}_{v}$ consists of cones 
    \begin{align*}
        \mathbf{L}_{\mathbf{e}_{ij}}^{\ast} &= \R\langle \mathbf{e}_{ij}\rangle \oplus \bigg(\R^n \big/ (\R\langle \mathbf{e}_{i} \rangle \oplus \R \langle \mathbf{e}_{j}\rangle ) \bigg) = \R \langle \mathbf{e}_{ij}\rangle \oplus \R^{[n] \setminus \{i, j\}}, \\
        \sigma^{\ast}_{i} &= \R \langle \mathbf{e}_{ij}\rangle \oplus \R^{[n] \setminus \{i, j\}} + \textrm{cone}(\mathbf{e}_{i}), \\
        \sigma^{\ast}_{j} &= \R \langle \mathbf{e}_{ij}\rangle \oplus \R^{[n] \setminus \{i, j\}} + \textrm{cone}(\mathbf{e}_{j}).
    \end{align*}
\end{example}

The following is immediate from construction, and classical facts about stellar subdivision. 
\begin{proposition}\label{prop:stellar-toric-morphisms}
    Let $\pi \colon X(\Sigma_{v}^{\ast}) \to X(\Sigma)$ be the toric morphism induced by fan refinement $\Sigma^{\ast}_{v} \to \Sigma$. 
    \begin{enumerate}[\normalfont(i), leftmargin=15pt]
        \item If $v \in \mathbf{L}$, and let $r = \dim \mathbf{L} - \dim \mathbf{L}_v^{\ast}$, then $\pi \colon X(\Sigma^{\ast}_{v}) \to X(\Sigma)$ is a $\P^r$-bundle of toric varieties. 
        \item If $v \notin \mathbf{L}$, then $\pi \colon X(\Sigma^{\ast}_{v}) \to X(\Sigma)$ is a proper birational toric morphism. 
    \end{enumerate}
\end{proposition}

\subsection{Tropical modifications}
Tropical modification is a tropical analogue of blowing up inµ classical algebraic geometry. We refer to \cite[Section 2.3]{Shaw_2013}, and \cite[Section 5]{amini2023hodgetheorytropicalfans} for more details.
Given a piecewise linear function $f$ on $\Sigma$,  $\Gamma_f$ denotes the \emph{graph} of $f$ on $\abs{\Sigma}$
\[
\Gamma_f \colon \abs{\Sigma} \to V \times \R, \quad v \mapsto (v, f(v)).
\]
If we denote the zero vector in $V$ as $0$, then for every cone $\delta$ in the subfan $\Sigma_f$, the \emph{lift} $\Lambda(\delta)$ of $\delta$ is defined as the Minkowski sum of polyhedral cones
\[
\Lambda_f(\delta) \coloneqq \Gamma_f(\delta) + \R_{\geqslant 0}(0, 1) \subseteq V \times \R. 
\] 
\begin{definition}
Let $\Sigma\subset V$ be a tropical fan with weight $\omega$, with a piecewise affine function $f$ on $\Sigma.$ The \emph{tropical modification} of $\Sigma$ along $f$ is the tropical fan $\widetilde{\Sigma}_f$ in $V \times \R$ given by cones in the graph and the lift of $f$ 
\[
\widetilde{\Sigma}_f \coloneqq \{\Gamma_f(\sigma) \mid \sigma \in \Sigma\} \cup \{\Lambda_f(\sigma) \mid \sigma \in \Sigma_f\},
\] together with the weight $\widetilde{\omega}$ defined as follows: for any cone $\sigma \in \widetilde{\Sigma}$,  
\[
\widetilde{\omega}(\sigma) \coloneqq \begin{cases}
    \omega(\sigma) & \text{ if $\sigma \in \Gamma_f(\Sigma)$}, \\
    \omega_{\Sigma_f}(\sigma) & \text{ if $\sigma \in \Lambda_f(\Sigma_f)$}. 
\end{cases}
\]
\end{definition}
By construction, there is a linear projection of tropical fans 
\[
\pi_{f} \colon \widetilde{\Sigma}_f \twoheadrightarrow \Sigma
\] 
A tropical modification is an isomorphism of tropical fans away from the divisor $\Delta_f$. If $f$ is a linear function on $\Sigma$, then $\Delta_f$ is the empty fan, and the tropical modification is an isomorphism. Therefore, the tropical modification of $\Sigma$ along any tropical Cartier divisor $f \in \mathrm{D}(\Sigma)$ is well-defined.

For nested set fans of built matroids, tropical modifications relate three operations on built matroids: 
If $\Mcut$ is a nonempty, atom-free, $\G$-compatible modular cut of a built matroid $(\M, \G)$, and $e$ not in $E$,  the fans 
\[
\overline{\Sigma} \coloneqq \Sigma_{\Tr_{\Mcut}(\M, \G)}, \quad  \Sigma \coloneqq \Sigma_{\M, \G}, \quad \Sigma_e \coloneqq \Sigma_{\M\cup_{\Mcut_e} e, \G\cup_{\Mcut_e}e}
\] are related via tropical intersection theory, and tropical modification. 

By \cite[Corollary 4.6]{backman2025convexgeometrybuildingsets}, there exists a complete nestohedral fan $\Sigma_{\B} \subset \R^{E}$ associated to the Boolean matroid $\B$ on $E$ such that $\Sigma \hookrightarrow \Sigma'$. The nested set fan $\overline{\Sigma}$ is a tropical Weil divisor on $\Sigma$, if and only if $\overline{\Sigma}$ defines a $(r-1)$-dimensional Minkowski weight on $\Sigma$, if and only if there exists a piecewise linear function $\varphi$ on $\Sigma_{\B}$ such that 
\[
[\varphi] \cap \Sigma = \overline{\Sigma} \in \MW_{r-1}(\Sigma_{\B}).
\]
For $i\in E$, let $\alpha_{\Mcut} = \alpha_{\Mcut,i}$ be the piecewise linear function on $\Sigma_{\B}$ defined by its values on rays
\[
\alpha_{\Mcut, i}(\mathbf{e}_{S}) \coloneqq \begin{cases}
    1 & \text{ if } i \notin S \in \G \cap \Mcut, \\
    -1  & \text{ if } i \in S \in \G \setminus \Mcut, \\
    0 & \text{ otherwise. }
\end{cases} 
\] 
By unimodularity of $\Sigma_{\B}$, this defines a piecewise-linear function on $\Sigma_{\B}$, independent on $i$. We have the following new result, describing the truncation of a built matroid as a tropical Weil divisor of $\Sigma$, extending the classical result \cite[Theorem 3.2.3]{backman2023simplicial} for principal truncation of maximal built matroids. 
\begin{proposition}\label{prop:truncation-alpha-cap-product}
    There is an equality of $(r-1)$-dimensional Minkowski weights
    \[
    \alpha_{\Mcut} \cap \Sigma = \overline{\Sigma}.
    \]
\end{proposition}

\begin{proof}
    Consider a codimension-$1$ cone $\tau_S$ of $\Sigma$, corresponding to a nested set $\S$ of $(\M, \G)$. The nested set $\S$ has a unique local interval $L^G(\S)$ of rank different from $1$, and the rank of $L^G(\S)$ is $2.$ Denote by $F_1, \ldots, F_k$ the flats strictly between $J^G$ and $G$. The maximal nested sets of $(\M, \G)$ containing $\S$ are the nested sets $\S \cup \{F_j/J^G\}$ where $1\leqslant j \leqslant k.$
    
    First consider the case where $G$ belongs to $\Mcut$ and there exists an index $j_0$ such that $F_{j_0}\notin \Mcut$. This is exactly when $\S$ is a maximal nested set of $\Tr(\M, \G).$ In this situation, let $i\in E(\M)$ be any element such that $\cl_{\M}(J^G\cup i) = F_{j_0}.$ By Proposition ~\ref{prop:tropical-divisor-as-cap-product} we have 
    \begin{align*}
    (\alpha_{\Mcut,i} \cap \Sigma)(\tau_{\S}) = \alpha_{\Mcut, i} \left(\sum_{1\leqslant j \leqslant k} \mathbf{e}_{F_j/J^G} \right) - \sum_{1 \leqslant j \leqslant k} \alpha_{\Mcut, i}(\mathbf{e}_{F_j/J^G}).
    \end{align*}
    Since $F_{j_0}/J^G$ is the unique factor of $F_{j_0}$ not below $J^G$, we must have $i \in F_{j_0}/J^G$. Moreover, there is no other index $j\neq j_0$ such that $i \in F_j/J^G$. This gives the equality 
    \[
    - \sum_{1 \leqslant j \leqslant k} \alpha_{\Mcut, i}(\mathbf{e}_{F_j/J^G}) = 1.
    \] 
    On the other hand, notice that the sets of atoms $(F_j/J^G)\setminus J^G$ with $1\leqslant j \leqslant k$ form a partition of the set of atoms $G\setminus J^G$, so one can write 
    \[
    \sum_{1\leqslant  j\leqslant k} \mathbf{e}_{F_j/J^G} = \mathbf{e}_G + \mathbf{v},
    \] 
    where $\mathbf{v}$ is a sum of elements of the form $\mathbf{e}_{H}$ where $H$ is a factor of $J^G.$ For those elements we have $\alpha_{\Mcut,i}(\e_H)= 0,$ and we also have $\alpha_{\Mcut,i}(\e_G) = 0$, which leads to the equality 
    $$ \alpha_{\Mcut, i}\left(\sum_{1\leqslant  j\leqslant k} \mathbf{e}_{F_j/J^G}\right) = 0.$$
    From this we obtain the desired equality 
    $$ (\alpha_{\Mcut,i} \cap \Sigma)(\tau_{\S}) = 1.$$
    Let us now consider the complementary cases, for which we are looking to prove that $(\alpha_{\Mcut,i} \cap \Sigma)(\tau_{\S})$ vanishes. 
    
    First, consider the case where $G$ (the top element of the dimension $2$ local inverval) does not belong to $\Mcut.$ If $i \in E(\M)$ is any element not contained in $G,$ we immediately have 
    $$ \alpha_{\Mcut, i}(\mathbf{e}_{F_j/J^G}) = 0$$
    for all $1\leqslant j \leqslant k,$ as well as 
    $$ \alpha_{\Mcut, i}\left(\sum_{1\leqslant  j\leqslant k} \mathbf{e}_{F_j/J^G}\right) = 0,$$
    by the same argument as in the previous case. 

    Next consider the case where $J^G$ belongs to $\Mcut.$ Note that in that case, by the fact that $\Mcut$ is $\G$-compatible we have that at least one of the factors of $J^G$ belongs to $\Mcut$, and by modularity of $\Mcut$ we have that at most one the factors of $J^G$ belong to $\Mcut.$ Let $i\in E$ be any element which belongs to the unique factor $H_0$ of $J^G$ in $\Mcut$. We see that $\alpha_{\Mcut, i}$ vanishes on $\e_{H_0}$, because $H_0$ contains $i$ and belongs to $\Mcut,$ and $\alpha_{\Mcut, i}$ vanishes on $\e_H$ for every other factor $H$, because those factors do not contain $i$ and are not contained in $\Mcut$. Furthermore, $\alpha_{\Mcut,i}$ vanishes on $\e_G$, so we obtain 
    $$ \alpha_{\Mcut, i}\left(\sum_{1\leqslant  j\leqslant k} \mathbf{e}_{F_j/J^G}\right) = 0.$$
    For all $1\leqslant j \leqslant k,$ we either have that $F_i/J^G$ contains $H_0$, or $\{F_i/J^G, H_0\}$ is a nested antichain. In the first scenario we immediately obtain $\alpha_{\Mcut,i }(e_{F_i/J^G}) = 0$ and in the second one, by modularity of $\Mcut$ we see that $F_i/J^G$ is not in $\Mcut$ and does not contain $i$, which leads to the same conclusion. 

    Finally, the case where $G \in \Mcut$, $J^G\notin \Mcut$ and $F_i\in \Mcut$ for all $1\leqslant i \leqslant k$ is not possible because by modularity of $\Mcut$, having $F_1, F_2 \in \Mcut$ implies $J^G \in \Mcut.$
\end{proof}

\begin{remark}
    When the matroid $\M$ is realizable, the vanishing locus of a generic section of the nef line bundle associated to $\alpha_{\Mcut}$ on the wonderful variety of $W_{\M, \G}$ is the wonderful variety associated to the truncation $W_{\Tr_{\Mcut}(\M,\G)}$ along $\Mcut$. 
\end{remark}

Let $\widetilde{\Sigma}$ be the tropical modification of $\Sigma$ along $\overline{\Sigma}$.  
\begin{proposition}
    There is an isomorphism of tropical fans $\Sigma_e \cong \widetilde{\Sigma}$. 
\end{proposition}
\begin{proof}
For simplicity let us assume that $(\M, \G)$ is irreducible. When viewing the nested set fans $\overline{\Sigma}$ and $\Sigma$ in $\R^E$ instead of $\R^E/\e_E$, the class $\alpha_{\Mcut}$ lifts to the sum $\sum_{G\in \Mcut}x_G$. When performing the tropical modification we end up with a fan living in $\R^E\times \R\e_e$. The rays corresponding to flats $G\in \Mcut$ are lifted in the graph of $\alpha_{\Mcut}$ and become $\e_{G} + \e_e = \e_{\cl_{\M\cup_{\Mcut} e}(G)}$. The rays corresponding to flats $G\notin \Mcut$ are left unchanged. This shows that when viewed in $\R^E\times \R\e_e$, the fans $\Sigma_e$ and $\widetilde{\Sigma}$ have the same rays. The star around $\e_e$ in the tropical modification $\widetilde{\Sigma}$ can be identified with $\overline{\Sigma}$, which is also the case in the nested set fan $\Sigma_e$ by Proposition \ref{lem:built-truncation-equality}. The cones not containing $\e_e$ in the tropical modification $\widetilde{\Sigma}$ are lifts of cones of $\Sigma$ in the graph of $\alpha_{\Mcut}$, that is, they are obtained from those latter cones by applying $\cl_{\M\cup_{\Mcut} e}$ to each indexing flat. This is also the case in $\Sigma_e.$
\end{proof}

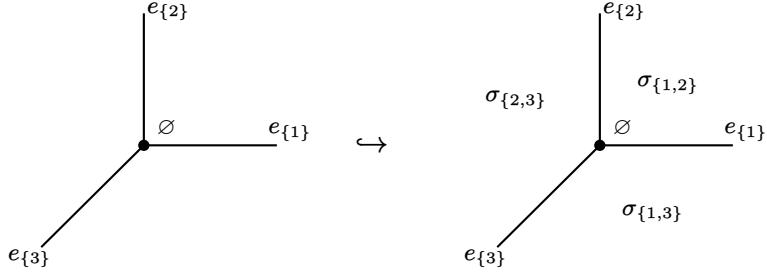
\begin{figure}[h!]
\begin{tikzpicture}[scale=0.5]
\node[circle,fill,inner sep=1.5pt] at (0,0) {};
\draw[thick] (0.125,0) -- (3.5,0);
\draw[thick] (0,0.125) -- (0,3.5);
\draw[thick] (-0.09,-0.09) -- (-2.7,-2.7);
\node at (0.6,0.45) {\small$\varnothing$};
\node at (3.85,0.35) {\small$e_{\{1\}}$};
\node at (0.65,3.55) {\small$e_{\{2\}}$};
\node at (-3,-3) {\small$e_{\{3\}}$};
\node at (6,0) {\large$\hookrightarrow$};
\begin{scope}[shift={(12,0)}]
\node[circle,fill,inner sep=1.5pt] at (0,0) {};
\draw[thick] (0.125,0) -- (3.5,0);
\draw[thick] (0,0.125) -- (0,3.5);
\draw[thick] (-0.09,-0.09) -- (-2.7,-2.7);
\node at (0.6,0.45) {\small$\varnothing$};
\node at (3.85,0.35) {\small$e_{\{1\}}$};
\node at (0.65,3.55) {\small$e_{\{2\}}$};
\node at (-3,-3) {\small$e_{\{3\}}$};
\node at (1.8,1.6) {\small$\sigma_{\{1,2\}}$};
\node at (-2.2,1.2) {\small$\sigma_{\{2,3\}}$};
\node at (1.4,-1.8) {\small$\sigma_{\{1,3\}}$};
\end{scope}
\end{tikzpicture}
\caption{The fan of $(\U_{2, 3}, \G_{\max})$ is a tropical Weil divisor of the fan of $(\U_{3, 3}, \G_{\max})$ in $\R^{[3]}/\R\langle \mathbf{e}_{123} \rangle$ defined  by the piecewise linear function $f = - \min_{i \in [3]}(x_i)$.}
\end{figure}

\subsection{Pullback maps}\label{subsec:pullback-maps}
We now describe several basic pullback homomorphisms of Chow rings given by the operations studied above. 
\begin{enumerate}[(1), leftmargin = 20pt, itemsep = 3pt] 
    \item (Pullback under stellar subdivision within linearity space) Let $X(\Sigma)$ be a smooth toric variety associated with the unimodular fan $\Sigma$ with linearity space $\mathbf{L}$ of dimension $\dim \mathbf{L}$. If $v \in \mathbf{L}$ is the stellar subdivision center and let $r = \dim \mathbf{L} 
    - \dim \mathbf{L}_{v}^{\ast}$, the fan refinement $\Sigma_{\sigma}^{\ast} \to \Sigma$ induces a $\P^r$-bundle by Proposition~\ref{prop:stellar-toric-morphisms}. 
    The pullback induces an isomorphism of Chow groups,  with the class of a hyperplane $x \in \uCH^1(\P^r)$ \cite[Theorem 3.3]{fulton2013intersection}
    \[
    \uCH(\Sigma^{\ast}_v) \cong \bigoplus_{i=0}^{r} x^i \uCH(\Sigma). 
    \]
    As a result, the Chow polynomial admits the decomposition 
    \[
    \uH(\Sigma^{\ast}_{v})(t) = (1 + t + \cdots + t^{r}) \cdot \uH(\Sigma)(t). 
    \]
    In the special case when $r = \dim \mathbf{L} - \dim \mathbf{L}_v^{\ast} =1$, we have isomorphism of Chow groups
    \[
    \uCH(\Sigma^{\ast}_v) \cong \uCH(\Sigma) \oplus x \uCH(\Sigma), 
    \] and the Chow polynomial is multiplied with $(1+t)$ while its $\gamma$-transform remains
    \[
    \uH(\Sigma^{\ast}_{v})(t) = (1+t) \uH(\Sigma)(t), \quad \gamma(\Sigma_v^{\ast})(t) = \gamma(\Sigma)(t). 
    \]
    \item (Pullback under stellar subdivision outside of linearity spaces) Let $X(\Sigma)$ be the smooth toric variety associated with the unimodular fan $\Sigma$ with rational coefficients.
    For a cone $\sigma \in \Sigma$ of dimension $k$, the torus orbit closure corresponding to $\sigma$ in $X(\Sigma)$ is identified with the smooth toric variety associated with the star fan $\textrm{star}_{\sigma}(\Sigma)$. The embedding of 
    \[
    \iota_{\sigma} \colon X(\textrm{star}_{\sigma}(\Sigma))\hookrightarrow X(\Sigma)
    \] is regular, and the restriction map $\iota^{\ast}$ is a surjective ring morphism. 
    Let $\Sigma^{\ast}_{\sigma}$ be the stellar subdivision of $\Sigma$ at $\sigma$, and the subdivision of fans induces a toric morphism:
    \[
    \pi_{\sigma} \colon X(\Sigma^{\ast}_{\sigma}) \to X(\Sigma). 
    \]  
    The pullback induces an isomorphism of groups \cite[Proposition 6.7]{fulton1997intersection} \cite[Theorem 1, Appendix]{keel1992intersection}
    \[
    \uCH^\bullet(\Sigma^{\ast}_{\sigma}) \cong \pi_{\sigma}^{\ast}\uCH^\bullet(\Sigma) \oplus \bigoplus_{i=1}^{k-1} \uCH^\bullet(\textrm{star}_\sigma(\Sigma))[-i].
    \]
    In the special case when $\dim \sigma = 2$, we have the decomposition of graded Chow groups
    \[
    \uCH^\bullet(\Sigma^{\ast}_{\sigma}) \cong \uCH^\bullet(\Sigma) \oplus x_{\sigma } \uCH^\bullet(\textrm{star}_\sigma(\Sigma)) \cong \uCH^\bullet(\Sigma) \oplus \uCH^\bullet(\textrm{star}_\sigma(\Sigma))[-1].
    \]
    As a result, the Chow polynomial admits the decomposition 
    \[
    \uH^\bullet(\Sigma^{\ast}_{\sigma})(t) =  \uH^\bullet(\Sigma)(t) + t \cdot  \uH^\bullet(\textrm{star}_\sigma(\Sigma))(t).
    \]
    \item (Pullback under tropical modification)
    Let $\widetilde{\Sigma}_{f}$ be the tropical modification of $\Sigma$ along a divisor associated to a rational function $f$. There is a injective morphism of fans 
    \[
    \iota_f \colon \Sigma \hookrightarrow \widetilde{\Sigma}_{f},  
    \]
    which induces is a surjective morphism of graded algebras 
    \[
    \iota_f^{\ast} \colon \uCH(\widetilde{\Sigma}_{f}) \twoheadrightarrow \uCH(\Sigma)
    \]
    Under the assumption that every piecewise linear function on the star fan $\textrm{star}_{\sigma}(\Sigma)$ is uniquely determined by its divisor class in $\uCH^{1}(\textrm{star}_{\sigma}(\Sigma))$, that is, the morphism 
    \[
    \textrm{D}(\textrm{star}_{\sigma}(\Sigma)) \to \uCH^{1}(\textrm{star}_{\sigma}(\Sigma)) 
    \] is injective for every $\sigma \in \Sigma$, the pullback $\pi^{\ast}$ is an isomorphism \cite[Theorem 1.11]{amini2023hodgetheorytropicalfans}
    \[
    \iota_f^{\ast} \colon \uCH(\widetilde{\Sigma}_{f}) \cong \uCH(\Sigma). 
    \]
\end{enumerate}

\begin{corollary}\label{lem:proper-ext-chow}
For any loopless built matroid $(\M, \G)$, any $\G$-compatible modular cut $\Mcut$, and any $e \notin E$, there is an isomorphism of rings 
 \[ \uCH(\M , \G) \cong \uCH(\M\cup_{\Mcut} e, \G\cup_{\Mcut} e).\]
\end{corollary}

\begin{example}[Boolean matroid $\B_3$ from $\G_{\min}$ to $\G_{\max}$]
    For the minimal building set $\G_{\min}$, we have the following binary filtration from $\G_{\min}$ to the flag building set $\G_{\max}$
    \begin{align*}
    \underbrace{\{1, 2, 3\}}_{\mathcal{G}_0} \subseteq 
    \underbrace{\{1, 2, 3, 12\}}_{\mathcal{G}_1} \subseteq 
    \underbrace{\{1, 2, 3, 12, 123\}}_{\mathcal{G}_2} \subseteq 
    \underbrace{\{1, 2, 3, 12, 123, 13\}}_{\mathcal{G}_3} \subseteq 
    \underbrace{\{1, 2, 3, 123, 13, 23\}}_{\mathcal{G}_4}.
\end{align*}
    The Chow group admits a decomposition given by the iterated $\P^1$-bundles followed by iterated blowups at codimension-$2$ subvarieties, starting with 
    \begin{align*}
        \uCH(\U_{3, 3}, \G_{\min}) = \Q, 
    \end{align*}
    with Chow polynomial $\uH(\U_{3, 3}, \G_{\min}) = 1$. 
    
    For $\G_{1}$, stellar subdivision at $\mathbf{e}_{12}$ gives the Chow group decomposition:
    \begin{align*}
        \uCH(\U_{3, 3}, \G_{1}) \cong \uCH(\U_{3, 3}, \G_{0}) \oplus \uCH(\textrm{star}_{12}\Sigma_{\U_{3, 3}, \G_0})[-1] = \Q \oplus \Q[-1], 
    \end{align*}
    with Chow polynomial $\uH(\U_{3, 3}, \G_1) = 1 + t$. 
    
    For $\G_2$, stellar subdivision at $\mathbf{e}_{123}$ gives the Chow group decomposition: 
    \begin{align*}
        \uCH(\U_{3, 3}, \G_2) &\cong \uCH(\U_{3, 3}, \G_1) \oplus \uCH(\textrm{star}_{123}\Sigma_{\U_{3, 3}, \G_1})[-1] \\
        &= \left(\Q \oplus \Q[-1]\right) \oplus \big(\Q \oplus \Q[-1]\big)[-1], 
    \end{align*}
    with Chow polynomial $\uH(\U_{3, 3}, \G_{\min} \cup \{12, 123\}) = 1 + 2t + t^2$. 

    For $\G_3$ and $\G_4$, stellar subdivision at $\mathbf{e}_{13}, \mathbf{e}_{23}$ gives the Chow group decomposition: 
    \begin{align*}
        \uCH(\U_{3, 3}, \G_{\max}) &\cong \left(\Q \oplus \Q[-1]\right) \oplus\left( \Q \oplus \Q[-1]\right)[-1] \oplus \left(\Q[-1] \right) \oplus \left( \Q[-1]\right)
    \end{align*}
    with Chow polynomial $\uH(\U_{3, 3}, \G_{\max})(t) = 1 + 4t + t^2$. 
\end{example}

\begin{lemma}\label{lem:trop-modif-gamma-pos}
Let $\Sigma \subset \Sigma'$ be a pair of tropical rational fans such that 
    \begin{itemize}
        \item $\Sigma$ is a tropical divisor of $\Sigma'$ corresponding to a rational function $f$ on $\Sigma'$, 
        \item The toric varieties $X(\Sigma)$ and $X(\Sigma')$ have $\gamma$-positive Hilbert--Poincar\'{e} series, and 
        \item for every cone in $\sigma \in \Sigma$, $\sigma' \in \Sigma'$, the toric subvarieties of the star fans $X(\textrm{star}_{\sigma}\Sigma)$ and $X(\textrm{star}_{\sigma'} \Sigma')$ have $\gamma$-positive Hilbert--Poincar\'{e} series. 
    \end{itemize}
    Then the toric variety of the tropical modification $X(\widetilde{\Sigma}_f)$ has $\gamma$-positive Hilbert--Poincar\'{e} series and so does every star fan of $X(\widetilde{\Sigma}_f)$. 
\end{lemma}

\begin{proof}
    By $\uCH(\widetilde{\Sigma'}_f) \cong \uCH(\Sigma')$, the Chow ring of $\widetilde{\Sigma}_f$ has $\gamma$-positive Poincar\'{e} polynomial. 

    Let $\sigma \in \widetilde{\Sigma}'_f$. There are $2$ cases: 
    \begin{enumerate}[(i), itemsep = 5pt, leftmargin = 20pt]
        \item If $\sigma \in \{\Gamma_f(\sigma') \mid \sigma' \in \Sigma'\}$, then by \cite[Proposition 5.3]{amini2023hodgetheorytropicalfans}, the star fan $\textrm{star}_{\sigma}(\widetilde{\Sigma}'_f)$ is identified with the tropicalization of $\textrm{star}_{\sigma'}(\Sigma')$ along the divisor $\textrm{star}_{\sigma'}(\Sigma)$.
        \item If $\sigma \in \{\Lambda_f(\tau) \mid \tau \in \Sigma\}$, then inducting on $\dim \sigma$, we may reduce to the following two cases 
        \begin{itemize}
            \item $\sigma = \Lambda_f(\tau)$ for $\tau \in \Sigma$ and $\textrm{star}_{\sigma}(\widetilde{\Sigma}'_f) = \textrm{star}_{\tau}\Sigma$, or 
            \item $\sigma = \R_{\geqslant 0}(0, 1)$ and $\textrm{star}_{\sigma}(\widetilde{\Sigma}'_f) = \Sigma$. 
        \end{itemize} 
    \end{enumerate}
    In any case, the Chow ring of $\textrm{star}_{\sigma}(\widetilde{\Sigma}'_f)$ has $\gamma$-positive Hilbert--Poincar\'{e} series. 
\end{proof}

\section{Chow polynomials of complete built matroids}\label{sec:complete-built-matroids}
In this section we introduce a new class of built matroids called complete built matroids, which strictly contains matroids with maximal building sets. We present several different proofs that the Chow polynomials of complete built matroids are $\gamma$-positive. The first proof is the most elementary, and is centered around the existence of filtrations of building sets given by Proposition \ref{prop:building-sets-augmentation}. The second proof is the simplest one and relies solely on the relative Lefschetz decomposition proved by Pagaria and Pezzoli \cite{pagaria2023hodge}. The third proof is more combinatorial and leads to an explicit formula for the $\gamma$-polynomial associated with the Chow polynomial of complete built matroids (Theorem~\ref{thm:complete-formula}). The formula appearing in that theorem generalizes results by Stump \cite[Theorem~1.1]{stump} and Ferroni--Matherne--Vecchi for maximal building sets \cite[Theorem~4.25]{ferroni-matherne-vecchi}. We want to emphasize that the framework needed for the third proof also yields the instruments to prove Theorem~\ref{thm:main-complete-implies-nevo-petersen+balanced} (see Section~\ref{sec:gamma-is-f-vector} below), and their various corollaries, via the construction of a suitable simplicial complex.

\begin{notation}
We will henceforth make use of the following notation. 
\begin{align*}
    \Flag &\coloneqq \{\Sigma(\M, \G) \mid (\M, \G) \text{ is a flag built matroid}\}, \\
    \Gpos &\coloneqq \{\Sigma \textrm{ tropical fan} \mid \text{the Hilbert--Poincar\'{e} series of $X(\Sigma)$ is $\gamma$-positive}\},\\
    \Gpos^{\ast} &\coloneqq \{ \Sigma \text{ tropical fan}\mid \text{the Hilbert--Poincar\'{e} series of $X(\textrm{star}_{\sigma}\Sigma)$ is $\gamma$-positive, for all } \sigma \in \Sigma\}.
\end{align*}
In particular, $\Gpos^{\ast} \subseteq \Gpos$. 
\end{notation}

\subsection{Gamma positivity for complete built matroids}

We are now ready to give the first and the second proof of the main theorem of this section. 

\begin{thm}\label{thm:complete-gamma-pos}
For any complete built matroid $(\M, \G)$, the Chow polynomial $\H(\M, \G)(t)$ is $\gamma$-positive. 
\end{thm}

\begin{proof}[First proof of Theorem \ref{thm:complete-gamma-pos}]
Let $(\M,\G)$ be a $\vartriangleleft$-complete built matroid for some total order $\vartriangleleft$ on $E(\M)$, and denote $e\coloneqq \max_{\vartriangleleft}E(\M).$ We will prove that $\Sigma_{\M,\G}$ belongs to $\Gposs$ by induction on the cardinality of $\G$ and $E(\M)$. By Proposition \ref{prop:complete-stability} (v), the deletion $(\M\setminus e, \G\setminus e)$ is complete, and so by Proposition \ref{prop:complete-stability} the built matroid $(\M, \G\setminus S_e)$ is complete, where  
\[
S_e \coloneqq \{ G\in \G \, | \, e \textrm{ is a coloop in } G \textrm{ and } G \neq \{e\}\}.
\] 
If $S_e$ is nonempty, our inductive hypothesis implies that $\Sigma_{\M, \G\setminus S_e} \in \Gposs$. By \cite[Theorem 4.2]{FM_2005} there exists a filtration
\[
\G\setminus S_e = \G_1\subset \cdots \subset \G_p = \G.
\] 
For every $1 \leqslant i \leqslant p-1$, let $F$ be the unique element of $\G_{i+1}\setminus \G_i$, and denote $F = \{e_1\vartriangleleft \cdots \vartriangleleft e_k\vartriangleleft e\}$. Since $\G$ is $\vartriangleleft$-complete, we have that $F\setminus e = \sigma(\{e_1, \ldots, e_k\})$ belongs to $\G$. This means that the factors of $F$ in $\G_i$ are $F\setminus e$ and $e$, and $\G_i\subset \G_{i+1}$ is a binary augmentation. Section~\ref{subsec:pullback-maps} shows that $\Sigma_{\M, \G}$ belongs to $\Gposs.$
\end{proof}

\begin{proof}[Second proof of Theorem \ref{thm:complete-gamma-pos}]
Let $(\M,\G)$ be a $\vartriangleleft$-complete built matroid, for some total order $\vartriangleleft$ on $E(\M).$ We will prove that $\Sigma_{\M, \G}$ belongs to $\Gposs$ by induction on the cardinality of $E(\M)$, and the rank of $\M$. By induction and Proposition \ref{prop:complete-stability} it is enough to prove that the Chow polynomial $\H(\M, \G)(t)$ is $\gamma$-positive. Moreover, we can assume that $(\M, \G)$ is irreducible, since a product of $\gamma$-positive polynomials is $\gamma$-positive. 

Let us denote $e\coloneqq \max_{\vartriangleleft}E(\M).$ By \cite[Theorem 5.4]{pagaria2023hodge} we have the formula
\begin{multline}\label{eq:formula-deletion}
    \H(\M, \G)(t) = \H(\M \setminus e, \G\setminus e)(t) + \\\sum_{F\in S_e}(t+\cdots + t^{n_F})\H(\M^{F\setminus e}, \G^{F\setminus e})(t)\H(\M_F, \G_F)(t),
\end{multline}
where $n_F$ denotes the number of factors of $F\setminus e$ in $\G$. For all $F \in S_e$, if we denote $F = \{e_1 \vartriangleleft \cdots \vartriangleleft e_k\vartriangleleft e\}$, we see that $\sigma(\{e_1, \ldots, e_k\}) = F\setminus e$ and so $F\setminus e$ belongs to $\G$ by completeness. This shows that we have $n_F= 1$ for all $F\in S_e.$ Let us now distinguish between two cases. 

If $e$ is a coloop in $\M$, then equation \eqref{eq:formula-deletion} can be rewritten as 
$$ \H(\M, \G)(t) = (1+t)\H(\M\setminus e, \G\setminus e)(t) + \sum_{F\in S_e\setminus \un}t\cdot\H(\M^{F\setminus e}, \G^{F\setminus e})(t)\H(\M_F, \G_F)(t).$$
By our induction hypothesis and Proposition \ref{prop:complete-stability}, all the summands on the right are $\gamma$-positive and moreover they have the same center of symmetry, which proves that $\H(\M, \G)(t)$ is $\gamma$-positive. 

If $e$ is not a coloop in $\M$, then equation \eqref{eq:formula-deletion} can be rewritten as
$$ \H(\M, \G)(t) = \H(\M\setminus e, \G\setminus e)(t) + \sum_{F\in S_e\setminus \un}t\cdot\H(\M^{F\setminus e}, \G^{F\setminus e})(t)\H(\M_F, \G_F)(t),$$
and we arrive at the same conclusion.
\end{proof}

\begin{remark}
    In the special case of maximal building sets, the second proof above is essentially the same as the one appearing in \cite[Theorem~3.25]{ferroni-matherne-stevens-vecchi}. That proof relies on the semi-small decomposition of Chow rings with maximal building sets \cite{braden-huh-matherne-proudfoot-wang}.
\end{remark}

\subsection{A formula for the $\gamma$-polynomial of complete built matroids}

In this subsection we aim to prove Theorem~\ref{thm:main-complete-des-formula}, which gives a combinatorial formula for the $\gamma$-vector of complete built matroids. This will constitute the third and last proof of the $\gamma$-positivity of the foregoing Chow polynomials. Having this formula will lead us to a proof of the analogue of the Nevo--Petersen conjecture for Chow polynomials of complete built matroids in the next section. 

We start by recalling the notion of a descent of a nested set, introduced in \cite[Section 3.2]{coron2026structuralpropertiesnestedset}.
Throughout, we assume that $(\M, \G)$ is an irreducible built matroid with a total order $\vartriangleleft$ on $E(\M)$. Let $\S \in c\N(\L, \G)$ be a nested set of $(\M,\G)$. 

For any $G\in \S$, if the local interval $L^G(\S)$ has rank $1$, then we set 
\[
\lambda_{\S}(G) \coloneqq \min_{\vartriangleleft} \left\{ e \in E(\M) \, | \, J^G \vee e = G \right\}.
\]
If $\S$ is a maximal nested set, then by \cite[Lemma 2.27]{coron2026structuralpropertiesnestedset} all the local intervals of $\S$ have rank $1$ and so $\lambda_{\S}(G)$ is defined for every $G\in \S$. 

It is a classical fact that any nested set $\S$ of $(\M, \G)$ is a tree in the order graph of $\L(\M)$, meaning that for every element $G\in \S$, the set $\S_{>G}$ is either empty or has a unique minimum (see for instance \cite[Lemma~2.19]{coron2026structuralpropertiesnestedset}). In the latter case,  we denote by $p(G)$ this unique minimum ($p$ standing for ``parent''). An element $G'$ with parent $p(G')=G$ is a \emph{child} of $G$; two elements sharing the same parent are \emph{siblings}.

\begin{definition}
    An element $G$ in a maximal nested set $\S$ of $(\M, \G)$ is a \emph{descent} of $\S$ with respect to $\vartriangleleft$, if $G$ is not maximal in $\S$, and we have the strict inequality
    \[ \lambda_{\S}(G) \vartriangleright \lambda_{\S}(p(G)).\]
    A descent $G$ of $\S$ is a \emph{bottom descent}, if $G \in \min \S$;  a descent $G$ of $\S$ is a \emph{double descent}, if $G$ is not minimal in $\S$ and all children of $G$ are descents of $\S$.
\end{definition}

\begin{notation}\label{notation:max-stable}
    For any maximal nested set $\S$, we denote 
    \[\Des(\S) \coloneqq \{ G\in \S \,| \, G \textrm{ descent of } \S\} \qquad \text{ and } \qquad  \des(\S) \coloneqq |\Des(\S)|.\]
    Let $\Nest(\M,\G)^{\max}_{\stable}$ be the set of maximal nested sets of $(\M, \G)$ with no double descents and no bottom descents. 
\end{notation}

\begin{example}\label{ex:descent-computation}
Let $\B_4$ be the Boolean matroid of rank $4$ on ground set $E=\{1,2,3,4\}$, and let $\G$ be the building set $\{1,2,3,4, 12,34, 1234\}$. The built matroid $(\B_4, \G)$ has $8$ maximal nested sets drawn in Figure~\ref{fig:nested-sets-EL-labeling}, with labels given by considering the natural order on $\{1,2,3,4\}$. The stable maximal nested sets in $\Nest(\B_4,\G)^{\max}_{\stable}$ are $\S_1$ and $\S_5$.  
\begin{center}
\begin{figure}
    \includegraphics[scale=0.75]{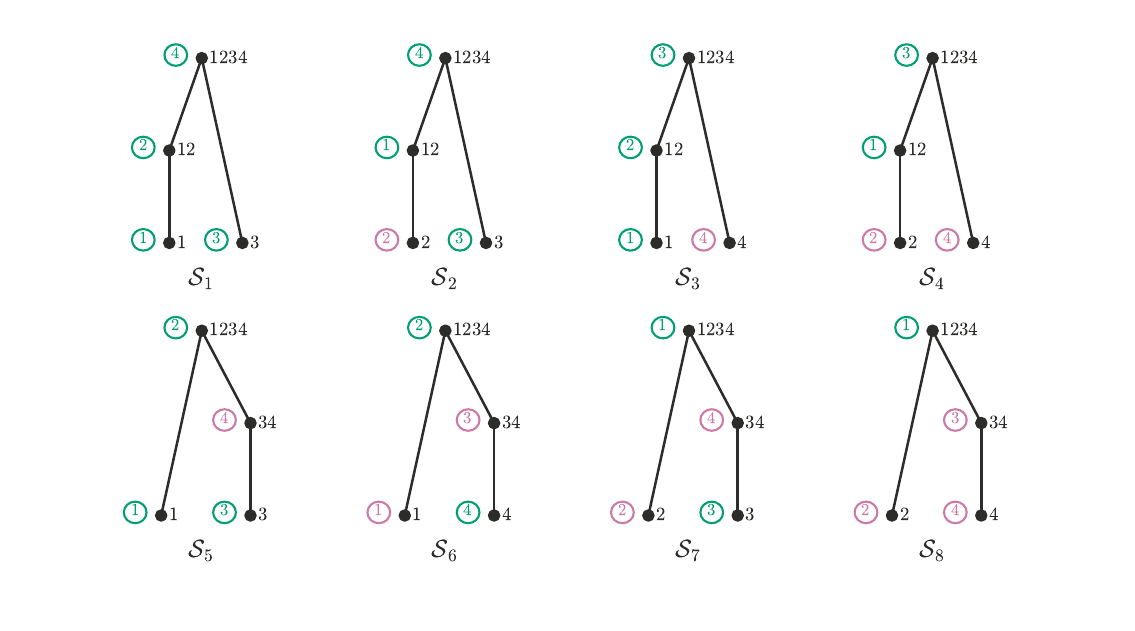}
    \caption{Maximal nested sets of $\B_4$ with $\G = \{1, 2, 3, 4, 12, 34, 1234\}$ with labels in blueish green, and descent labels in pink.}
    \label{fig:nested-sets-EL-labeling}
\end{figure}
\end{center}
\end{example}

\begin{example}
    For any $\vartriangleleft$-complete built matroid $(\M, \G)$, the maximal nested set $\ch_{\vartriangleleft}(\zero, \un)$ has no descents. 
\end{example}

\begin{remark}\label{rmk:nodd-rank1}
    Note that for a maximal nested set $\S$, having no double descent and no bottom descent is equivalent to $\Des(\S)$ having no rank $1$ local interval. 
\end{remark}

We need to prepare the proof of Theorem~\ref{thm:main-complete-des-formula} with additional background.

\begin{definition}
    The \emph{Feichtner--Yuzvinsky monomials} $\FY(\M, \G)$ are the set   
    \[ 
    \{m = x_{F_1}^{\alpha_1}\cdots x_{F_k}^{\alpha_k} \mid \{F_1, \ldots, F_k\} \in \N(\M, \G), 0 < \alpha_i < \rk(F_i) - \rk(J^{F_i}) \text{ for } 1 \leqslant i \leqslant k\} 
    \]
    The \emph{support} of $m$ is the nested set $\S$. 
\end{definition}
In the rest of this paper we will abbreviate $x_{F_1}^{\alpha_1}\cdots x_{F_k}^{\alpha_k}$ by $x_{\S}^\alpha$.

\begin{remark}\label{rmk:fy-monomial-no-rk1-local-interval}
It follows from the definition that the support of a Feichtner--Yuzvinsky monomial does not have any local interval of rank $1$.
\end{remark}

In \cite{feichtner2004chow}, the authors proved that $\CH(\M,\G)$ admits a Gröbner basis whose associated normal monomials are exactly the Feichtner--Yuzvinsky monomials of $(\M, \G).$ In particular, this gives the formula
$$ \H(\M, \G)(t) = \sum_{m \in \FY(\M, \G)} t^{\deg(m)}.$$
The proof strategy for Theorem \ref{thm:main-complete-des-formula} is to find a map $\Psi$ from $\FY(\L,\G)$ to $\Nest(\M,\G)^{\max}_{\stable}$ such that the fibers of $\Psi$ have a binomial distribution, meaning that their generating series are of the form $t^k(1+t)^{n-2k}$. The map $\Psi$ will be constructed using the composition of nested sets defined by Proposition~\ref{prop:composition-nested-sets}. First, the following lemma shows that the composition of nested sets behaves well with respect to the labels $\lambda_{\S}$ defined earlier.
\begin{lemma}[{\cite[Lemma 2.24]{coron2026structuralpropertiesnestedset}}]\label{lem:composition-descent}
Let $(\M,\G)$ be an irreducible built matroid, and let $\vartriangleleft$ be a total order on $E(\M)$. In the situation of Proposition~\ref{prop:composition-nested-sets} \ref{it:comp-nested-sets-3}, if the local interval $L^{G'}(\S_G)$ has rank $1$, then the composition of nested sets preserves the labels, meaning that we have the equality of labels 
\[
\lambda_{\S_G}(G') = \lambda_{\S\circ(\S_G)_G}(G'/J^G).\] 
\end{lemma}
\begin{definition}
Let $\S \in \N(\L, \G)$ be a nested set. The \emph{completion of $\S$} (with respect to $\vartriangleleft$) is the nested set 
\[
\widehat{\S} \coloneqq \S\circ (\ch_{\vartriangleleft}(J^G,G))_{G\in \S\cup\un}.
\] 
\end{definition}

By Proposition~\ref{prop:composition-nested-sets} \ref{it:comp-nested-sets-3} the local intervals of $\widehat{\S}$ all have rank $1$ and so by \cite[Lemma~2.27]{coron2026structuralpropertiesnestedset} the completion $\widehat{\S}$ is a maximal nested set.

We are ready to prove Theorem~\ref{thm:main-complete-des-formula}, which for convenience we also restate below.
\begin{thm}\label{thm:complete-formula}
    Let $(\M,\G)$ be a complete built matroid, and let $\gamma(\M, \G)(t)$ be the $\gamma$-expansion of the Chow polynomial of $(\M, \G)$. We have 
    \[ \gamma(\M, \G)(t) = \sum_{\S \in \Nest(\M,\G)^{\max}_{\stable}} t^{\des(\S)}.\]
    In particular, the polynomial $\gamma(\M,\G)(t)$ has positive coefficients. 
\end{thm}

\begin{proof}
We define a set map 
\[
\Psi:\FY(\M,\G) \rightarrow \Nest^{\max}(\M, \G), \quad m = \left(x_{\S}^{\alpha_{\S}} \right) \cdot \left(x_{\un}^{\alpha_{\un}}\right) \mapsto \widehat{\S}. 
\]

\medskip
\textbf{Claim 1.} The map $\Psi$ is well-defined. 
\smallskip

\emph{Proof of Claim 1.} Let $m$ be a Feichtner--Yuzvinsky monomial with support $\S$. By the definition of completion of a nested set and Lemma \ref{lem:composition-descent}, the descent set of $\widehat{\S}$ is a subset of 
\[
\S\cup \bigcup_{G\in \S\cup \un}\Des(\ch_{\vartriangleleft}(J^G,G))/J^G. 
\] However, for every $G \in \S \cup \un$, the descent set of $\ch_{\vartriangleleft}(J^G,G)$ is $\varnothing$. Therefore, $\Des(\widehat{\S}) \subset \S$. By Remark \ref{rmk:fy-monomial-no-rk1-local-interval} and Remark \ref{rmk:nodd-rank1}, we have $\widehat{\S} \subset \Nest^{\max}_{\stable}(\M, \G).$ 

\medskip

\textbf{Claim 2.} For every $\S \in \Nest^{\max}_{\stable}(\M, \G)$, we have the equality
\begin{equation}\label{eq:fiber-above-S}
\sum_{m \in \Psi^{-1}(\S)} t^{\deg(m)}= t^{\des(\S)}(1+t)^{n-1-2\des(\S)}.
\end{equation}
\smallskip
The general idea behind this claim is that the support of a monomial in $\Psi^{-1}(\S)$ will be of the form $\S\circ (\S_G)_G$ where each $\S_G$ is a subset of $\ch_{\vartriangleleft}(J^G,G)$. The question is then to check that for each $G\in\S$, the multidegrees on $\ch_{\vartriangleleft}(J^G,G)$ satisfying the defining condition for Feichtner-Yuzvinsky monomials, are in bijection with subsets of $\ch_{\vartriangleleft}(J^G,G)$ (thus giving the claimed binomial distribution). We start the proof below by this simple check. \\

\smallskip 
\emph{Proof of Claim 2.} 
For any finite totally ordered poset $(P,<)$ (e.g. $\ch_{\vartriangleleft}(J^G, G)$ for any $G\in \S$), one can associate to any subset $I \subseteq P$ a multidegree $\mult_I:P \rightarrow \Z_{\geq 0}$ defined by the formula 
\begin{equation*}
    \mult_I(p) \coloneqq \begin{cases}
        0 \textrm{ if } p+1 \in I, \\
        \max \{i \geq 0 \mid p-i+1, \ldots, p \in \S\} \textrm{ otherwise.}
    \end{cases}
\end{equation*}
The map $\mult$ gives a bijection between the subsets of $I$ and the multidegrees $m$ on $P$ satisfying the Feichtner-Yuzvinsky condition 
\begin{equation}\label{eq:FY-condition-chain}
    \forall p \in P, \quad m(p) < \rk(p) - \rk(\max_{<}(\supp(m)_{<p})),
\end{equation}
where $\supp(m)_{<p}$ denotes the set of elements in the support of $m$ which are strictly less than $p$. By convention if $\supp(m)_{<p}$ is empty we set $\rk(\max_{<}(\supp(m)_{<p})) \coloneqq -1.$ Moreover, the map $\mult$ sends the cardinality statistics to the total degree statistics. 

For any $G$ in $\S\cup \un$, let us denote 
\begin{align*}
    &\ch_{\vartriangleleft}^{>0}(J^G,G) \coloneqq \ch_{\vartriangleleft}(J^G,G)\setminus J^G, \\
    &\ch_{\vartriangleleft}^{>1}(J^G,G) \coloneqq \{F \in \ch_{\vartriangleleft}(J^G,G) \mid \rk(F) - \rk(J^G) > 1\}, \\
    &\ch_{\vartriangleleft}^{>1,\neq \textrm{top}}(J^G,G) \coloneqq \ch_{\vartriangleleft}^{>1}(J^G,G)\setminus G. 
\end{align*}
Recall from the definition of the composition of nested sets (Proposition \ref{prop:composition-nested-sets}) that there is a bijection between $\widehat{\S} = \S\circ(\ch_{\vartriangleleft}(J^G,G))_{G\in \S\cup \un}$ and $\bigsqcup_{G\in \S\cup \un}\ch_{\vartriangleleft}^{>0}(J^G,G)$, which we denote in this proof by $\varphi$. Let us now consider the set 
\[
U(\S) \coloneqq \ch_{\vartriangleleft}^{>1}(J^{\un},\un)\sqcup \bigsqcup_{G\in \S} \ch_{\vartriangleleft}^{>1, \neq\textrm{top}}(J^G,G), 
\]
and $\calP(U(\S))$ the power set of $U(\S)$. We prove the desired equality of generating functions \eqref{eq:fiber-above-S} by describing a bijection between $\calP(U(\S))$ and the preimage $\Psi^{-1}(\S)$
\[
\Psi_{\S} \colon \mathcal{P}\left(U(\S)\right) \xrightarrow{\sim } \Psi^{-1}(\S). 
\]
This bijection will satisfy $\deg(\Psi_{\S}(I)) = \des(\S) + \abs{I}$, for every $I \subset U(\S)$. By counting, we have
\[
\abs{U(\S)} = n-1-2 \: \des(\S),
\]
which will conclude the proof. For any $I = \sqcup_{G\in \S\cup\{\un\}}I_G \subset U(\S)$ where $I_G\in \ch_{\vartriangleleft}^{>1, \neq\textrm{top}}(J^G,G)$ for all $G\in \S$, and $I_{\un} \in \ch_{\vartriangleleft}^{>1}(J^{\un},\un)$, we set 
\[ \Psi_{\S}(I) \coloneqq x_{\widehat{S}}^{\mult_{I\cup\S}\circ \varphi},\]
where $\mult_{I\cup \S}$ is the multidegree on $\bigsqcup_{G\in \S\cup \un}\ch_{\vartriangleleft}^{>0}(J^G,G)$ defined as the disjoint union
\[ 
    \mult_{I\cup \S} \coloneqq \bigsqcup_{G\in \S\cup \{\un\}}\mult_{(I\cup \S)\cap \ch_{\vartriangleleft}^{>0}(J^G,G)}.
\]
In the above formula, the mult maps in the right hand side are taken with respect to the totally ordered sets $\ch_{\vartriangleleft}^{>0}(J^G,G)$.

See Example~\ref{ex:bijection-fiber} below for a complete computation of $\Psi_{\S}$. 

Since a multidegree given by a mult map satisfies the Feichtner-Yuzvinsky condition \eqref{eq:FY-condition-chain}, we see that $\Psi_{\S}$ lands in $\FY(\M,\G)$. Moreover, for all $I \subset U(\S)$ we see that the support of $\Psi_{\S}(I)$ contains $\S$ and is contained in $\widehat{S}$, so the completion of this support is $\widehat{\S}$. This shows that $\Psi_{\S}$ lands in $\Psi^{-1}(\S).$ We now show that $\Psi_{\S}$ is bijective. 

The map $\Psi_{\S}$ is injective by the injectivity of $\Des(\S)\circ -$ and the injectivity of the mult maps. By surjectivity of the mult maps, any Feichtner--Yuzvinsky monomial of the form $x_{\S'}^\alpha$ where $\S'$ can be written as $\Des(\S) \circ (\S'_G)_G$ (with each $\S'_G$ a subset of $\ch_{\vartriangleleft}(J^G,G)$), is in the image of $\Psi_{\S}$. For any $m$ in $\Phi^{-1}(\S),$ we have $m = x_{\S'}^{\alpha}$ for some nested set $\S'$ satisfying the equality $\widehat{\S'} = \S.$ By Lemma \ref{lem:composition-descent} we have $\Des(\S) = \Des(\widehat{\S'}) \subset \S'.$ By Proposition~\ref{prop:composition-nested-sets} \ref{it:comp-nested-sets-4} this implies that $\S'$ can be uniquely written as a composition $\Des(\S)\circ (\S'_G)_G$ where the $\S'_G$'s are some nested sets in the local intervals $L^{G}_{\Des(\S)}$. By associativity of the composition of nested sets (Proposition~\ref{prop:composition-nested-sets} \ref{it:comp-nested-sets-5}) we have the equality $$\widehat{\S'} = \Des(\S)\circ (\widehat{\S'_G})_G.$$ The equality $\widehat{\S'} = \S$ now gives $$\Des(\S) \circ (\widehat{\S'_G})_G= \Des(\S)\circ (\ch_{\vartriangleleft}(J^G,G))_G $$ and so by injectivity of $\Des(\S)\circ -$ we get the equality $$\widehat{\S'_G} = \ch_{\vartriangleleft}(J^G,G)$$ for all $G \in \Des(\S)\cup \un.$ This implies that we have $\S'_G \subset \ch_{\vartriangleleft}(J^G,G)$ for all $G\in \Des(\S)\cup \{\un\}$, which proves the surjectivity.
\end{proof}

\begin{remark}
    The formula in Theorem~\ref{thm:complete-formula} does not hold if one removes the completeness assumption, even if $(\M, \G)$ is flag. For instance, in Example~\ref{ex:descent-computation}, the built matroid $(\B_4, \G)$ is flag. One can compute $\H(\B_4, \G)(t) = (1+t)^3$ which gives $\gamma(\B_4,\G)(t) = 1.$ However, in that example we saw that $\Nest(\B_4,\G)^{\max}_{\stable}$ contains $\S_5$ which has descent number $1$. 
\end{remark}

\begin{example}\label{ex:bijection-fiber}
Let $\B_8$ denote the Boolean matroid over $\{1,\ldots,8\}$ with $\G_{\max}$ its maximal building set. The built matroid $(\B_8, \G_{\max})$ is complete with respect to any total order on $\{1,\ldots,8\}$, in particular also with respect to the natural order $<$. The nested sets of $(\B_8, \G_{\max})$ are the chains in $\B_8$ and the composition of nested sets is the concatenation of chains. Consider the maximal nested set 
\[\S = \{6,67,678,3678,34678,345678,1345678,12345678\}.\]
We have $\Des(\S) = \{678, 345678\}$ which has three local intervals, two of them having rank $3$ and the top one having rank $2$. Let us compute the bijection 
\[\Psi_{\S}: \mathcal{P}([2,\rk \un - \rk J^{\un}]\sqcup [2,\rk 678 - \rk \zero - 1]\sqcup [2,\rk 345678 - \rk 678 -1])\rightarrow \Psi^{-1}(\S).\]
The power set $\mathcal{P}([2,\rk \un - \rk J^{\un}]\sqcup [2,\rk 678 - \rk \zero - 1]\sqcup [2,\rk 345678 - \rk 678 -1])$ is identified with $\{\emptyset, 2\}\times \{\emptyset, 2\}\times \{\emptyset, 2\}.$
We have 
\[\Psi_{\S}(\emptyset, \emptyset, \emptyset) = x^{\alpha}_{(678, 345678)\circ (\{678\}, \{345678\})} = x_{678,345678}^{\alpha},\]
where $\alpha(678)$ is the height of $3$ in $\{3\}$ which is $1$, and $\alpha(345678)$ is the height of $3$ in $\{3\}$ which is $1.$
We have 
\[\Psi_{\S}(2, \emptyset, \emptyset) = x^{\alpha}_{(678, 345678)\circ (\{678\}, \{345678\}, \un )} = x_{678,345678, \un}^{\alpha},\]
where $\alpha(678) = 1, \alpha(345678) = 1$, and $\alpha(\un)$ is the height of $2$ in $[2]$ which is $1.$
We have 
\[\Psi_{\S}(\emptyset, 2, \emptyset) = x^{\alpha}_{(678, 345678)\circ (\{678\}, \{345678\})} = x_{678,345678}^{\alpha},\]
where $\alpha(345678)=1$, and $\alpha(345)$ is the height of $3$ in $\{2,3\}$ which is $2$. Similarly,
\begin{align*}
    \Psi_{\S}(\emptyset, \emptyset, 2) &= x_{678}x_{345678}^2, \\ 
    \Psi_{\S}(2, 2, \emptyset) &= x_{678}^2x_{345678}x_{\un}, \\ 
    \Psi_{\S}(2, \emptyset, 2) &= x_{678}x_{345678}^2x_{\un}, \\ 
    \Psi_{\S}(\emptyset, 2, 2) &= x_{678}^2x_{345678}^2, \\ 
    \Psi_{\S}(2, 2, 2) &= x_{678}^2x_{345678}^2x_{\un}.  
\end{align*}
\end{example}

\begin{remark}
    Recall from Example \ref{ex:boolean-chordal} that chordal building sets, in the sense of \cite{Postnikov2008}, are complete. In \cite[Theorem~11.6]{Postnikov2008} the authors give a combinatorial formula for the $\gamma$-transform of Chow polynomials of chordal building sets. It is unclear to the authors how \cite[Theorem 11.6]{Postnikov2008} and Theorem \ref{thm:complete-formula} (restricted to Boolean matroids) relate to each other. 

\end{remark}

\subsection{The Nevo--Petersen conjecture for complete built matroids}\label{sec:gamma-is-f-vector}
In this section we use the formula given in Theorem \ref{thm:complete-formula} to prove that the $\gamma$-transform of the Chow polynomial of a complete built matroid is the $f$-vector of a simplicial complex.

Let $(\M, \G)$ be a complete built matroid with respect to some total order $\vartriangleleft$ on $E(\M)$. Denote $V \coloneqq \G \setminus (\ch_{\vartriangleleft}(\zero,\un) \cup \At(\L))$ and define the \emph{$\Gamma$-complex} of the pair $(\M,\G)$ as the collection of sets:
\begin{equation} \label{eq:def-of-Gamma-complex}
\Gamma_{\vartriangleleft}(\M,\G) \coloneqq \{ \Des(\S) \, | \, \S \in \N^{\max}_{\stable}(\L,\G) \} \subset \mathcal{P}(V).
\end{equation}

We can now prove the analogue of Conjecture~\ref{conj:nevo-petersen} for Chow polynomials of complete built matroids, hence verifying Theorem~\ref{thm:main-complete-implies-nevo-petersen+balanced}.

\begin{thm}\label{thm:gamma-is-simplicial-complex}
    Let $(\M,\G)$ be a complete built matroid with respect to a total order $\vartriangleleft$. The sets in $\Gamma_{\vartriangleleft}(\M,\G)$ are the faces of a simplicial complex. Moreover, the $f$-vector of $\Gamma$ equals the $\gamma$-vector of the Chow polynomial $\H(\M,\G)(t)$.
\end{thm}

\begin{proof}
The claim on the $f$-vector follows automatically from Theorem~\ref{thm:complete-formula}, so what remains is proving that $\Gamma$ is a simplicial complex.  

For simplicity, we can assume that $(\M, \G)$ is irreducible. We denote $\Gamma \coloneqq \Gamma_{\vartriangleleft}(\M,\G)$. Let $\mathcal{D} = \{G_1, \ldots, G_k\}$ be the descent set of some stable maximal nested set $\S$. For any $i \in [k]$, let us prove that the descent set of the completion $\widehat{\mathcal{D}\setminus G_i}$ is $\mathcal{D}\setminus G_i$, which will show that $\mathcal{D}\setminus G_i$ belongs to $\Gamma$ and thus conclude the proof. 

By the proof of Claim 1 in the proof of Theorem \ref{thm:complete-formula}, we have the inclusion $\Des(\widehat{\mathcal{D}\setminus G_i})\subset \mathcal{D} \setminus G_i.$ Moreover, it is clear that every element of $\mathcal{D} \setminus G_i$ which is neither a sibling, nor a child nor the parent of $G_i$ belongs to $\Des(\widehat{\mathcal{D}\setminus G_i}),$ since the labels above and below those elements have not changed when passing from $\S$ to $\widehat{\mathcal{D}\setminus G_i}$. 

Let us first prove that if $G_i$ is not maximal in $\mathcal{D}$, then the parent $p(G_i)$ of $G_i$ belongs to $\Des(\widehat{\mathcal{D}\setminus G_i}).$ Assume by contradiction that the contrary is true. If we denote $F = \bigvee (\mathcal{D}\setminus\{G_i, p(G_i)\})_{< p(p(G_i))}$, this exactly means that $F\vee p(G_i)$ belongs to $\ch_{\vartriangleleft}(F, p(p(G_i)))$. However, by Lemma \ref{lem:first-chain-restriction-contraction}\ref{item:contraction} we have the equality of chains 
\[ \ch_{\vartriangleleft}(F\vee G_i, p(p(G_i))) = G_i\vee \ch_{\vartriangleleft}(F, p(p(G_i))),\]
which implies that $F\vee p(G_i)$ belongs to $\ch_{\vartriangleleft}(F\vee G_i, p(p(G_i)))$, which contradicts the fact that $p(G_i)$ is a descent of $\S$. 

Let us now prove that the siblings of $G_i$ are descents of $\Des(\widehat{\mathcal{D}\setminus G_i})$. Let $G$ be a sibling of $G_i$, and assume by contradiction that $G$ is not a descent. In that case, if we denote $F = \bigvee (\mathcal{D}\setminus \{G_i, G\})_{<p(G_i)}$, this means that $F\vee G$ belongs to $\ch_{\vartriangleleft}(F, p(G_i))$. By Lemma \ref{lem:first-chain-restriction-contraction}\ref{item:contraction} this implies that $F\vee G \vee G_i$ belongs to $\ch_{\vartriangleleft}(F\vee G_i, p(G_i))$, which contradicts the fact that $G$ is a descent of $\S.$

Finally, let us prove that the children of $G_i$ are descents of $\Des(\widehat{\mathcal{D}\setminus G_i})$. Let $G$ be a child of $G_i$, and assume by contradiction that $G$ is not a descent. In that case, if we denote $F = \bigvee (\mathcal{D}\setminus \{G_i, G\})_{<p(G_i)}$, this means that $F\vee G$ belongs to $\ch_{\vartriangleleft}(F, p(G_i))$. By Lemma \ref{lem:first-chain-restriction-contraction}\ref{item:restriction} this implies that $F\vee G$ belongs to $\ch_{\vartriangleleft}(F, F\vee G_i)$, which contradicts the fact that $G$ is a descent of $\S$. 
\end{proof}

In the special case of $\G = \G_{\max}$, we can deduce the following statement, which directly implies Theorem~\ref{thm:main-complete-implies-nevo-petersen+balanced}. Recall that a pure simplicial complex $\Delta$ of dimension $d$ is \emph{balanced} if there exists a coloring of the vertices of $\Delta$ with $d+1$ colors, such that no simplex of $\Delta$ has two vertices of the same color. 

\begin{thm}\label{thm:gamma-maximal-is-balanced}
    Let $\M$ be a matroid and $\G = \G_{\max}$. Fix any total ordering $\vartriangleleft$ on the ground set of $\M$. The $\Gamma$-complex $\Gamma_{\vartriangleleft}(\M,\G)$ is balanced. 
\end{thm}

\begin{proof}
If $\G$ is the maximal building set then one can check by induction (using the same recursion as in the second proof of Theorem \ref{thm:complete-gamma-pos}) that $\gamma(\M, \G_{\max})(t)$ has maximal possible degree which is $\lfloor \frac{\rk \M - 1}{2} \rfloor.$ If we color every vertex $F$ of $\Gamma$ by $\lfloor \frac{\rk(F)}{2}\rfloor$ then the number of colors we are using is $\lfloor \frac{\rk \M - 1}{2} \rfloor$. Two vertices $F_1,F_2$ of same color cannot both belong to the descent set of the same stable maximal nested set, which proves that $\Gamma$ is balanced. 
\end{proof}

\begin{example}\label{ex:failure-flagness}
    Let us focus on the case in which $\mathcal{G}=\mathcal{G}_{\max}$ is the maximal building set. Since any bipartite graph is triangle-free, then any $1$-dimensional balanced complex is flag. This means that the $\Gamma$-complex of $(\M,\G)$ is flag when $\G=\G^{\max}$ and $\rk(\L) \leqslant 6$. However, the complex $\Gamma$ constructed above may fail to be flag when $\rk(\L)\geqslant 7$.
    If $\L$ is the Boolean lattice over $7$ elements $1,2,3,4,5,6,7$ with the natural order, and $\G$ is the maximal building set, then $\{67, 234567\}$ is a simplex of $\Gamma$ because this set is the descent set of the maximal chain 
    \[ 6\leqslant 67 \leqslant 267\leqslant 2367 \leqslant 23467\leqslant 234567 \leqslant 1234567.\]
    Likewise, the set $\{2367, 234567\}$ is a simplex of $\Gamma$ because this is the descent set of the maximal chain 
    \[ 2\leqslant 23 \leqslant 236 \leqslant 2367 \leqslant 23467 \leqslant 234567 \leqslant 1234567.\]
    Finally, the set $\{67, 2367\}$ is a simplex of $\Gamma$ because this is the descent set of the maximal chain 
    \[ 6 \leqslant 67 \leqslant 267 \leqslant 2367 \leqslant 12367 \leqslant 123467 \leqslant 1234567.\]
    However, the set $\{67, 2367, 234567\}$ is not a simplex of $\Gamma$ because $2367$ belongs to $\ch_{\vartriangleleft}(67, 234567)$.
\end{example}

\section{Chow polynomials of flag built matroids}\label{sec:chow-flag}
In this section, we analyze the Chow ring of a flag built matroid, and prove that its Hilbert-Poincar\'{e} polynomial admits a $\gamma$-expansion with positive coefficients. The main result of this section is the following.

\begin{thm}\label{thm:flag-gamma-pos}
    If $(\M, \G)$ is a flag built matroid, then its Chow polynomial $\H(\M, \G)(t)$ is $\gamma$-positive. 
    In terms of tropical fans, 
    \[
    \Flag \, \subset \, \Gposs \subseteq \, \Gpos.
    \] 
\end{thm}

\begin{proof}[{Proof of Theorem~\ref{thm:flag-gamma-pos}}]
Let $(\M, \G)$ be a flag built matroid. We will prove that $\Sigma_{\M, \G}$ belongs to $\Gposs$ by induction on the cardinality of $E(\M)$ and the cardinality of $\G$. There are $2$ cases depending on whether $\M$ is a Boolean matroid. 

If $\M$ is the Boolean matroid $\U_{n, n}$ on $n$ elements, then we induct on the size of the building set $\G$. The base case is when $\G = \G_{\min}$. Then explicit computation in Example~\ref{ex:fan-Boolean-min} shows that the Chow polynomial is $\H(\U_{n, n}, \G_{\min})(t) = 1$, which admits the $\gamma$-expansion $\gamma(\H(\U_{n, n}, \G_{\min}); t) = 1$.  
When $\G_{\min} \subsetneq \G$, Proposition~\ref{lem:flag-binary-filtration} implies that there exists a binary filtration from $\G_{\min}$ to $\G$. By Section~\ref{subsec:pullback-maps}, this binary filtration induces a sequence of iterated $\P^1$-bundles and iterated blowups at codimension-$2$ torus orbit closures on the toric variety $X(\Sigma_{\U_{n, n}, \G_{\min}})$ to yield the toric variety $X(\Sigma_{\U_{n, n}, \G})$. Therefore, the resultant tropical fan 
\[
\Sigma_{\M, \G} \in \Gpos^{\ast} \subseteq \Gpos.
\] 

If $\M$ is not a Boolean matroid, then there exists an element $e \in E$ of $\M$, which is not a coloop of $\M$. We induct on the cardinality of $E$ and $\G$ by deleting a non-coloop of $\M$. Since $(\M, \G)$ is a flag built matroid, then so is the single-element deletion $(\M, \G) \setminus e$ by Lemma \ref{lem:flag-deletion} and the contraction $\Sigma_{(\M, G)/e}$ by Lemma~\ref{lem:restriction-of-flag-is-flag}. 
Let our inductive hypothesis posit that 
\[
\Sigma_{(\M, \G)\setminus e}, \Sigma_{(\M, G)/e} \in \Gpos^{\ast} \subseteq \Gpos. 
\] 
We outline our strategy for proving that $\Sigma_{\M, \G} \in \Gpos^{\ast}$ using this hypothesis as follows. At the end of each step, we obtain a tropical fan in $\Gpos^{\ast}$ that is closer to $\Sigma_{\M, \G}$: 
\begin{enumerate}[label=\arabic*, leftmargin=36pt, itemsep=3pt]
    \item[Step 1:] \label{item:flag-step-1} $\Sigma_{(\M, \G) \setminus e}$ contains $\Sigma_{\M/e, \Tr_{\Mcut_e}(\G \setminus e)}$ as a tropical Weil divisor, which we refine via an iterated sequence of stellar subdivisions until we obtain the subfan $\Sigma_{\M/e, \G/e}$. 
    \item[Step 2:] We extend the same refinements to the ambient fan $\Sigma_{(\M, \G)\setminus e}$, producing a refined ambient fan $\Sigma^{\ast}_{(\M, \G) \setminus e}$, which contains the subfan $\Sigma_{(\M, \G)/e}$ as a tropical Weil divisor. 
    \item[Step 3:] We perform a tropical modification of the tropical fan $\Sigma^{\ast}_{(\M, \G) \setminus e}$ along the tropical Weil divisor $\Sigma_{(\M, \G)/e}$. 
    \item[Step 4:] We obtain $\Sigma_{\M, \G}$ via a further sequence of iterated stellar subdivisions on the tropical modification obtained after Step 3. 
\end{enumerate}
In summary, we have the following diagram of tropical fans, where all the surjections are given by iterated stellar subdivisions unless otherwise specified: 
\[
\begin{tikzcd}[column sep=large]
    \Sigma_{(\M, \G)/e} \arrow[r, hook] \arrow[d, two heads] & \Sigma^{\ast}_{(\M, \G)\setminus e} \arrow[d, two heads] & &  \widetilde{\Sigma}_{(\M, \G) \setminus e} \arrow[ll, two heads, "\textrm{tropical modification}"'] & \Sigma_{\M, \G} \arrow[l, two heads] \\
    \Sigma_{\M / e, \Tr_{\Mcut_e}(\G\setminus e)} \arrow[r, hook] & \Sigma_{(\M, \G)\setminus e} & &  & & 
\end{tikzcd}
\]

For Step 1:  Since $\G$ is a flag building set on $\M$, and $\G \setminus S_e \subsetneq \G$ is a building set on $\M$ by Proposition~\ref{prop:built-del-ext}, Lemma \ref{lem:flag-binary-filtration} produces a binary filtration of building sets on $\M$,
\begin{equation}\label{eq:M-filtration}
    \G\setminus S_e = \G_p \subsetneq \cdots \subsetneq \G_0 = \G
\end{equation}
where $S_e = \{ G\in \G \, | \, e \textrm{ is a coloop in } G \textrm{ and } G \neq \{e\}\}$. By Proposition~\ref{prop:truncation-alpha-cap-product}, the tropical fan $\Sigma_{\M/e, \Tr_{\Mcut_e}(\G \setminus e)}$ is a tropical Weil divisor of $\Sigma_{(\M, \G)\setminus e}$. Furthermore, on the matroid $\M/e$, the building set $\G/e$ is a flag building set, and $\Tr_{\Mcut_e}(\G \setminus e) = (\G \setminus S_e)/e \subsetneq \G /e$ is a building set on $\M/e$. The binary filtration in ~\eqref{eq:M-filtration} induces the binary filtration on $\M/e$: 
\begin{equation}\label{eq:M/e-filtration}
\Tr_{\Mcut_e}(\G \setminus e) = (\G\setminus S_e)/e = \G_p/e \subsetneq \cdots \subsetneq \G_0 /e = \G/e
\end{equation}
This binary filtration on $\M/e$ induces a sequence of stellar subdivisions on the tropical fan $\Sigma_{\M/e, \Tr_{\Mcut_e}(\G \setminus e)}$, to obtain the fan 
\[
\Sigma_{(\M, G)/e} \in \Gpos^{\ast}.
\]

For Step 2: We extend the same refinements to the ambient fan $\Sigma^{\ast}_{(\M, \G) \setminus e}$. This refinement is an iterated sequence of stellar subdivisions on the tropical fan $\Sigma_{(\M, \G) \setminus e}$. By Section~\ref{subsec:pullback-maps}, 
\[
\Sigma^{\ast}_{(\M, \G) \setminus e} \in \Gpos^{\ast}. 
\]
Furthermore, the fan $\Sigma^{\ast}_{(\M, \G) \setminus e}$ contains the subfan $\Sigma_{(\M, \G)/e}$ as a tropical Weil divisor by construction. Crucially, the tropical fan $\Sigma^{\ast}_{(\M, \G) \setminus e}$ is not a nested set fan of some built matroid. 

For Step 3: We perform tropical modification of the tropical fan $\Sigma^{\ast}_{(\M, \G) \setminus e}$ along $\Sigma_{(\M, \G)/e}$, and denote the resultant tropical fan $\widetilde{\Sigma}_{(\M, \G) \setminus e}$. Lemma~\ref{lem:trop-modif-gamma-pos} implies that 
\[
\widetilde{\Sigma}_{(\M, \G) \setminus e} \in \Gpos^{\ast}. 
\] 

For Step 4: The binary filtration in ~\eqref{eq:M-filtration} implies 
\[
\G \setminus S_e \subsetneq \G_{p} \subsetneq \cdots \subsetneq  \G_0  =\G 
\] 
induces an iterated sequence of stellar subdivisions on $\widetilde{\Sigma}_{(\M, \G) \setminus e}$, and the resultant fan is $\Sigma_{\M, \G}$. Therefore, the associated toric variety $X(\Sigma_{\M, \G})$ is an iterated blowup on an iterated $\P^1$-bundle of the toric variety $X(\widetilde{\Sigma}_{(\M, \G) \setminus e})$. By Section~\ref{subsec:pullback-maps}, its Poincar\'{e} polynomial $\H(\Sigma_{\M, \G})$ admits a $\gamma$-expansion with positive coefficients, and $\Sigma_{\M, \G} \in \Gpos^{\ast} \subseteq \Gpos$, as desired.  
\end{proof}

\section{Application I: the Poincar\'e polynomial of \texorpdfstring{$\Moduli_{0,n}$}{barM0n}} 

\noindent For all $n \geqslant 2$ we denote by $\Braid_n$ the $n$-th complex braid arrangement 
$$ \Braid_n \coloneqq \{ \{z_i = z_j\} \subset \C^n, 1\leqslant i< j \leqslant n\}.$$
The matroid associated to $\Braid_n$ is the simple matroid whose lattice of flats is the $n$-th partition lattice denoted $\Pi_n$. The elements of $\Pi_n$ are the set partitions of $[n]$, which are ordered by refinement. 

For all $n\geqslant 3$ let us denote by $\Moduli_{0,n}$ the Deligne--Mumford moduli space of stable rational curves of genus $0$ with $n$ marked points. It is a classical fact that for all $n\geqslant 2$ the algebraic variety $\Moduli_{0,n+1}$ is isomorphic to the wonderful compactification $\overline{Y}_{\Braid_n, \G_{\min}}$ of the arrangement complement of $\Braid_n$ with respect to the minimal building set $\G_{\min}$ of $\Pi_n$ (see \cite[Section 4]{Kapranov1993Chow} for instance). In particular, we have the isomorphisms of commutative algebras 
$$ \mathrm{CH}(\overline{Y}_{\Braid_n, \G_{\min}}; \Q) \simeq \mathrm{CH}(\Moduli_{0,n+1}; \Q) \simeq \CH(\Pi_n, \G_{\min}),$$
which give the equalities of polynomials 
$$ P_{\overline{Y}_{\Braid_n, \G_{\min}}}(t) = P_{\Moduli_{0,n+1}}(t) = \H(\Pi_n, \G_{\min})(t).$$ 

The lattice $\Pi_n$ is known to be supersolvable \cite[Example 2.6]{stanley1972supersolvable} with a maximal chain of modular flats witnessing this supersolvability given by the partitions 
\begin{equation}\label{eq:maximal-chain-partition}
\zero \prec |12|3|\cdots \prec |123|4|\cdots \prec \cdots \prec |1\cdots n| = \un.
\end{equation}
By \cite[Proposition 3.11]{Coron_2025} the pair $(\Pi_n, \G_{\min})$, with $\G_{\min}$ the minimal building set of $\Pi_n$, is a supersolvable built matroid (see Example \ref{ex:supersolvable-built-matroid}) and so it is complete. Hence, the results of Section \ref{sec:complete-built-matroids} apply to the study of the Poincaré polynomial $P_{\Moduli_{0,n+1}}(t)$ of $\Moduli_{0,n+1}$. Our Theorem \ref{thm:complete-gamma-pos} implies that $P_{\Moduli_{0,n+1}}(t)$ is $\gamma$-positive, which recovers a result of Aluffi, Chen and Marcolli \cite[Theorem 1.2]{aluffi-chen-marcolli}. Furthermore, Theorem \ref{thm:complete-formula} can be applied to give a combinatorial model for the entries of the $\gamma$-transform of $P_{\Moduli_{0,n+1}}(t)$. Let us make this model explicit. 

The minimal building set $\G_{\min}$ of $\Pi_n$ consists of the set of partitions of $[n]$ having exactly one non-singleton equivalence class. The nested set complex $\mathcal{N}(\Pi_n, \G_{\min})$ has been studied extensively and in particular in \cite{feichtner} it is shown to be isomorphic to the complex of trees $T_n$ (see \cite{Boardman1971HomotopyTrees} and \cite{trappmann-ziegler} for instance). The maximal simplices of $T_n$ are given by combinatorial types of binary rooted trees with leaves labeled from $1$ to $n.$ If $\mathcal{T}$ is one such tree, for all vertices or leaves $v \in \T$ we denote by $\Leaf(v)$ the set of values of the leaves descending from $v$. We decorate any vertex $v \in \T$ by the value 
$$ \ell(v) \coloneqq \max (\min \Leaf(v_1), \min \Leaf(v_2))$$
where $v_1,v_2$ are the two children of $v$. We say that a vertex $v \in \T$ is a descent if $v$ has a parent $p(v)$ (i.e. $v$ is not the root of $\T$), and we have the strict inequality 
$$ \ell(v) > \ell(p(v)).$$
We also say that $v$ is a bottom descent if $v$ is a descent and the two children of $v$ are leaves, and we say that $v$ is a double descent if $v$ is a descent and both children of $v$ are descents as well. We denote by $T_{n,\stable}^{\max}$ the set of all maximal simplices of $T_n$ with no double descent and no bottom descent. Theorem \ref{thm:complete-formula} has the following corollary. 

\begin{corollary}\label{cor:braid-min}
For all $n\geqslant 2$, we have the equality of polynomials
\[ \gamma(P_{\Moduli_{0,n+1}};t) = \sum_{\T \in T_{n,\stable}^{\max}}t^{\des(\T)}.\]
Furthermore, the coefficients of $\gamma(P_{\Moduli_{0,n+1}};t)$ form the $f$-vector of a simplicial complex.
\end{corollary}

\begin{proof}
Let us consider the total order $\vartriangleleft$ on the atoms of $\Pi_n$ defined by 
$$ |i < j| \vartriangleleft |p < q| \textrm{ if and only if } i < j, \textrm{ or } i = p \textrm{ and }j < q.$$
Here $|i < j|$ means the partition with unique non-singleton equivalence class $\{i,j\}$. Notice that $\vartriangleleft$ satisfies Condition \eqref{eq:condition-ss-order} with respect to the maximal chain of modular flats \eqref{eq:maximal-chain-partition}, which implies that $(\Pi_n, \G_{\min})$ is $\vartriangleleft$-complete (see Example \ref{ex:supersolvable-built-matroid}). By Theorem \ref{thm:complete-formula} we obtain the formula
\begin{equation}\label{eq:complete-formula-partition}
\gamma(P_{\Moduli_{0,n+1}};t) = \sum_{\S \in \N(\Pi_n, \G_{\min})_{n,\stable}^{\max}}t^{\des(\S)}.
\end{equation} 
The isomorphism between $T_n$ and $\N(\Pi_n, \G_{\min})$ described in the proof of \cite[Theorem 3.1]{feichtner} sends a maximal vertex $\T$ to the maximal nested set $\S(\T) = \{ \Leaf(v) \, | \, v \textrm{ vertex of } \T \}.$ One easily checks that a vertex $v$ of $\T$ is a descent of $\T$ if and only if $Leaf(v)$ is a descent of $\S(\T)$, which implies that $\S(-)$ induces a descent-preserving bijection between $T_{n,\stable}^{\max}$ and $\N(\Pi_n, \G_{\min})_{n,\stable}^{\max}$. Thus, from \eqref{eq:complete-formula-partition} we obtain the equality of polynomials 
\[ \gamma(P_{\Moduli_{0,n+1}};t) = \sum_{\T \in T_{n,\stable}^{\max}}t^{\des(\T)}.\]
The coefficients of this polynomial correspond to an $f$-vector thanks to Theorem~\ref{thm:gamma-is-simplicial-complex}.
\end{proof}

\section{Application II: matroids of low rank}

\subsection{Essentials on $f$-vectors of complexes}\label{subsec:f-vector-inequalities}

Let $\Delta$ be a finite simplicial complex of dimension $d-1$. For $i\in \{0,1,\ldots,d\}$, let $f_i(\Delta)$ denote the number of simplices of cardinality $i$ (equivalently, of dimension $i-1$). The \emph{$f$-vector} of $\Delta$ is
\[
(f_0(\Delta),f_1(\Delta),\ldots,f_d(\Delta)).
\]
The \emph{$h$-vector} $(h_0(\Delta),\ldots,h_d(\Delta))$ is defined by
\[
  \sum_{i=0}^d f_i(\Delta)(y-1)^{d-i}
  \;=\;
  \sum_{i=0}^d h_i(\Delta)y^{d-i}.
\]
Accordingly,
\[
  f(\Delta;x)\coloneqq \sum_{i=0}^d f_i\,x^i,
  \qquad
  h(\Delta;y)\coloneqq \sum_{i=0}^d h_i\,y^i.
\]

A sequence $(f_0,\ldots,f_d)$ is an \emph{$f$-vector}, or \emph{Kruskal--Katona vector}, if it is the $f$-vector of a simplicial complex. Likewise, a sequence $(h_0,\ldots,h_d)$ is an \emph{$M$-vector}, or \emph{$O$-sequence}, if there exists a multicomplex $\Xi$ whose degree-$i$ monomials are counted by $h_i$ for all $0\leqslant i\leqslant d$. Since simplicial complexes are special cases of multicomplexes, every $f$-vector is an $M$-vector, but not conversely. Both classes admit explicit characterizations by inequalities; see \cite[Chapter~II]{stanley-96}. If the underlying simplicial complex or multicomplex is pure, one speaks of a \emph{pure $f$-vector} or \emph{pure $M$-vector}, respectively.

The \emph{Frankl--F\"uredi--Kalai (FFK) inequalities} \cite[Theorem~1.2]{frankl-furedi-kalai} characterize the $f$-vectors of balanced complexes. For instance, a $1$-dimensional balanced complex has $f$-vector $(1,f_1,f_2)$ with
\[
4f_2\leqslant f_1^2,
\]
equivalently, the polynomial $1+f_1t+f_2t^2$ is real-rooted. This follows from the fact that a bipartite graph on $f_1$ vertices has at most $f_1^2/4$ edges.

There is no conjectural characterization of the $f$-vectors of flag simplicial complexes. Nevertheless, Fr\"ohmader \cite{frohmader} proved that every such $f$-vector is also the $f$-vector of a balanced complex, and hence satisfies the FFK inequalities.

By the $g$-theorem for simplicial polytopes, due to Billera--Lee and Stanley \cite{billera-lee,stanley-g-theorem}, one has the following characterization of the $h$-vectors of boundary complexes of simplicial polytopes.

\begin{thm}[$g$-theorem for polytopes]\label{thm:g-theorem-for-polytopes}
    A sequence of integers $(h_0,\ldots,h_d)\in \mathbb{Z}^{d+1}$ is the $h$-vector of a simplicial polytope if and only if it satisfies:
    \begin{enumerate}[\normalfont (i), leftmargin = 20pt, itemsep = 3pt] 
        \item \label{it:g-thm(i)} $h_0 = 1$.
        \item \label{it:g-thm(ii)} $h_i = h_{d-i}$ for each $0\leqslant i\leqslant d$.
        \item \label{it:g-thm(iii)} The $g$-vector, defined as $(g_0,\ldots,g_{\lfloor d/2\rfloor})$ where $g_0 = h_0 = 1$, and $g_i := h_i - h_{i-1}$ for each $1\leqslant i\leqslant \lfloor d/2\rfloor$, is an $M$-vector.
    \end{enumerate}
\end{thm}

Sequences satisfying these conditions are often called \emph{Stanley--Iarrobino sequences}, or \emph{SI-sequences}. In particular, SI-sequences are unimodal. A result announced by Adiprasito and Papadakis--Petrotou \cite{adiprasito,papadakis-petrotou,adiprasito-papadakis-petrotou} states that SI-sequences are exactly the $h$-vectors of simplicial spheres.

\subsection{Two useful results on $f$-vectors}

The next result is essentially due to Nevo, Petersen, and Tenner \cite{nevo-petersen-tenner}.

\begin{proposition}\label{prop:gamma-f-vector-g-pure-f-vector}
    Let $(h_0,\ldots,h_d)$ be a sequence of integers satisfying \ref{it:g-thm(i)} and \ref{it:g-thm(ii)} in \Cref{thm:g-theorem-for-polytopes}.  Consider the $\gamma$-expansion of the palindromic polynomial $h(x) = \sum_{i=0}^d h_i\, x^i$
    \[
    h(x) = \sum_{i=0}^{\lfloor d/2\rfloor} \gamma_i\,x^i(1+x)^{d-2i}.
    \]
    If $(\gamma_0,\ldots,\gamma_{\lfloor d/2\rfloor})$ is an $f$-vector, then the $g$-vector $(g_0,\ldots,g_{\lfloor d/2\rfloor})$ is a pure $f$-vector (and hence a pure $M$-vector as well). In particular, $(h_0,\ldots,h_d)$ is a SI-sequence. 
\end{proposition}

\begin{proof}
    A version of the above result (without the conclusion that the $g$-vector is a \emph{pure} $f$-vector) was established in work of Nevo, Petersen, and Tenner \cite[Proposition~6.4]{nevo-petersen-tenner}. However, the complex $\Delta$ constructed by these authors is in fact pure: to see this, let us use the same notation as in their proof. Fix any facet $F\sqcup G$ of their complex $\Delta$. If $G$ were not a facet of $\mathcal{B}(d-2|F|)$, by taking any bigger face $G'$ such that $G\subsetneq G'\in \mathcal{B}(d-2|F|)$, one can readily see that $F\sqcup G'$ is a face of $\Delta$ (here it is crucial that the vertex set of the ballot complexes is set to be disjoint from the vertex set of $\mathcal{F}(\gamma)$). In particular, if $F\sqcup G$ is a facet of $\Delta$, then $G$ is a facet of $\mathcal{B}(d-2|F|)$, and given that the ballot complexes are pure themselves (cf. \cite[Proposition~6.3]{nevo-petersen-tenner}), one gets: 
    \[
    |F\sqcup G| = |F| + |G| = |F| + \left\lfloor\tfrac{d-2|F|}{2}\right\rfloor = \left\lfloor d/2\right\rfloor
    \] which gives the desired purity.
\end{proof}

As a consequence of the preceding result, we obtain the following corollary in the cases $d\leqslant 5$.  
A slightly weaker form of the first implication in the next statement can be found in \cite[Proposition~6.1]{schweitzer-vecchi}.

\begin{corollary}\label{coro:small-rank}
    Preserve the notation of Proposition~\ref{prop:gamma-f-vector-g-pure-f-vector}. Let $d\in \{4,5\}$. If $(\gamma_0,\gamma_1,\gamma_2)\in \mathbb{Z}^3$ is an $f$-vector, then $(h_0,\ldots,h_d)$ is log-concave. If $(\gamma_0,\gamma_1,\gamma_2)$ is the $f$-vector of a balanced complex, then the polynomial $h(t)$ is real-rooted.    
\end{corollary}

\begin{proof}
    The log-concavity part follows from the fact that the $g$-vector $(g_0,g_1,g_2)$ is a pure $f$-vector, and $f$-vectors of pure $1$-dimensional complexes are trivially log-concave, because $g_2 \leqslant \binom{g_1}{2} \leqslant g_1^2$. The log-concavity of $h(t)$ follows by multiplying the log-concave polynomials $1+g_1t+g_2t^2$ and $1+t+t^2$. 
    
    The real-rootedness part follows from the already mentioned fact that a $1$-dimensional balanced complex corresponds to a bipartite graph, and for bipartite graphs with $\gamma_1$ vertices, the maximum possible number of edges is $\lfloor\frac{\gamma_1}{2}\rfloor^2$. This gives $\gamma_2 \leqslant \gamma_1^2/4$, which is equivalent to saying that $1+\gamma_1t+\gamma_2t^2$ is real-rooted. This, in turn, is equivalent to assert that $h(t)$ is real-rooted.
\end{proof}

\begin{remark}
    It is tempting to believe that similar statements will hold for $d\in \{6,7\}$, perhaps by considering $(\gamma_0,\gamma_1,\gamma_2,\gamma_3)$ contained in the set of $f$-vectors of complexes of dimension $d$ that are balanced or pure. However, we have been able to find examples in which $(\gamma_0,\gamma_1,\gamma_2,\gamma_3)$ is the $f$-vector of a balanced complex (resp. a pure complex) and $h(t)$ even fails to be log-concave. To see the existence of such balanced (resp. pure) complexes it suffices to rely on the FFK inequalities (resp. on the characterization of pure $f$-vectors of length $4$ found by Colbourn et al. \cite{colbourn}). 
\end{remark}

\subsection{Beyond the real-rootedness of Chow polynomials}
\label{subsec:beyond-rr}
Since the Frankl--F\"uredi--Kalai inequalities imply that the $f$-vector of a $1$-dimensional balanced complex $\Delta$ forms the coefficients of a real-rooted polynomial, we can prove the following result.

\begin{corollary}\label{cor:rk6-RR}
    Let $\M$ be a matroid of rank $r\leqslant 6$. The Chow polynomial $\H(\M,\G_{\max})(t)$ is real-rooted.
\end{corollary}

\begin{proof}
    Since $\M$ has rank $r\leqslant 6$, the Chow polynomial $\H(\M,\G_{\max})(t)$ has degree $r-1 \leqslant 5$. In particular, the associated $\gamma$-polynomial has degree at most $\lfloor \frac{r-1}{2}\rfloor = 2$. Since Theorem~\ref{thm:gamma-maximal-is-balanced} implies that this $\gamma$-vector is the $f$-vector of a balanced simplicial complex, we can conclude that $\H(\M,\G^{\max})(t)$ is real-rooted thanks to Corollary~\ref{coro:small-rank}.
\end{proof}
As mentioned in the introduction, the same result but for matroids of rank at most $5$ was previously deduced in \cite[Theorem~5.10]{ferroni-matherne-stevens-vecchi}. To achieve that proof, the authors relied on the Koszulness of the Chow ring $\uCH(\M,\G_{\max})$, which was proved by Mastroeni and McCullough in \cite{mastroeni-mccullough}---the technical ingredient to deduce real-rootedness are determinantal inequalities proved by Polishchuk and Positselski in \cite[Chapter~7, Theorem~2.1]{polishchuk-positselski}. 

The same strategy employed in the proof of the last corollary also allows us to deduce (via Theorem~\ref{thm:gamma-is-simplicial-complex} and Corollary~\ref{coro:small-rank}) the following log-concavity result.

\begin{corollary}\label{cor:rk6-complete-logconcave}
    Let $(\M,\G)$ be a complete built matroid, and assume that $\M$ has rank $r\leqslant 6$. The Chow polynomial $\H(\M,\G)(t)$ is log-concave.
\end{corollary}

Recently Schweitzer and Vecchi \cite[Theorem~6.2]{schweitzer-vecchi} proved the above result in the special case $\G=\G_{\max}$ (they work with Chow polynomials in the sense of \cite{ferroni-matherne-vecchi}, so their results are formulated for weakly ranked posets). We can in fact extend \cite[Theorems~1.1 and 1.3]{schweitzer-vecchi} to all complete building sets.

\begin{corollary}\label{cor:g-complete-is-pure-f-vector}
    Let $(\M,\G)$ be a complete built matroid, and assume that $\M$ has rank $r$. If $h_i$ is the $i$-th coefficient of $\H(\M,\G)(t)$ for each $i$, then the vector $(h_0, h_1-h_0,\ldots, h_{\lfloor (r-1)/2\rfloor } - h_{\lfloor (r-1)/2\rfloor-1})$ is a pure $f$-vector.
\end{corollary}

\begin{proof}
    This follows directly from Theorem~\ref{thm:gamma-is-simplicial-complex} together with Proposition~\ref{prop:gamma-f-vector-g-pure-f-vector}.
\end{proof}

\section{Open questions}

We finish this article by discussing several open problems and directions we would regard as interesting for further research.

\subsection{Koszulity of Chow rings and \texorpdfstring{$\gamma$}{gamma}-positivity}

A result of Mastroeni and McCullough~\cite{mastroeni-mccullough} shows that the Chow ring of a matroid with its maximal building set is Koszul. They also prove that the augmented Chow ring of a matroid is Koszul. On the other hand, Coron~\cite{Coron_2025} proved that the Chow ring of a supersolvable built lattice is Koszul; the special case of braid matroids with minimal building set had previously been established by Dotsenko~\cite{dotsenko}. We conjecture the following generalization. 

\begin{conjecture}\label{conj:complete-implies-koszul}
    If $(\M,\G)$ is a complete built matroid, then the Chow ring of $(\M,\G)$ is Koszul.
\end{conjecture}

\begin{conjecture}\label{conj:flag-implies-koszul}
    If $(\M,\G)$ is a flag built matroid, then the Chow ring of $(\M,\G)$ is Koszul.\footnote{This question was also raised independently by Spencer Backman.}
\end{conjecture}

Note that all of the cases above fall within the class of complete built matroids, but the augmented Chow ring case does not correspond to a flag built lattice.

As Reiner and Welker explain in \cite[Section~4]{reiner-welker}, there is a mysterious connection among Koszulity, $\gamma$-positivity, and flagness. Jason McCullough brought to our attention the existence of Koszul Gorenstein Artinian algebras whose Hilbert--Poincar\'e series are not $\gamma$-positive in the work of D'Al\`i and Venturello \cite[Example~7.5]{dali-venturello}. 
\begin{question}\label{question:koszul-gorenstein-implies-gamma}
    Let $R$ be a standard graded Koszul, Artinian, Gorenstein ring. What extra conditions on $R$ ensure the $\gamma$-positivity of the Hilbert--Poincar\'e series of $R$? 
\end{question}

Both the Chow ring of a built lattice and the cohomology of the toric variety associated with a simplicial polytope satisfy all the properties in the above question; in addition, they satisfy the K\"ahler package. In particular, it would be very interesting to know whether there exist Koszul Artinian algebras satisfying the K\"ahler package whose Hilbert--Poincar\'e series are not $\gamma$-positive. A related question was posed by Mastroeni and McCullough in \cite[Question~6.8]{mastroeni-mccullough}. If Koszulity together with the K\"ahler package were sufficient for $\gamma$-positivity, this would imply Gal's conjecture for flag simplicial polytopes and recover recent results on the $\gamma$-positivity of Chow rings of built matroids in the cases where Koszulity is known; see \cite{ferroni-matherne-stevens-vecchi, aluffi-chen-marcolli}. Consequently, a proof of Conjectures~\ref{conj:complete-implies-koszul} and \ref{conj:flag-implies-koszul} would recover several of the main results of the present article.

\subsection{Complete and flag building sets on a given matroid}

Consider a matroid $\M$ with lattice of flats $\L$. The set $\mathbb{B}(\L)$ of all building sets on $\L$ ordered by containment was proved by Backman and Danner \cite[Corollary~3.14]{backman2025convexgeometrybuildingsets} to be a lattice. Let us denote $\mathbb{B}_{\text{flag}}(\L)$ (resp. $\mathbb{B}_{\text{comp}}(\L)$) the family of all flag (resp. complete) building sets on $\L$.

\begin{question}
    How do the posets $(\mathbb{B}_{\text{flag}}(\L),\subset)$  and $(\mathbb{B}_{\text{comp}}(\L),\subset)$ embed into $(\mathbb{B}(\L),\subset)$? Is any of these two subposets a sublattice?
\end{question}

Furthermore, Backman and Danner prove that $\mathbb{B}(\L)$ is supermodular and it yields a supersolvable convex geometry structure on $\L\setminus\{\widehat{0}\}$ (see \cite[Proposition~3.13]{backman2025convexgeometrybuildingsets}). It would be interesting to know if some of their results also apply to the subposets corresponding to flag and complete building sets. Progress in this direction would provide additional structural information on these classes of building sets. 

\subsection{Real-rootedness problems}

In view of Conjectures~\ref{conj:chow-rr-maximal} and \ref{conj:chow-rr-braid-min}, it is natural to ask whether the gamma-positivity results of this paper can be strengthened to real-rootedness. The following example shows that this is not the case.\footnote{In an early version of the present paper (arXiv v2), we left this as an open question. The same question was also raised independently by Spencer Backman.} The example below was obtained with the assistance of ChatGPT~5.5 Pro. 

\begin{example}[Non-real-rooted Chow polynomials of flag and complete built matroids]
\label{ex:non-RR-chordal-flag-nestohedra}
Let $\B_{7}$ denote the Boolean lattice over seven elements $1,\ldots, 7,$ and let $\G$ denote the building set 
\begin{multline*} 
\G \coloneqq \{1,\ldots, 7, 12,13,123,134,1345,1235,1234,\\
13456,12345,134567,123457,123456,1234567 \}.
\end{multline*}
One can check that $\G$ is a building set of $\B_7$, and the built lattice $(\B_7, \G)$ is complete with respect to the natural order $1 < \cdots <7$. In particular, since $\B_7$ is supersolvable, the built lattice $(\B_7, \G)$ is flag (see Example \ref{ex:flag-built-lattices} (iv)). One can compute
\[ \H(\B_7, \G)(t) = 1 + 13t + 48t^2 + 73t^3 + 48t^4 + 13t^5 + t^6.\]
This polynomial has complex roots (approximated to $3$ decimal places): 
\begin{align*}
-1.851-0.137i, 
-1.851+0.137i, 
-0.537-0.039i, 
-0.537+0.039i. 
\end{align*}
More generally, for any $k \geqslant 7$, if one considers the building set 
\[ 
\G_k \coloneqq \G \cup \{8, \ldots, k , 12345678, 123456789, \ldots, 1\cdots k\}
\]
of $\B_k$, then $(\B_k, \G_k)$ is also a complete built lattice, and the corresponding Chow polynomial is 
\[\H(\B_k, \G_k)(t) = \H(\B_7, \G)(t)(1+t)^{k-7},\]
which is not real-rooted either. 
\end{example}

The above example implies the following result.

\begin{thm}\label{thm:chordal-not-RR}
    For every dimension at least $6$, there exist flag and chordal nestohedra with a non-real-rooted $h$-polynomial.
\end{thm}

A sensible question is how the real-rootedness of Chow polynomials relates to the Kruskal--Katona and FFK inequalities for the corresponding $\gamma$-vector. It is well known that, for a palindromic polynomial $h(t)$ with nonnegative coefficients, real-rootedness of $h(t)$ is equivalent to real-rootedness of its associated $\gamma$-polynomial. About two decades ago, Bell and Skandera \cite[Conjecture~4.1]{bell-skandera} conjectured that every real-rooted polynomial with positive integer coefficients and constant term equal to $1$ is the $f$-polynomial of a simplicial complex. They further asked in \cite[Question~5.1]{bell-skandera} whether every such polynomial is in fact the $f$-polynomial of a balanced complex. Both the conjecture and the question remain widely open, although some progress has been made in recent years; see \cite{mu-welker}. In particular, if the question by Bell and Skandera admits an affirmative answer, then one could view our Theorem~\ref{thm:main-complete-implies-nevo-petersen+balanced} as supporting evidence towards Conjecture~\ref{conj:chow-rr-maximal}.

\subsection{\texorpdfstring{$g$}{g}-vectors of nested set complexes}

Another polynomial naturally associated with a built matroid is the $h$-polynomial of its nested set complex. In \cite{coron-ferroni-li}, we showed that, in the case of the maximal building set, the study of the real-rootedness of this polynomial can be run in parallel with that of the Chow polynomial. Moreover, in \cite[Conjecture~6.1]{coron2026structuralpropertiesnestedset}, we conjectured that the nested set complex of any flag built matroid has real-rooted $h$-polynomial. On the other hand, this phenomenon already fails for complete built matroids, as shown by the examples in \cite[Example~4.2]{athanasiadis-ferroni}.

Recall that if $\Delta$ is a simplicial complex of dimension $d-1$, then its $g$-vector is
\[
(g_0(\Delta),\ldots,g_{\lfloor d/2\rfloor}(\Delta)),
\]
where $g_0(\Delta)=h_0(\Delta)$ and
\[
g_i(\Delta)=h_i(\Delta)-h_{i-1}(\Delta)
\qquad \text{for } i=1,\ldots,\lfloor d/2\rfloor.
\]

\begin{question}
    Consider the $g$-vector of the nested set complex $\N(\L,\G)$. Is this $g$-vector an $f$-vector
    \begin{enumerate}[\normalfont(i)]
        \item when $\G$ is a complete building set?
        \item when $\G$ is a flag building set?
    \end{enumerate}
\end{question}

In \cite[Corollary~1.4]{coron2026structuralpropertiesnestedset}, using a result of Swartz \cite[Corollary~3.2]{swartz}, we proved that the $g$-vector of $\N(\L,\G)$ is always an $M$-vector, without any assumption on the building set. On the other hand, even in the case of the maximal building set, it is not known whether the $g$-vector of $\N(\L,\G)$ (equivalently, of the order complex of $\L$) is necessarily an $f$-vector. This question is particularly relevant in light of recent conjectures of Athanasiadis and Kalampogia--Evangelinou \cite{athanasiadis-kalampogia} concerning $h$-vectors of order complexes of geometric lattices.

\subsection{A flag complex for flag built lattices}

In their paper, Nevo and Petersen conjectured that the $\gamma$-vector of a flag simplicial sphere is the $f$-vector of a balanced simplicial complex. They also asked whether one can strengthen this statement by requiring the complex to be flag; see \cite[Problem~6.4]{nevo-petersen}. This is a strictly stronger requirement, since being the $f$-vector of a flag complex is more restrictive than being the $f$-vector of a balanced complex; see \cite{frohmader} (we note that being a flag complex does not imply being a balanced complex).

\begin{question}
    Let $(\M,\G)$ be a flag built lattice. Is the $\gamma$-vector of $\H(\M,\G)(t)$ the $f$-vector of a flag simplicial complex?
\end{question}

When $\M$ is a Boolean matroid, that is, in the setting of flag nestohedra, the answer to this question is affirmative by a construction of Aisbett \cite{aisbett}. In the present manuscript, however, we prove only the nonnegativity of the $\gamma$-vector. A natural intermediate step would therefore be to show that these $\gamma$-vectors satisfy the Kruskal--Katona inequalities.

\bibliographystyle{amsalpha}
\bibliography{biblio}
\end{document}